\documentclass{amsart}
\usepackage{amssymb,latexsym,graphics,enumerate,color}
\usepackage{fullpage}
\numberwithin{equation}{section}

\usepackage[normalem]{ulem}

\usepackage{xcolor,cancel}

\newtheorem{thm}{Theorem}[section]

\newtheorem{lem}[thm]{Lemma}
\newtheorem{prop}[thm]{Proposition}

\newtheorem{definition}[thm]{Definition}

\theoremstyle{remark}
\newtheorem{rem}[thm]{Remark}

\theoremstyle{remark}

\theoremstyle{definition}

\newcommand{\cal}{\mathcal}

\newcommand{\lp}[2]{\Vert \, #1 \, \Vert_{#2}}

\newcommand{\llp}[2]{ {|\!|\!| \, #1 \, |\!|\!|}_{#2} }

\newcommand{\td}{\widetilde}

\newcommand{\LE}{ {\mathcal{LE}} }
\newcommand{\U}{U}

\newcommand{\la}{\langle}
\newcommand{\ra}{\rangle}

\renewcommand{\O}{{\mathcal{O}}}

\newcommand{\R}{{\mathbb{R}}}

\begin{document}
\bibliographystyle{plain}

\title[Local Energy Decay on Non-trapping Backgrounds]
{Local Energy Decay for Scalar Fields on Time Dependent Non-trapping
  Backgrounds}
\author{Jason Metcalfe}
\author{Jacob Sterbenz}
\author{Daniel Tataru}
\thanks{}
\date{}
\dedicatory{}
\commby{}

\thanks{The first author is supported in part by NSF grant
  DMS-1054289, and the third author was supported by NSF grant
  DMS-1266182 as well as by a Simons Investigator Award from the
  Simons Foundation.}


\begin{abstract}
  We consider local energy decay estimates for solutions to scalar
  wave equations on nontrapping asymptotically flat space-times.  Our
  goals are two-fold.   First we consider the stationary case, where we can provide a full
  spectral characterization of local energy decay bounds; this
  characterization simplifies in the stationary symmetric case.
Then we consider  the almost stationary, almost symmetric case. There we
  establish two main results: The first is a ``two point'' local
  energy decay estimate which is valid for a general class of
  (non-symmetric) almost stationary wave equations which satisfy a
  certain nonresonance property at zero frequency. The second result,
  which also requires the almost symmetry condition, is to establish
  an exponential trichotomy in the energy space via finite dimensional
  time dependent stable and unstable sub-spaces, with an infinite
  dimensional compliment on which solutions disperse via the usual
  local energy decay estimate.
\end{abstract}

\maketitle
\tableofcontents

\section{Introduction}

The aim of this paper is to contribute to the understanding of local energy decay
bounds for the wave equation on non-trapping, asymptotically flat space-times.
This is a well understood question for small perturbations of the Minkowski space-time.
Our aim here, instead, is to consider large perturbations. In that, our goals are two-fold.
\medskip

First we consider the {\em stationary} case. There we provide a full
spectral characterization of the local energy decay estimates in
terms of the eigenvalues and resonances of the corresponding elliptic
problem. There are three such objects which are of interest to us:
\begin{itemize}
\item Complex eigenvalues outside the continuous spectrum $\mathbb R$
and in the lower half-space,
\item Zero eigenvalues/resonances, and
\item Nonzero resonances embedded inside the continuous spectrum.
\end{itemize}
Our main result here asserts that local energy decay holds iff
none of these 
three obstructions occurs. 

One significant simplification occurs in the {\em symmetric} case,
where no nonzero resonances can occur inside the continuous
spectrum.  There our
results are consistent with the standard spectral theory for
self-adjoint elliptic operators.  In that case, we also consider the
problem of continuity of our spectral assumptions along one parameter
families of operators; the main idea being that complex eigenvalues 
can only emerge via the zero mode. If complex eigenvalues do occur, a slightly more  
complicated picture emerges, and the flow splits into two finite dimensional 
subspaces where exponential growth, respectively decay occurs, and a bulk 
part with uniform energy bounds.

\medskip

Secondly, we study the case of {\em time dependent} operators and show that 
the results in the stationary case are stable with respect to perturbations.
More precisely,  we consider  almost symmetric, almost stationary 
operators which  satisfy a quantitative  zero spectral assumption uniformly in time,
but allowing for eigenvalues off the real axis. Then we establish an exponential 
trichotomy, splitting the energy space as a direct sum of three subspaces as follows:
\begin{itemize}
\item A finite dimensional subspace of spatially localized,
  exponentially increasing solutions, associated to eigenvalues in the
  lower half-space
\item A finite dimensional subspace of spatially localized,
  exponentially decreasing solutions, associated to eigenvalues in the
  upper half-space
\item A remaining infinite dimensional subspace of bounded energy solutions
with good local energy bounds.
\end{itemize}
If there are no eigenvalues off the real axis then only the last subspace is nontrivial,
and local enegy decay holds globally.

One key intermediate step of our approach here is to establish a
weaker bound, which we call {\em two point local energy decay}, see \eqref{LE-2pt} below.
This has the advantage that it  allows for nonreal eigenvalues; however, it
 prohibits nonzero resonances in the symmetric case.

\medskip

Our motivation for this work is that the local energy decay estimates have 
emerged in recent years as the core decay bounds for asymptotically flat 
operators.  As examples we note results such as \cite{MT}, which assert
that local energy decay implies global Strichartz estimates or the results 
in \cite{T}, \cite{MTT} which show that sharp pointwise decay bounds 
can also be obtained as a consequence of local energy decay. 

Another motivation is that we seek to lay out the groundwork for a similar 
analysis in  black hole space-times, where trapping does occur. 
 A subsequent article will be  devoted to the black hole case.

\subsection{The geometric setup}

We consider Lorentzian manifolds $(\mathbb{R}^4,g)$ parametrized by
coordinates $(t,x)\in \mathbb{R}\times\mathbb{R}^3$. We assume that $g$ has 
signature $(3,1)$. We set $r =
|x|$.  On such manifolds we consider wave operators of the form
\begin{equation}\label{P-def}
P(t,x,D)=\Box_{A,g} +   V(t,x), \qquad \Box_{A,g} = (D_{\alpha} + A_{\alpha})g^{{\alpha\beta}} (D_{{\beta}}+A_{{\beta}})
\end{equation}
where $A$ and $V$ are complex valued, and $D_\alpha =
(1/i)\partial_\alpha$.  We assume the metric $g$ satisfies a weak
asymptotic flatness condition and is either stationary or
slowly varying in time. We also put similar conditions on the lower
order terms $A$ and $V$. To describe the class of metrics as well as
the asymptotic flatness assumption we use the class of spaces $l^p
L^q$ where the $l^p$ norm is taken with respect to nonnegative
integers indexing spatial dyadic regions, and the $L^q$ norm is taken
over dyadic regions. Then we define the class of norms for functions
$q(t,x)$:
\begin{equation}\label{triple-def}
	\llp{ q }{k} \ = \  \sum_{|\alpha| \leq k} 	\|\la x\ra^{|\alpha|} \partial^\alpha q \|_{l^1 L^\infty}. 
\end{equation}
We will measure the asymptotic flatness of the metric using the norm
\begin{equation}
		\llp{ (h,A,V) }{AF} \ =\  \llp{h}{2} + \llp{ \la x\ra A}{1} 
		+ \llp{\la x\ra^2 V}{0} \ . \notag
\end{equation}

We use this to introduce the class of asymptotically flat operators $P$ 
by comparing with the Minkowski metric $m$ as follows:

\begin{definition}
 We say that the operator $P$ is asymptotically flat (AF) if $\llp{ (g-m,A,V) }{AF}$ is finite.
\end{definition}

To fix the parameters our estimates and results will depend on,
we set $M_0, R_0$ so that
\begin{equation}\label{AF-def}
\llp{ (g-m,A,V) }{AF} \leq M_0, \qquad \llp{ (g-m,A,V) }{AF, > R_0} \leq { \bf c} \ll 1,
\end{equation}
where the subscript $>R_0$ above restricts the norm to the
cylinder $\{r > R_0\}$.
Here the implicit constant $\bf c$ is small but universal, in that it
does not depend on any of the other parameters in our problem.
As a quantitative way of measuring the decay of the $AF$ norm at infinity
we will use a frequency envelope $\{c_k\}_{k \geq \log_2 R_0}$ with the following properties:
\begin{itemize}
\item Dyadic bound: $\llp{ (g-m,A,V) }{AF(\{r \approx 2^k\})} \lesssim c_k$.
\item Summability: $\sum_k c_k \lesssim \bf c$.  
\item Slowly varying: $c_k/c_j \leq 2^{\delta |j-k|}$ for a small universal constant $\delta$.
\end{itemize}

\bigskip

The first property of AF operators $P$ which plays a role in our analysis
is stationarity:

\begin{definition}
a) We say that $P$ is stationary if $(g,A,V)$ are time independent. 

b) We say that $P$ is  $\epsilon$-slowly varying if there exists a decomposition
\[
(g-m,A,V) = (g_1,A_1,V_1)+ (g_2,A_2,V_2)
\]
so that
\[
\llp{ (g_1,A_1,V_1)}{AF} + \llp{\partial_t (g_2,A_2,V_2)}{AF} \leq \epsilon.
\] 
\end{definition}

\bigskip

The second global property of the operator $P$ that plays a role in our results
is its self-adjointness.  Precisely, given a measure $dV = \mu(x) dx$ we can
define the adjoint operator $P^*$ with respect to $L^2(dV)$ by
\[
\la Pu, v \ra_{\mu} =  \la u, P^* v \ra_{\mu}. 
\]
Here one could consider two standard scenarios:

(i) $\Box_g$ is in the Laplace-Beltrami form, and $dV = \sqrt{g} dx$.

(ii) $\Box_g$ is in divergence  form, and $dV =  dx$.

But the first case reduces to the second after conjugating by
$g^\frac14$, without affecting anything in our analysis. Thus from
here on we assume that we are in the latter case. We remark
that $P^*$ satisfies the same asymptotic flatness bounds as $P$.

\begin{definition}
a) We say that the operator $P$ is symmetric if $P^* = P$, i.e. $A$ and $V$ are real.

b) We say that $P$ is  $\epsilon$-almost symmetric if $\llp{(0,\Im A,\Im
  V)}{AF} \lesssim \epsilon$.
\end{definition}

\bigskip

Next, we state a quantitative form of nontrapping
for the space-time metric $g$, which
will play a major role in the present work:

\begin{definition}\label{NT_def}
We say that the asymptotically Lorentzian metric $g$  is nontrapping if 
\begin{enumerate}
		\item \label{zero_ell} The vector field $\partial_t$ is  uniformly time-like.
		\item \label{nt2} There exists $T_0 > 0$ such that all null geodesics spend a time no 
                longer than $T_0$ inside  $B(0,R_0)$.
\end{enumerate}
\end{definition}
If $g$ were stationary, then a compactness argument shows that 
the above quantitative condition follows from a similar qualitative condition.
However, the quantitative version of
nontrapping is needed for time dependent operators $P$.

We also remark that the first condition above is an ellipticity condition for the operator 
$P_0 = P |_{D_t=0}$. The asymptotic flatness conditions already guarantee that 
this holds outside of a compact set. If in addition the metric is $\epsilon$-slowly varying
then we can obtain this first condition everywhere from the second:

\begin{lem}
  Let $g$ be an $\epsilon$-slowly varying asymptotically flat metric
  which satisfies the condition \eqref{nt2} in
  Definition~\ref{NT_def}.  If $\epsilon$ is sufficiently small,
\[
\epsilon \ll_{R_0,T_0,M_0} 1,
\]
then the vector field $\partial_t$ is uniformly timelike.
\end{lem}

\begin{proof}
  If that were not the case, then there is a point where $(t_0,x_0)$ where
  $\la\partial_t,\partial_t \ra_g=0$. Then we can  use the $\epsilon$-approximate
  conservation of $\la\partial_t,\partial_t \ra_g$ along the null
  geodesic with initial direction $\partial_t$ at $(t_0,x_0)$ to
  derive a contradiction, because one must have
  $\la\partial_t,\partial_t \ra_g\approx -1$ at points along the null
  geodesic by the time $T_0$, when it reaches the exterior region $|x|> R_0$.
\end{proof}


\subsection{Local energy norms}
 The space-time local energy norms are defined as follows:
\begin{equation}
\lp{u}{LE} \ = \ \lp{u}{\ell^\infty L^{2,-\frac{1}{2}}}, \qquad
		\lp{u}{LE^1} \ = \ \lp{(\partial u,\la x\ra^{-1}u)}{LE}. \ 
\end{equation}
Here $\partial = (\partial_\alpha)_{\alpha=0,\dots, 3}$ and
 $\|w\|_{L^{q,\sigma}} = \|w\|_{L^q(\la x\ra^\sigma\,dx)}$.  We shall
  reserve the notation $\nabla$ for $\nabla =
  (\partial_j)_{j=1,2,3}$. 
For the source term we use the dual space
\begin{equation}
		\qquad \lp{F}{LE^*} \ = \ \lp{F}{\ell^1 L^{2,\frac{1}{2}}} \ . \notag
\end{equation}
For the corresponding analysis of the resolvent we also need the related 
norms at fixed time; these are denoted by $\LE$, $\LE^1$, $\LE^*$.
By $LE_0$, $\LE_0$, etc. we denote the closure of $C_0^\infty$ in these spaces;
this corresponds to replacing $\ell^\infty$ by $c^0$ in the definition
above.  For later purposes when time derivatives will be replaced
  by a spectral parameter $\omega$, we also define
\[\LE^1_\omega = \LE^1\cap |\omega|^{-1}\LE,\quad \dot{H}^1_\omega =
\dot{H}^1 \cap |\omega|^{-1}L^2.\]

Our main goal is to understand which space-times $(\mathbb R^4,g)$ and operators $P$
 satisfy {\em local energy decay} estimates:
\begin{definition}
 We say that local energy decay holds for a nontrapping asymptotically flat operator $P$ 
if  the following estimate holds:
\begin{equation}
		\lp{u}{LE^1[0,T]} +\lp{\partial u}{L^\infty L^2[0,T]}   \ 
\lesssim \ \lp{\partial u(0)}{L^2} 
		 + \lp{P u}{LE^*+ L^1L^2[0,T]
                   } . \label{LE}
\end{equation}
\end{definition}

Estimates of this type have a long history, and were first introduced
in the work of Morawetz~\cite{morawetz1, morawetz2, morawetz3} for the Minkowski space-time;
for further references we refer the reader to \cite{Alinhac},
\cite{bpstz}, \cite{hy},
\cite{KSS}, \cite{kpv}, \cite{MetSo, MetSo2}, \cite{mrs},
\cite{smithsogge}, \cite{sterb}, \cite{Strauss}.  As starting point for this work we
use the result in \cite{MT}, which conveniently uses the same
definition of asymptotic flatness as here.

\begin{thm}[\cite{MT}]\label{t:small}
Assume that $P$ is a small AF perturbation of $\Box$. Then local energy decay holds
for $P$.
\end{thm}

Implicit in the above estimate is a global-in-time uniform energy
bound.  A qualitatively weaker version of the local energy decay is
what we call the {\em two point local energy decay} estimate. This allows
for spatially localized exponentially growing or decaying solutions:

\begin{definition}
  We say that two point local energy decay holds for a nontrapping
  asymptotically flat operator $P$ if the following estimate holds:
\begin{equation}
		\lp{u}{LE^1[0,T]} +\lp{\partial u}{L^\infty L^2[0,T]}    \ 
\lesssim \ \lp{\partial u(0)}{L^2}+  \lp{\partial u(T)}{L^2}
		 + \lp{P u}{LE^*+ L^1L^2[0,T] } . \label{LE-2pt}
\end{equation}
\end{definition}

In general this estimate also does not exclude nonzero
  resonances, discussed below, except in the special case of symmetric
  operators $P$. However, it does preclude stationary or nearly
  stationary spatially localized solutions.  An even weaker estimate
whose role is to only preclude stationary solutions is a
condition introduced in \cite{MTT}, called {\em stationary local
  energy decay}:

\begin{definition}
  We say that stationary  local energy decay holds for a nontrapping
  asymptotically flat operator $P$ if the following estimate holds:
\begin{equation}\label{LE-stat}
		\lp{u}{LE^1[0,T]}  +\lp{\partial u}{L^\infty L^2[0,T]} \ 
\lesssim \ \lp{\partial u(0)}{L^2}+  \lp{\partial u(T)}{L^2}+	\lp{\partial_0 u}{LE[0,T]}
		 + \lp{P u}{LE^*+ L^1L^2[0,T]}. \ 
\end{equation}
\end{definition}

\subsection{Energy estimates}

As defined in the previous section, uniform energy bounds of the form
\begin{equation}\label{en}
\| \partial u\|_{L^\infty L^2} \lesssim \| \partial u(0)\|_{L^2} + \|Pu\|_{L^1 L^2}
\end{equation}
are included as a subset of the local energy decay estimates. Further, they provide a direct 
connection between the more robust two point local energy decay bound \eqref{LE-2pt} 
and the local energy bound \eqref{LE}.

For this and other reasons, it is useful to briefly discuss the energy conservation properties
for the wave equation $Pu = f$. This is most easily done in the
symmetric stationary case, where we have a conserved energy functional for the homogeneous problem,
namely
\begin{equation}\label{E}
E(u) = \int_{\R^3}  \bar u P_0 u  - g^{00} |\partial_t u|^2 \,dx, \qquad P_0 = P_{|D_t = 0}.
\end{equation}
For the inhomogeneous problem $Pu = f$, we have the relation
\begin{equation}\label{E-inhom}
\frac{d}{dt} E(u) = 2 \Re \int_{\R^{3}} \bar f \partial_t u \,dx.
\end{equation}
The difficulty in using this energy is that it is not necessarily
positive definite. However, in view of condition \eqref{zero_ell}
in Definition~\ref{NT_def}, the operator $P_0$ is elliptic and
asymptotically flat, so the obstruction to that is localized in a
compact set, i.e.
\begin{equation}\label{e-pos}
\| (u,\partial_t u)\|_{\dot H^1 \times L^2}^2 \lesssim  E(u) + \|
u\|_{L^2_{comp}}^2 .
 \end{equation}
For a further discussion of this case we refer the reader to the last
section of the paper.  

If we drop the stationarity assumption, then the energy functional
defined above becomes time dependent and is no longer
conserved. If we also drop the symmetry assumption then it is natural
to work with the (time dependent) energy associated to the symmetric
part of $P$. This leads us to a straightforward almost
conservation property

\begin{lem} \label{l:en} Assume that $P$ is asymptotically flat,
  $\epsilon$-slowly varying and $\epsilon$-almost symmetric. Then we have 
the energy conservation relation
\begin{equation}
\label{eps-en}
E(u(T)) = E(u(0)) + O(\epsilon) \|u\|_{LE^1[0,T]}^2
  \end{equation}
for solutions to the homogeneous equation $Pu=0$.
\end{lem}
The proof of this is straightforward and is left for the reader.

\section{Main results}

\subsection{ The stationary case.} 
Here we discuss the easier case of operators $P$ which are stationary,
where we can provide a full spectral description for our local energy
decay problem.  Later, we will use this case as a guide for what to
expect in the more general case. 

The results here are more in line with earlier work on the subject,
at least in the self-adjoint case, see for instance \cite{ggh}, \cite{g}, \cite{T}. However,
for clarity and further reference it is perhaps of interest  to provide a full, 
self-contained set of results.

To set the notations and introduce the corresponding spectral problems,
 we consider  special solutions to the homogeneous wave equation $P(x,D) u = 0$
that have the form $u=e^{i\omega t}u_\omega(x)$.
Then we are led to the quadratic
eigenvalue problem:
\begin{equation}\label{omega-mode}
 P_\omega u_\omega  = 0,
\end{equation}\label{p_omega}
where $ P_\omega $ is the operator
\begin{equation}
P_\omega= 	 \Delta_{g,A}  + W(x,D_{x}) + \omega B(x,D_x) + \omega^2 g^{00} 
\end{equation}
with
\begin{align}
		\Delta_{g,A}  \ &= \   (D+A)_j  g^{jk}(D+A)_k \ , \notag\\
		W(x,D_x) \ &= \ 
		A_0g^{0k}(D_k+A_k) + (D_k+A_k) A_0 g^{0k}+ g^{00}A_0^2+ V(x) \ , \notag\\
		B(x,D_x)\ &=\ g^{0k}(D_k+A_k) + (D_k+A_k) g^{0k}   +2A_0g^{00}\ . \notag
\end{align}
We remark that if $P$ is self-adjoint and  $\omega$ is real, then $P_\omega$ is
self-adjoint. However, the above eigenvalue problem is not self-adjoint
so there is no reason for $\omega$ to be real, and in fact if $\omega
\not\in \mathbb R$ then one can expect there to be nontrivial root
space structure associated to $u_\omega$.

The resolvent $R_\omega$ is defined as the inverse of $P_\omega$ whenever the 
inverse exists. It can also be defined as the Fourier-Laplace transform of the 
wave group as follows. Suppose that $u$ solves the homogeneous wave equation
\[
Pu = 0 , \qquad u(0) = 0, \qquad g^{00} u_t(0) = f.
\]
Then the resolvent $R_\omega$ is given by 
\begin{equation}\label{fourier-laplace}
R_\omega f = \int_{0}^\infty e^{-it \omega} u(t) dt.
\end{equation}

A priori the local theory for the wave equation yields an exponential bound of the 
form
\begin{equation}\label{local-en}
\| \partial u(t)\|_{L^2} \lesssim e^{\beta t} \|f\|_{L^2}.
\end{equation}
This implies that the resolvent $R_\omega$ is defined at least in the 
halfspace $ \{ \Im \omega < - \beta \}$, with an uniform bound
\begin{equation}\label{local-res}
\| R_\omega \|_{L^2 \to \dot H^1_{\omega}} \lesssim | \Im \omega+ \beta|^{-1}, \qquad   \Im \omega < - \beta.
\end{equation}
Further, one can invert the Fourier transform and express the wave evolution 
in terms of the resolvent as
\begin{equation}\label{fourier-laplace-inverse}
u(t) =  \int_{\Im \omega = - \sigma}  
e^{i \omega t} R_\omega f d\omega , \qquad \sigma > \beta.
\end{equation}

The local energy bounds for the forward evolution associated to the operator $P$ are 
closely related to uniform energy bounds for the forward evolution. As in the above 
discussion, this corresponds to an extension of the resolvent into the lower half space.
In relation to this, we have the following standard result:

\begin{prop} \label{p:e-res}
  a) The resolvent $R_\omega: L^2 \to \dot H^1_{\omega}$ admits a meromorphic
  extension to the lower half-plane $H = \{ \Im \omega < 0 \}$. 

b) If uniform energy bounds hold for the forward evolution associated to $P$
then the resolvent is holomorphic in $H$ and satisfies 
\begin{equation}\label{enres}
\| R_\omega \|_{L^2 \to \dot H^1_{\omega}} \lesssim | \Im \omega|^{-1}, \qquad   \omega \in H.
\end{equation}

c) Conversely, if the above bound holds then the solutions $u$ to the
homogeneous equation $Pu = 0$ satisfy the energy bound
\begin{equation}\label{S=t}
\| \partial u(t)\|_{L^2} \lesssim (1+t) \| \partial u(0)\|_{L^2} .
\end{equation}
\end{prop} 

The proof of this theorem is fairly standard and the details are left for the
reader. For part (a) it suffices to observe that the operators
$P_\omega$ are uniformly elliptic for $\omega$ in a compact subset of
$H$.  Part (b) follows by direct integration in
\eqref{fourier-laplace} and represents one direction (the easy one)
in the Hille-Yosida theorem.  

Part (c) is a consequence of the
averaged $L^2$ bound
\[
\| e^{-\sigma t} \partial u\|_{L^2_{x,t}} \lesssim \sigma^{-1} \|   e^{-\sigma t} Pu \|_{L^2_{x,t}},
\qquad  \sigma > 0,
\]
which follows by Plancherel from \eqref{enres}. For the forward homogeneous
initial value problem $Pu = 0$ this gives by truncation
\[
\| e^{-\sigma t} \partial u\|_{L^2_{x,t}([0,\infty) \times \R^3)} \lesssim \sigma^{-1} \|\partial u(0)\|_{L^2}
\qquad  \sigma > 0.
\]
One then chooses $\sigma = t^{-1}$, and the energy bound at time $t$ 
follows from the averaged bound on the time interval $[t-1,t]$.
Most likely one should
be able to improve the polynomial growth rate to $(1+t)^\frac12$; it
is not clear to us what the correct threshold is.

The poles of the resolvent $R_{\omega}$ in the lower half-space $H$
are zero eigenfunctions for the corresponding elliptic operators
$P_\omega$. By Fredholm theory the associated generalized eigenspaces
are finite dimensional.  Further, also due to the ellipticity of
$P_{\omega}$, the associated eigenfunctions and generalized
eigenfunctions are spatially localized and decay exponentially at
infinity. These correspond to exponentially growing solutions to the
homogeneous wave equation $Pu = 0$. Thus we have identified a first
obstruction to both uniform energy bounds and local energy decay:
\begin{definition}
A negative eigenfunction for $P$ is a function $u_\omega \in L^2$ so that 
$P_\omega u_\omega = 0$ with $ \omega  \in H$. 
\end{definition}

We also observe that one can similarly define positive eigenfunctions and associate
them to exponentially decreasing solutions to the wave equation. This can also be viewed
as the identical construction but with reversed time.

We now consider the interpretation of the local energy bounds in terms
of corresponding resolvent bounds. For this we have the following result:
\begin{prop}\label{p:le-res}
Local energy decay holds for the operator $P$
if and only if the resolvent $R_\omega$ satisfies a uniform bound
\begin{equation}\label{LE-fourier}
\| R_\omega \|_{\LE^* \to \LE^1_{\omega}} \lesssim 1, \qquad \Im \omega < 0 .
\end{equation}
\end{prop}

This is a fairly straightforward result, which is more or less a consequence of the 
Plancherel formula. The proof is in Section~\ref{s:stat}. See also \cite{T}.

This result leads us to a second obstruction to local energy decay,
which is obtained by taking the limit of the resolvent as $\omega$
approaches the real axis. There are two cases to consider, depending 
whether the limit is zero or not.

  In the nonzero case, in the limit we obtain a boundary condition
at infinity, called the outgoing radiation condition or the
Sommerfeld condition. This takes the form
\begin{equation}\label{outgoing}
2^{-j/2} \|(\partial_r + i \omega) u\|_{L^2(|x|\approx 2^j)} \to 0.
\end{equation}
Thus we define
\begin{definition}
  $\omega \in \mathbb R\setminus\{ 0\}$ is an embedded resonance for
  $P$ if there exists a function $u_\omega \in \LE^1_{\omega}$, satisfying the
  outgoing radiation condition \eqref{outgoing}, so that $P_\omega
  u_\omega = 0$.
\end{definition}

We carefully exclude the case $\omega = 0$ here.
It is not too difficult to show that in polar coordinates $(r,\theta)$ the 
embedded resonances must have the asymptotics
\[
u_{\omega}(r,\theta) = v(\theta) r^{-1} e^{i\omega r} + o(r^{-1}), \qquad v \neq 0
\]
where the condition $v \neq 0$ is a consequence of unique continuation
results, see \cite{KT}.  We observe that solutions of the form
$e^{i\omega t} u_\omega$ cannot be directly used as obstructions to
local energy decay or energy estimates, as they do not have finite
energy. However, at least for nice enough metrics,
one can consider truncated versions of them, of the
form
\[
\tilde u_{\omega}   = \chi(t -r) e^{i\omega t} u_\omega  
\]
with a smooth cutoff function $\chi$ which selects the positive real line and 
using suitable Regge-Wheeler type coordinates.
Then a direct computation yields the asymptotic relations
\[
\| \partial \tilde u_\omega(t)\|_{L^2} \approx t^{\frac12}, \qquad 
\| P  \tilde u_\omega(t)\|_{L^2} \in L^1.
\]
This shows that embedded resonances are obstructions to both uniform energy 
bounds and to local energy decay. The generated energy growth is $t^\frac12$.
This is localized in the outgoing region but not concentrated near a light cone.

It may also be interesting to observe that, as part of the proof of our main result 
in Theorem~\ref{t:stat}, we establish the following:

\begin{prop}[Limiting Absorption Principle] \label{p-nonzero}
The following statements are equivalent for stationary, asymptotically flat 
operators $P$ and real nonzero $\omega_0$: 

a) $\omega_0$ is not a resonance.

b) The estimate 
\begin{equation}\label{LE-}
\|u\|_{\LE^1_{\omega_0}} \lesssim \| P_{\omega_0} u\|_{\LE^*}
\end{equation}
holds for all $u \in \LE^1_{\omega_0}$ satisfying the outgoing resonance condition \eqref{outgoing}.

c) The resolvent bound \eqref{LE-fourier} holds uniformly for $\omega \in H$
near $\omega_0$, and the limit   
\[
R_{\omega_0} f = \lim_{H \ni \omega \to \omega_0} R_\omega f 
\]
exists uniformly on compact sets and satisfies the outgoing resonance
condition \eqref{outgoing}.

Furthermore, these properties are stable with respect to small, stationary, asymptotically flat 
perturbations of $P$.
\end{prop}

A significant simplification occurs in the case when $P$ is symmetric:

\begin{thm}[\cite{KT}] \label{nores}
  Assume that $P$ is a symmetric, asymptotically flat wave operator. Then
  there are no nonzero resonances for $P$.
\end{thm}
This is a classical result but, in the context we are using, it was
proved in \cite{KT}.  We emphasize here the self-adjointness
requirement, without which the result is no longer true. However,
assuming nontrapping, it is stable with respect to small
non-self-adjoint perturbations. 

\begin{rem} The proof of this result has two complementary
steps:
\begin{enumerate}
\item  Use the symmetry and the outgoing radiation
condition \eqref{outgoing} to improve the decay of a resonance
$u_\omega \in \LE^1$ to $u_\omega \in \LE^1_0$.

\item Use Carleman estimates to establish a unique continuation type result from infinity
which shows that $u_\omega \in \LE^1_0$ implies $u_\omega = 0$.
\end{enumerate}
The second step above no longer uses the symmetry, and we will exploit this fact in the 
nonstationary part of the paper.
\end{rem}
\bigskip

The second limit we need to consider is as the spectral parameter $\omega$  
goes to zero. In that case the outgoing radiation condition is no longer meaningful,
so we set  
\begin{definition}
A function $u_0$ is  a zero resonance/eigenfunction  for $P$ if $u_0 \in \LE_0$ 
and $ P_0 u_0 = 0$.
\end{definition} 

The condition $u_0 \in \LE_0$ excludes asymptotically constant functions.
Then by standard elliptic analysis one can show that zero resonances
 have better decay at infinity and simpler asymptotics, 
\[
u(r,\omega) = c r^{-1} +o(r^{-1}), \qquad \nabla u = O(r^{-2}).
\]
We remark that zero resonances for $P_0$ correspond to stationary
finite energy solutions for $P$. These have infinite local energy
norm.  Generically the constant $c$ in the above asymptotics is
nonzero. However, unlike the case of resonances, here we may encounter
a stronger decay (though no faster than polynomial, due to unique
continuation results from infinity, see \cite{KT-ucp}). Also, the
homogeneous equation $Pu = 0$ may have further solutions with
polynomial growth in time.

In a direct fashion, one would relate the fact that zero is not an eigenvalue/resonance
to a bound of the form
\begin{equation}\label{zero-res}
\| u\|_{\LE^1} \lesssim \|P_0 u\|_{\LE^*}, \qquad u \in \LE^1_0.
\end{equation}
However, since $P_0$ is a classical elliptic operator, its behavior at infinity 
is perturbative off the Laplacian. But the Laplacian satisfies a larger family of 
weighted estimates, so the choice of weights at infinity here is not so important.
This motivates the following definition:

\begin{definition}
We say that $P$ satisfies a zero resolvent bound (or zero
  nonresonance condition) if
\begin{equation}\label{zero-res2}
\| u\|_{\dot H^1} \leq K_0 \|P_0 u\|_{\dot H^{-1}}.
\end{equation}
\end{definition} 

 By analogy with Proposition~\ref{p-nonzero}, we have the following result,
which is proved in Section~\ref{s:low}:

\begin{prop}[Zero resolvent bound] \label{p-zero}
The following statements are equivalent for stationary, asymptotically flat 
operators $P$: 

a) $0$ is not a resonance or eigenvalue.

b) The estimate \eqref{zero-res} holds for all $u \in \LE^1_0$.

c) The estimate \eqref{zero-res2} holds for all $u \in \dot{H}^1$.

d) The resolvent bound \eqref{LE-fourier} holds uniformly for $\omega \in H$
near $0$, and the limit   
\[
R_{0} f = \lim_{H \ni \omega \to 0} R_\omega f  \in \LE^1_0
\]
exists uniformly on compact sets.

e) The stationary local energy decay bound \eqref{LE-stat}
holds.{\footnote{{In the proof we work with \eqref{S_LE}
  rather than \eqref{LE-stat}.  We note, however, that the latter
  follows from the former for bounded ranges of frequencies and that,
  by \eqref{high_LE}, the bound is generically true for high frequencies.}}}

Furthermore, these properties are stable with respect to small, stationary, asymptotically flat 
perturbations of $P$.
\end{prop}

The discussion above of eigenvalues and resonances is local in nature, so in order 
to obtain global information out of it, it would be very convenient if we were 
able to restrict these considerations to a compact set. This is where the nontrapping 
assumption first comes in. Precisely, we have the following high frequency local 
energy decay bound:

\begin{thm}[High frequency local energy decay]\label{1pt_thm_high}
For any nontrapping asymptotically flat operator $P$  we have the estimate:
\begin{equation}
		\lp{u}{LE^1[0,T]} + \lp{\partial u}{L^\infty L^2[0,T]}  \ \lesssim \ \lp{\partial u(0)}{L^2} 
		+ \lp{\la x \ra^{-2} u}{LE} + \lp{Pu}{LE^*+ L^1L^2[0,T]} \ , \label{high_LE}
\end{equation}
where the implicit constant is uniform in $T$.
\end{thm}

We remark that stationarity plays no role in this theorem. However, as
we shall see in the proof, in the stationary case this yields good
resolvent bounds outside a compact set of frequencies.  Now we can
state our main result for stationary operators:

\begin{thm}\label{t:stat}
 a)  Let $P$ be stationary, nontrapping and asymptotically
  flat. Then local energy decay holds for $P$ if and only if $P$ has
  no negative eigenvalues or real  resonances.

 b)  Let $P$ be stationary, symmetric, nontrapping and asymptotically
  flat. Then local energy decay holds for $P$ if and only if $P$ has
  no negative eigenvalues or zero  eigenvalues/resonances.
\end{thm}

We remark that this result  is also stable with respect to small asymptotically
flat  perturbations. In particular, in the symmetric case one can continue
the property of absence of negative eigenvalues along continuous families
of operators $P(h)$ for as long as no zero eigenvalues/resonances occur
(see Proposition~\ref{p:3S} below). 

\bigskip

In what follows we take a closer look at the symmetric case if we
retain the no zero resonance assumption \eqref{zero-res2}, but drop
the requirement that there are no negative eigenfunctions.  As the
eigenfunctions with negative imaginary part are isolated and localized to a
compact\footnote{This is a consequence of the nontrapping assumption,
  and part of the proof of the last theorem.} set, it follows that
there are finitely many of them. Further, by Fredholm theory, each of
them has an associated finite dimensional invariant subspace. The sum
of these subspaces is associated to solutions which grow exponentially
in time.  One can argue in a similar manner for eigenvalues with
positive imaginary part (which are the complex conjugates of the ones
with negative imaginary part). This heuristic reasoning leads to the
following theorem:

\begin{thm}\label{t:stat+}Let $P$ be symmetric, stationary, nontrapping, asymptotically
  flat, and satisfying the zero resolvent bound \eqref{zero-res2}.
Then there is a direct sum decomposition of the energy space into invariant subspaces
\begin{equation}
\mathcal E = S_0 + S^+ + S^-,
\end{equation}
with $S^+$ and $S^-$ of equal finite dimension, so that the wave flow on the
subspaces $S^+$, $S^0$ and $S^-$ is as follows:

a) Solutions in $S^+$ are spatially localized and grow exponentially in time.

b) Solutions in $S^0$ are uniformly bounded and satisfy local energy decay.

c)   Solutions in $S^-$ are spatially localized and decay exponentially in time.
\end{thm}

Likely this result extends fully to the non-self-adjoint case under a no
resonance assumption; we leave this as an open problem. Further, this
result has a nice stability property with respect to $P$.  Precisely, as a consequence of
Propositions~\ref{p-nonzero}, \ref{p-zero} and \ref{p:ppm} we have
the following:

\begin{prop} \label{p:3S} Let $P(h)$ be a family of
  stationary, symmetric, nontrapping and asymptotically flat wave
  operators which is continuous in the $AF$ topology.  Suppose that for each $h$ there are no zero
  eigenvalues/resonances for $P(h)$.  Then the (equal) dimensions of
  the  invariant subspaces $S^\pm$ for $P(h)$ are independent of $h$.
 \end{prop}

\subsection{The nonstationary case}
 
Our main goal here is to show that the results in the stationary case
extend without any change to almost symmetric, almost stationary
operators.  We remark that this includes small AF perturbations of
stationary metrics, but it is certainly not limited to that. While the results
here are very similar\footnote{Except for the almost symmetry hypothesis,
which we need in order to avoid resonaces related issues.} to those in the 
stationary case, our approach requires some  new ideas, in part to account
for the need to have good bounds of the implicit constants in the estimates
in terms of our control parameters $R_0$, $M_0$, $T_0$ and $K_0$.     

Our starting point here is a local energy decay estimate with bounds in terms of
initial and final energies:

\begin{thm}[Two point local energy decay]\label{2pt_thm}
  Let $P$ be an $\epsilon$-slowly varying, nontrapping and
  asymptotically flat operator, which
  satisfies the zero nonresonance condition \eqref{zero-res2} for some
  $K_0>0$. Suppose that $\epsilon$ is small enough
\[
\epsilon \ll_{R_0,M_0,T_0,K_0} 1.
\]
  Then we have the estimate:
\begin{equation}
		\lp{u}{LE^1[0,T]}+ \lp{\partial u}{L^\infty L^2[0,T]}   \ \lesssim \ \lp{\partial u(0)}{L^2} 
		+ \lp{\partial u(T)}{L^2} + \lp{P u}{LE^*+ L^1L^2[0,T]} \ , \label{2pt_LE}
\end{equation}
where the implicit constant is uniform in $T$ and depends only on
$R_0$, $M_0$, $T_0$ and $K_0$..
\end{thm}

We remark here one key advantage of working with the two point local
energy decay, namely that it does not see the negative eigenvalues
of $P$ or nonzero resonances. Because of this, such an estimate is not
limited to the (almost) symmetric case. However, in order to use this
estimate to get bounds on the energy growth rate, we do need to use an
almost conservation of energy, which requires the almost symmetry
condition.  Under an additional almost symmetry assumption, from this
theorem one gains the following simple dichotomy:

\begin{thm}[Exponential dichotomy]\label{exp_dich_thm}
  Let $P$ be an $\epsilon$-slowly varying, $\epsilon$-almost
  symmetric, nontrapping and asymptotically flat operator, which
  satisfies the zero nonresonance condition \eqref{zero-res2} for
  some $K_0>0$. Suppose that $\epsilon$ is small enough
\[
\epsilon \ll_{R_0,M_0,T_0,K_0} 1.
\]
Then there exists an $\alpha_0>0$, depending on the parameters
${R_0,M_0,T_0,K_0}$, such that any solution $u$ to the inhomogeneous equation
$Pu = f$ with $u[0] \in \mathcal E$ and $f \in LE^*+ L^1L^2$
has one of the following two properties:

(i) Exponential growth for large $t$:
\begin{equation}
		\lp{\partial u(t)}{L^2} \ \ge c e^{\alpha_0 t}(\lp{\partial u(0)}{L^2}+
		 \lp{Pu}{LE^*+ L^1L^2[0,\infty)}) \ \qquad t > t_0 , \label{slow_growth}
\end{equation}

(ii)  Global bound and local energy decay estimate:
\begin{equation}
		\lp{u}{LE^1[0,T]} +\lp{\partial u}{L^\infty L^2} \ \lesssim \ \lp{\partial u(0)}{L^2} 
		 + \lp{Pu}{LE^*+  L^1L^2[0,\infty)} \ . \label{LE-last}
\end{equation}
\end{thm}

Heuristically, the role played by the (almost) symmetry is to
eliminate the nonzero resonances, at least in the stationary case.
Indeed, the proof of this result is a time dependent take on the classical
proof of the absence of embedded resonances for the stationary problem.

Finally, for a more accurate result we can also take into account the exponentially 
decaying solutions.

\begin{thm}[Exponential trichotomy]\label{exp_trich_thm}
Let $P$ satisfy all of the assumptions of Theorem \ref{exp_dich_thm}.
Then there exist  a direct sum decomposition
\[
\dot H^1 \times L^2 = S^+ + S^0 + S^-
\]
with $\kappa$ dimensional linear subspaces $S^\pm$, and
corresponding  commuting projections $P^0,\ P^\pm$,
such that if $\mathcal{F}(t)$ denotes
the flow of the homogeneous equation $P u=0$ on  $\dot{H}^1\times L^2$ with data at $t=0$
then:
\begin{align}
		\lp{\mathcal{F}(t)P^- \mathcal{F}^{-1}(s)}{\dot{H}^1\times L^2\to \dot{H}^1\times L^2}
		\ &\leq\  Ce^{-\alpha(t-s)} \ , &t\geq s\geq 0 \ , \label{ed1}\\
		\lp{\mathcal{F}(s)P^+ \mathcal{F}^{-1}(t)}{\dot{H}^1\times L^2\to \dot{H}^1\times L^2}
		\ &\leq\  Ce^{-\alpha(t-s)} \ , &t\geq s\geq 0 \ , \label{ed2}\\
		\lp{\mathcal{F}(t)P^0 \mathcal{F}^{-1}(s)}{\dot{H}^1\times L^2\to \dot{H}^1\times L^2}
		\ &\leq\  C \ , &\hbox{all\ \ } t,s\geq 0 \ . \label{ed3}
\end{align}
Here $\mathcal{F}^{-1}(s)$ is the algebraic inverse map of $\mathcal{F}(s)$ on $\dot{H}^1\times L^2$
(its existence as a bounded operator is easily established via standard energy estimates).
\end{thm}

By the previous theorem, the solutions with data in $S^0$ will also
satisfy local energy decay bounds.  To unravel a bit what this last
Theorem says let $P^\pm(t)= \mathcal{F}(t)P^\pm \mathcal{F}^{-1}(t)$
be the flow projections onto the image $S^\pm(t)=\mathcal{F}(t)S^\pm$,
set $P^0(t)=1-P^+(t)-P^-(t)$ with corresponding image $S^0(t)$, and
let $\mathcal{F}(t,s)= \mathcal{F}(t) \mathcal{F}^{-1}(s)$ denote the
propagator from times $s\to t$. Then estimates
\eqref{ed1}--\eqref{ed3} can be written in the equivalent form (for a
different $C$):
\begin{align}
		\lp{\mathcal{F}(t,s) }{S^-(s)\to  S^-(t)}
		\ &\leq\  Ce^{-\alpha(t-s)} \ , &t\geq s\geq 0 \ , \label{ed1'}\\
		\lp{\mathcal{F}(t,s) }{S^+(s)\to  S^+(t)}
		\ &\geq\  C^{-1}e^{\alpha(t-s)} \ , &t\geq s\geq 0 \ , \label{ed2'}\\
		C^{-1}\ \leq\ \lp{\mathcal{F}(t,s) }{ S^0(s)\to  S^0(t)}
		\ &\leq\  C \ , &\hbox{all\ \ } t,s\geq 0 \ . \label{ed3'}\\
		\lp{P^{\pm,0}(t)}{\dot{H}^1\times L^2\to \dot{H}^1\times L^2}  
		\ &\leq\  C \ , &\hbox{all\ \ } t,s\geq 0 \ . \label{ed4'}
\end{align}
In other words the flow along $S^-(t)$ is uniformly contracting at an
exponential rate forward in time, while the flow along $S^+(t)$ is
uniformly expanding at an exponential rate forward in time. On the
other hand $S^0(t)$ is a center manifold on which the flow has
uniformly bounded energy both forward and backward in time. According
to Theorem \ref{2pt_thm} the flow along $S^0(t)$ is therefore
uniformly dispersive with bounds on $\lp{\phi}{LE[s,t]}$ in terms of
$\lp{\partial\phi(s)}{L^2}\approx \lp{\partial \phi(t)}{L^2}$ for any
set of times $s,t\geq 0$. Finally, there is a bit of additional
information given by \eqref{ed4'} which says that the time dependent
subspaces $S^{\pm,0}(t)$ are ``uniformly transverse''.

\subsection{ An overview of the paper }

The bulk of the article is devoted to the proof of the various local energy decay bounds.
This is developed as follows:

In Section~\ref{s:ext} we expand upon the result of \cite{MT}, see Theorem~\ref{t:small},
which asserts that local energy decay holds for small
asymptotically flat perturbations 
of the d'Alembertian. Our goal is to work with operators $P$ which have this property 
outside a cylinder and to describe two efficient ways of truncating the estimates
to such an exterior region.

Section~\ref{s:high} is devoted to the proof of the high frequency
local energy bound in Theorem~\ref{1pt_thm_high}. This is a classical positive 
commutator estimate, with the added difficulty that we do not use any energy 
conservation property.

The following two sections contain the remaining two blocks in the
proof of the two point local energy estimate in
Theorem~\ref{2pt_thm}. We begin with the medium frequencies in Section~\ref{s:med};
the analysis there is based on Carleman estimates and can be viewed as 
as time dependent version of the Carleman estimates for the resolvent in \cite{KT}.
The low frequency analysis follows in Section~\ref{s:low}.

The three building blocks above, namely the high, medium and 
low frequency analysis are assembled together in Section~\ref{s:2pt}
in order to prove the two point local energy estimate in
Theorem~\ref{2pt_thm}. We also give there the closely related proof 
of the exponential dichotomy in Theorem~\ref{exp_dich_thm}.
 
In Section~\ref{s:stat} we consider the relation between the local energy estimates 
and resolvent bounds, leading to the proof of Propositions~\ref{p:le-res}, \ref{p-nonzero}
and \ref{p-zero}. We conclude with the proof of Theorem~\ref{t:stat}.

Last but not least, in Section~\ref{s:nonstat} we fully exploit the
two point local energy decay bound. We begin with the stationary case,
where we establish the exponential trichotomy in
Theorem~\ref{t:stat+}.  Then we switch to the nonstationary case, and
prove Theorem~\ref{exp_trich_thm}.


\section{Exterior estimates for  asymptotically 
flat perturbations of Minkowski}
\label{s:ext}

There are two exterior local energy decay estimates which we will
use. Our starting point is the standard local energy decay estimate
for small AF perturbations of $\Box$ in Theorem~\ref{t:small}, which
was proved in \cite{MT}.  Applying a cutoff function which selects 
an exterior region $\{r \geq R\}$ we directly obtain a truncated estimate
\begin{equation}
	\lp{  \phi }{LE^1_{>R}} + \| \partial \phi \|_{L^\infty L^2_{>R}}  
	\lesssim \ \|\partial \phi(0)\|_{L^2_{>R}}+ 	\lp{  \phi }{LE^1_{R}}   + 
	\lp{P\phi}{LE^*_{>R}}.
		\label{small-cut}
\end{equation}
Here we explore two nontrivial ways of truncating the bound in
Theorem~\ref{t:small} to an exterior region.  The first will handle
very low frequencies, while the second is useful in situations where
one can absorb a low frequency error.

Our first truncated estimate uses only the localized $L^2$ norm of
$\partial u$ as the truncation error: 

\begin{prop}
Suppose $P$ is asymptotically flat. Then for $R > R_0$ we have
\begin{equation}
	\lp{  \phi }{LE^1_{>R}}  
	\lesssim \ \|\partial \phi(0)\|_{L^2_{>R}}+ \lp{\partial \phi}{LE_{R}}  + 
	\lp{P\phi}{LE^*_{>R}}.
		\label{low_Mourre}
\end{equation}
\end{prop}

\begin{proof}
The restriction $R > R_0$ simply says that we are in a regime where $P$ is a small AF
perturbation of $\Box$. Without any restriction in generality we assume that this is the case 
globally, with the same AF constant $\bf c$ as in \eqref{AF-def}. 

To prove the above bound we will apply the global LE bound
for small AF perturbations of $\Box$ to a suitably chosen extension of $\phi$ inside  
the $R$ cylinder. To define this extension we introduce the local average of $\phi$
adapted to the $R$ annulus as 
\[
\bar \phi_{R}(t_0) = R^{-4} \int \phi(t,x) \chi\Bigl(\frac{t-t_0}{R}, \frac{|x|}{R}\Bigr) dx dt
\]
where $\chi(t,r)$ is a smooth, nonnegative bump function supported in $[-1,1]
\times [1,2]$, with unit integral.  By integration by parts, the
Schwarz inequality, and Young's inequality, we have
\begin{equation}
  \label{convolution}
  \|\partial_t^j \bar\phi_R\|_{L^2}\lesssim R^{-j} \|\partial
  \phi\|_{LE_R},\quad j\ge 1.
\end{equation}
We also have a Poincar\'e type inequality
\begin{equation}\label{poincare}
 \| \la r \ra^{-1}(\bar \phi_{R} -\phi)\|_{LE_R} \lesssim \| \partial \phi\|_{LE_R}. 
\end{equation}

Now we extend the function $\phi$ inside the $R$ cylinder by
\[
\phi_{ext} (t,x) = \beta_{>R}(|x|) \phi(t,x) + (1-\beta_{>R})(|x|)  \bar \phi_{R}(t),
\]
where, e.g., we may take $\beta_{>R}(r)=1-\beta(r/R)$ with $\beta\in C^\infty(\R_+)$ with $\beta\equiv 1$
on $[0,1]$ and supported in $[0,2]$.
By Theorem \ref{t:small}, it follows that
\[\|\phi_{ext}\|_{LE^1}\lesssim \|\partial \phi_{ext}(0)\|_{L^2}
+ \|P\phi_{ext}\|_{LE^*},\]
and its local energy norm is given by 
\[
\| \phi_{ext} \|_{LE^1} \approx \|\beta_{>R} \phi \|_{LE^1} + \|  \bar \phi_{R}\|_{L^2}
+ R \|\partial_t \bar \phi_{R}\|_{L^2}.
\]

On the other hand, we compute
\[
P \phi_{ext} = \beta_{>R} P \phi + (1-\beta_{>R})  P \bar \phi_{R} + [P,\beta_{> R}] (\phi - \bar \phi_{R}).
\]
The commutator is first order and supported inside the $R$ annulus with
\[[P,\beta_{>R}] (\phi-\bar\phi_R) = O(R^{-1})|\partial \phi| + O(R^{-1})
|\partial_t \bar\phi_{R}| + O(R^{-2})|\phi-\bar\phi_R|,\quad |x|\approx R.\]
Thus, using \eqref{convolution} and \eqref{poincare} respectively on
each of the last two terms, we see that
\[\|[P,\beta_{>R}](\phi-\bar\phi_R)\|_{LE^*}\lesssim
\|P\phi\|_{LE^*_{>R}} + \|\partial \phi\|_{LE_R}.\] Moreover, we note
that, as we are taking without loss of generality
$\llp{(g-m,A,V)}{AF}\le \mathbf{c}$,
\[\|(1-\beta_{>R})P\bar\phi_R\|_{LE^*} \lesssim R^2
\|\partial_t^2\bar\phi_R\|_{L^2} + R \|\partial_t \bar\phi_R\|_{L^2} +
\mathbf{c}\|\bar\phi_R\|_{L^2}.\]

Using \eqref{convolution}, using $\mathbf{c}\ll 1$ in order to bootstrap, and combining we have that
\[\|\beta_{>R} \phi\|_{LE^1} + \|\bar\phi_R\|_{L^2} \lesssim
\|\partial\phi_{ext}(0)\|_{L^2} + \|P\phi\|_{LE^*_{>R}} + \|\partial
\phi\|_{LE_R}.\]
Observing that the Schwarz inequality yields
\[\|\partial \phi_{ext}(0)\|_{L^2} \lesssim \|\partial
\phi(0)\|_{L^2_{>R}} + \|\partial \phi\|_{LE_R}\]
then completes the proof.
%
\end{proof}

Our second truncated estimate uses instead a localized $L^2$ norm of
$u$ as the truncation error. The price to pay is that we need the
energy of $u$ at both the initial and the final time. This bound will
be used later in the proof of the Carleman estimate \eqref{e:carl-cut}
in Proposition~\ref{p:Carl-cut}.

 \begin{prop}\label{Mourre_prop}
Suppose $P$ is asymptotically flat. Then for $R\geq R_0$ we have:
\begin{align}
		\lp{  u }{LE^1_{>R}} 
		\ &\lesssim \|\partial u(T)\|_{L^2_{>R}} + \|\partial
                u(0)\|_{L^2_{>R}} + 
		\lp{Pu}{LE^*_{>R}} + \ R^{-1}\lp{u}{LE_R}
		\ . \label{high_Mourre}
\end{align}
\end{prop}
 






\begin{proof}
  This is a multiplier estimate, which is similar to an estimate for
  the corresponding Schr\"odinger equation in \cite{MMT}, and is related 
to the earlier work \cite{BT}.

  Setting $\Box_g = \Box_{0,g}$, we begin by noticing that if
  $Q(x,D_x)$ is self-adjoint, then
\begin{multline}\label{commute}\frac{d}{dt}\Bigl\{2\Re\la g^{00}D_tu, Qu\ra + 2\Re\la g^{j0}D_ju,
Qu\ra\Bigr\} \\= -2\Im\la \Box_g u, Qu\ra + \la i[D_j g^{jk}D_k, Q]u,u\ra +
2\Im\la [Q,g^{j0}D_j]u, D_tu\ra -\la
i[Q,g^{00}]D_tu,D_tu\ra.\end{multline}
Here, the inner product is that of $L^2(\R^3)$.
We set 
\[Q=\beta(r) f(r) \frac{x_l}{r}g^{lm}D_m + D_m \beta(r)
f(r)\frac{x_l}{r}g^{lm}.\]
Here $f(r) = r/(r+2^j)$ where $2^j\ge R$, and $\beta$ is a smooth,
increasing function so that $\beta\equiv 0$ for $r\le R$ and
$\beta\equiv 1$ for $r\ge 2R$.
Then
\begin{align*}[D_jg^{jk}D_k,\beta(r)f(r)\frac{x_l}{r}g^{lm}D_m] &= [D_j
g^{jk}D_k,\beta(r)] f(r)\frac{x_l}{r}g^{lm}D_m + \beta(r) [ D_j
g^{jk}D_k, f(r)\frac{x_l}{r}g^{lm}D_m]\\
[D_j g^{jk}D_k, D_m \beta(r)
f(r)\frac{x_l}{r}g^{lm}]&=  [ D_j
g^{jk}D_k, D_m f(r)\frac{x_l}{r}g^{lm}]\beta+ D_m f(r) \frac{x_l}{r}g^{lm} [D_j
g^{jk}D_k,\beta(r)].
\end{align*}
Using that
\begin{align*}
  [D_jg^{jk}D_k,\beta] &= \frac{2}{i} D_j g^{jk}\frac{x_k}{r}\beta'(r)
+ \partial_j\Bigl(g^{jk}\beta'(r)\frac{x_k}{r}\Bigr)\\
&= \frac{2}{i} g^{jk}\beta'(r) \frac{x_j}{r} D_k - \partial_j\Bigl(g^{jk}\beta'(r)\frac{x_k}{r}\Bigr),
\end{align*}
we have
\begin{multline*}
  [D_j
g^{jk}D_k,\beta(r)] f(r)\frac{x_l}{r}g^{lm}D_m + D_m f(r) \frac{x_l}{r}g^{lm} [D_j
g^{jk}D_k,\beta(r)] \\= \frac{4}{i} D_j g^{jk}\frac{x_k}{r}\beta'(r)
f(r) \frac{x_l}{r}g^{lm}D_m - \frac{1}{i}\partial_m\Bigl(f(r)\frac{x_l}{r}g^{lm}\partial_k\Bigl(g^{jk}\beta'(r)\frac{x_k}{r}\Bigr)\Bigr).
\end{multline*}
By substituting $g^{jk}= m^{jk} + (g^{jk}-m^{jk})$, it follows then
that
\begin{multline*}
  \int \la i[D_jg^{jk}D_k,Q]u,u\ra\,dt \gtrsim 4 \int \int
  f(r)\beta'(r)\Bigl(g^{jk}\frac{x_k}{r} \partial_j u\Bigr)^2\,dx\,dt
\\+\int \la i[-\Delta, f(r)\frac{x_l}{r}D_l + D_l
\frac{x_l}{r}f(r)]u,u\ra\, dt
 - \|R^{-1}u\|^2_{LE_R} - \llp{ g-m }{1, >R}\|u\|^2_{LE^1_{>R}}.
\end{multline*}
Fixing $Q_0 = f(r)\frac{x_l}{r}D_l + D_l
\frac{x_l}{r}f(r)$, a standard computation (see, e.g.,
  \cite{MetSo}) gives
\[[-\Delta,Q_0] = \frac{4}{i}D_k \frac{x_k}{r}f'(r)\frac{x_j}{r}D_j + 
\frac{4}{i} D_l\Bigl(\delta_{lk}-\frac{x_kx_l}{r^2}\Bigr)
\frac{f(r)}{r} \Bigl(\delta_{jk}-\frac{x_kx_j}{r^2}\Bigr)D_j -
\frac{1}{i}\Delta \partial_k \Big(\frac{x_k}{r}f(r)\Bigr).\]
Using that $f$, $f'$, and
$-\Delta\Bigl((n-1)\frac{f(r)}{r}+f'(r)\Bigr)$ are everywhere positive
(on $\R^3$) and that $\frac{f(r)}{r}, f'(r)\approx R^{-1}$ and
$-\Delta\Bigl((n-1)\frac{f(r)}{r}+f'(r)\Bigr)\approx R^{-3}$ when
$|x|\approx R$, this shows that
\[ \int \la i[D_jg^{jk}D_k,Q]u,u\ra\,dt \gtrsim \|\nabla_x u\|^2_{LE_{>R}}
 + \||x|^{-1} u\|^2_{LE_{>R}} - \|R^{-1}u\|^2_{LE_R} - \llp{ g-m }{1,
   >R}\|u\|^2_{LE^1_{>R}}.\]
Integrating \eqref{commute} over $t$, plugging this in, and using a Hardy inequality, it
follows that
\begin{multline}\label{LEext1}
\|\nabla_x u\|^2_{LE_{>R}} + \||x|^{-1} u\|^2_{LE_{>R}} \lesssim
\|\partial u(T)\|^2_{L^2_{>R}} + \|\partial u(0)\|^2_{L^2_{>R}} + \|\Box_g
u\|_{LE^*_{>R}} \|Qu\|_{LE_{>R}} \\+ \|R^{-1} u\|^2_{LE_R} + \llp{g-m}{1,>R}\|u\|^2_{LE^1_{>R}}.
\end{multline}

We recover the time derivatives using a Lagrangian correction.  To
this end, we compute
\begin{multline*}
  \frac{d}{dt} 2\Im\la (g^{00}D_t + 2g^{0j}D_j)u, \beta(r) f'(r)u\ra = 2 \Re\la
  \Box_g u, \beta(r)f'(r) u\ra + 2\Im\la (\partial_t g^{0j}) D_j u, \beta(r)f'(r)u\ra \\- 2\Im\la
  (\partial_j g^{0j})D_t u, \beta(r)f'(r)u\ra - 2\la g^{00} D_t u,
  \beta(r)f'(r)D_t u\ra 
\\+\la \partial_k(g^{jk})u, \partial_j(\beta(r)f'(r))u\ra + \la g^{jk}u,\partial_k\partial_j(\beta(r)f'(r))u\ra
- 2\Re\la g^{jk}D_k u, \beta(r)f'(r)
  D_j u\ra.
\end{multline*}
From this, it follows that
\begin{multline}\label{LEext2}
\|\partial_t u\|^2_{LE_{>R}} \lesssim \|\partial u(T)\|^2_{L^2_{>R}} + \|\partial u(0)\|^2_{L^2_{>R}} +\|\Box_g u\|_{LE^*_{>R}}
\||x|^{-1} u\|_{LE_{>R}} \\+ \|\nabla_x u\|^2_{LE_{>R}} + \||x|^{-1}
u\|^2_{LE_{>R}} + \|R^{-1}u\|_{LE_R}^2 + \llp{g-m}{1,>R}\|u\|^2_{LE^1_{>R}}.
\end{multline}

The combination of \eqref{LEext1} and \eqref{LEext2} results in
\[\|u\|^2_{LE^1_{>R}} \lesssim \|\partial u(T)\|^2_{L^2_{>R}} +
\|\partial u(0)\|^2_{L^2_{>R}} + \|\Box_g u\|_{LE^*_{>R}}^2 +
\|R^{-1}u\|_{LE_R}^2 + \llp{g-m}{1,>R}\|u\|^2_{LE^1_{>R}}.
\]
We finally notice that 
\[\|\Box_g u\|_{LE^*_{>R}} \lesssim \|Pu\|_{LE^*_{>R}} + \llp{(g-m, A,
V)}{AF,>R}\|u\|_{LE^1_{>R}}.\]
Provided that $\mathbf{c}$ in \eqref{AF-def} is sufficiently small,
which can be guaranteed by fixing $R_0$ sufficiently large, we may
bootstrap to complete the proof.
\end{proof}


\section{High frequency analysis}
\label{s:high}

Here we prove Theorem \ref{1pt_thm_high}. We begin with several
simplifications.  

First we argue that it suffices to prove \eqref{high_LE} when the data
and forcing term are supported in $\{|x|<2R_0\}$.  Indeed, letting
$\tilde{P}$ be a small AF perturbation of $\Box$ that agrees with $P$
for $|x|>R_0$, we set $v$ to solve $\tilde{P}v = f$ with $v[0]=u[0]$.
Then, by Theorem \ref{t:small}, $\|v\|_{LE^1}$ is controlled by the
right side of \eqref{high_LE}.  Then it suffices to establish
\eqref{high_LE} for $u_1 = u- \beta_{>R_0} v$, which solves an
equation with data and forcing term supported in $\{|x|<2R_0\}$.

We next reduce to the case that $u[0]=0$ and $f\in LE^*$.  Indeed,
  suppose that we know 
\begin{equation}
		\lp{u}{LE^1[0,T]} + \lp{\partial u}{L^\infty L^2[0,T]}  
		\lesssim \lp{\la x \ra^{-2} u}{LE} + \lp{Pu}{LE_{comp}^*[0,T]} \label{high_LEa}
\end{equation}
when $u[0]=0$.  Let $w$ solve $Pw = g\in L^1L^2_{comp}$ with nontrivial
$w[0]$.  We write $w = \sum_k w_k$ where $Pw_k = \mathbf{1}_{[k,k+1]}(t) g$
and $w_k[0]=0$, $k>0$ and $w_0[0]=w[0]$.  We use $\sum_k
\beta_{[k,k+1]}(t) w_k$, where $\beta_{[k,k+1]}(t)$ is a test function
that is identically one on $[k,k+1]$, as an approximate solution.
By \eqref{local-en} and
Duhamel's principle, 
\[\sum_k \Bigl(\|\beta_{[k,k+1]} w_k\|_{LE^1} +
\|\partial (\beta_{[k,k+1]} w_k)\|_{L^\infty L^2}\Bigr) \lesssim \|\partial
w(0)\|_{L^2} + \sum_k \int_k^{k+1} \|g(s)\|_2\,ds.\]
The difference $w-\sum_k \beta_{[k,k+1]} w_k$ is estimated with
\eqref{high_LEa}.  Using that
$g$ is compactly supported and finite propagation speed, one obtains
\[\|\la x\ra^{-2} \sum_k \beta_{[k,k+1]}w_k\|_{LE} + \|P(w-\sum_k \beta_{[k,k+1]} w_k)\|_{LE^*} \lesssim \sum_k
\|\partial w_k\|_{L^\infty_{[k,k+1]}L^2},\]
which can be bounded as before.  So it remains to show \eqref{high_LEa}.

Next we remark that, while the result in the theorem is primarily 
about the high frequencies, it does have a low frequency component.
To account for that, we separate the high frequency part and shall prove the following
estimate, which uses a frequency threshold $\lambda \gg 1$:
\begin{equation} \label{main-high} 
\| u_{\geq \lambda} \|_{LE^1_{<
      2R_0}} \lesssim \|f\|^\frac12_{LE^*} \|u\|^\frac12_{LE^1}
  +\lambda^{-\delta} \|u\|_{LE^1},
 \end{equation}
with an implicit constant which depends only on $M_0,R_0,T_0$.
Here the last term accounts for the low frequency errors. 

We first show that this 
bound implies the desired estimate \eqref{high_LEa}. Adding in the low frequencies 
this becomes 
\[
\|  u \|_{LE^1_{< 2R_0}} \lesssim \|f\|^\frac12_{LE^*} \|u\|^\frac12_{LE^1} +\lambda^{-\delta}
\|u\|_{LE^1} + \lambda \lp{\la x \ra^{-2} u}{LE}.
\]
From here we can also obtain the uniform energy piece.  Indeed, using
the Mean Value Theorem, there is a sequence $t_k\in [k,k+1]$ so that
\[\|\partial u(t_k)\|_{L^2_{<2R_0}} \lesssim \|u\|_{LE^1_{<2R_0}},\]
and \eqref{local-en} allows us to estimate the energy at an arbitrary
time $t$ by an element of this sequence.
Next we apply \eqref{small-cut} to complete the local norm above to a full local energy 
norm
\[
\|  u \|_{LE^1} + \|\partial u\|_{L^\infty L^2}\lesssim \|f\|_{LE^*}^{1/2} \|u\|_{LE^1}^{1/2} + \|f\|_{LE^*}
+\lambda^{-\delta} \|u\|_{LE^1} + \lambda \lp{\la x \ra^{-2} u}{LE}.
\]
Then \eqref{high_LEa} follows by choosing $\lambda$ large enough (depending on the 
implicit constant, which in turn depends on $M_0,R_0,T_0$).

Now we turn our attention to the proof of \eqref{main-high}.  For this
we will use the positive commutator method with a suitably chosen
escape function.  To simplify the argument without loss we take
$g^{00}=1$.  Indeed, hypothesis \eqref{zero_ell} guarantees that
$g^{00}$ is uniformly bounded away from zero, and dividing through as
in \cite[Section 3]{MT} preserves the assumptions on the other
coefficients.  For convenience we will state separately the result 
concerning the existence of the escape function.  For that we 
introduce the notations 
\[
p(t,\tau,x,\xi) = \tau^2 - 2g^{0j}\tau\xi_j - g^{ij}\xi_i\xi_j , \qquad  s_{skew} =
\Im A_\alpha g^{\alpha 0} \tau + q \Im A_\alpha
g^{\alpha k}\xi_k.
\]
for  the principal symbols of the self-adjoint, respectively the skew-adjoint parts 
of $P$. Then we have:

\begin{lem} 
  Under the same assumptions as in Theorem~\ref{1pt_thm_high}, there
  exists smooth real symbols $q = \tau q_0 + q_1$ and $ m$ where
  $q_0 \in S^0$, $q_1 \in S^1$ and $m \in S^0$ are supported in
  $|\xi|,|\tau| \gtrsim \lambda$, so that the following relation holds:
\begin{equation}
\{ p, q\} + p m + s_{skew} q \gtrsim \mathbf{1}_{|\xi|,|\tau| \geq \lambda}\langle x \rangle^{-2} 
(|\xi^2| +\tau^2).
\end{equation}
\end{lem}

\begin{proof}
We factor  $p(t,\tau,x,\xi)$ as
\[\tau^2 - 2g^{0j}\tau\xi_j - g^{ij}\xi_i\xi_j =
(\tau-a^+(t,x,\xi))(\tau-a^-(t,x,\xi)).\]
Here $a^{\pm}$ are $1$-homogeneous in $\xi$ and $a^+>0>a^-$.  See,
e.g., \cite[Section 6]{MT}.

Set $p^{\pm}(t,\tau,x,\xi)=\tau-a^{\pm}(t,x,\xi)$. On each portion of
the light cone $\tau=a^{\pm}(t,x,\xi)$, we first construct a
multiplier $q^{\pm}$ akin to that used in, e.g., \cite{Doi}, see also
\cite{ST}. Precisely, we claim that we can find smooth homogeneous
zero order symbols $q^{\pm}$ with the property that 
\begin{equation}\label{Hpq}
H_{p^{\pm}}q^{\pm} > c_k 2^{-k},\qquad |x| \approx 2^k.
\end{equation}
To produce such symbols set $\chi$ to be a smooth, monotonically
decreasing function with $\chi\equiv 1$ for $|x|\le 1$ and $\chi\equiv
0$ for $|x|>2$, and define $\chi_M(|x|)=\chi(|x|/M)$.  Then, letting
$(t^{\pm}_s,\tau^{\pm}_s,x^{\pm}_s,\xi^{\pm}_s)$ solve
\[\dot{t}^{\pm}_s = p^{\pm}_\tau(t^{\pm}_s,\tau^{\pm}_s,x^{\pm}_s,\xi^{\pm}_s),\quad \dot{\tau}^{\pm}_s =
-p^{\pm}_t(t^{\pm}_s,\tau^{\pm}_s,x^{\pm}_s,\xi^{\pm}_s)\]
\[\dot{x}^{\pm}_s = p^{\pm}_\xi(t^{\pm}_s,\tau^{\pm}_s,x^{\pm}_s,\xi^{\pm}_s),\quad \dot{\xi}^{\pm}_s =
-p^{\pm}_x(t^{\pm}_s,\tau^{\pm}_s,x^{\pm}_s,\xi^{\pm}_s)\]
with $(t^{\pm}_0,\tau^{\pm}_0,x^{\pm}_0,\xi^{\pm}_0)=(t,\tau,x,\xi)$,
set
$q_{in}^{\pm}(t,\tau,x,\xi) = -\int_0^\infty \chi_{R}(|x^{\pm}_s|)\,ds$.  By the
nontrapping condition, this is bounded, and one can calculate
$H_{p^{\pm}} q_{in}^{\pm}=\chi_{R}(|x|)$.  

We shall supplement $q_{in}^\pm$ with a multiplier that is inspired by
that used to prove \eqref{high_Mourre}.  See, also, \cite{MMT}.  Set
\[q^{\pm}_{out} =  - (1-\chi_{R})(|x|) f(|x|) a^{\pm}_{\xi_j} 
\frac{x_j}{|x|}\]
where $f$ satisfies $f(x)\approx 1$ for $|x|\ge R_0$ and
$f'(|x|)\approx c_k 2^{-k}$ for $|x|\approx 2^k$.\footnote{Indeed, we
  may choose $f(r) = \exp(\sigma \int_0^r c(s) s^{-1}\,ds)$ where $c(s)$ is
  constructed from the sequence $c_k$ as was done in \cite[\S 2]{MT}
  and $\sigma$ is a large parameter.}  We
can then compute
\begin{multline*}
  H_{p^{\pm}} q_{out}^{\pm} = a^{\pm}_{\xi_k}\frac{x_k}{|x|}
  (1-\chi_{R})(|x|)f'(|x|) a^{\pm}_{\xi_j}\frac{x_j}{|x|} +
  a^{\pm}_{\xi_k}\Big(\delta_{jk} -\frac{x_k}{|x|}\Bigr)
  (1-\chi_{R})(|x|)\frac{f(|x|)}{|x|} \Bigl(\delta_{jl} -
  \frac{x_l}{|x|}\Bigr) a^{\pm}_{\xi_l}
\\ - R^{-1} \chi'(|x|/R) a^{\pm}_{\xi_k}\frac{x_k}{|x|}
f(|x|) a^{\pm}_{\xi_j}\frac{x_j}{|x|}
+\O(\la x\ra |\partial g|) (1-\chi_R)(|x|) |x|^{-1},
\end{multline*}
and note that the third term in the right hand side is nonnegative.
Using \eqref{AF-def}, it follows that $H_{p^{\pm}}q_{out}^{\pm}$ is
non-negative everywhere and strictly positive on $|x|\in [2R, 4R]$.

We now set
\[q^{\pm} = q^{\pm}_{out} + c \chi_{4R}(x) q_{in}^\pm.\]
Then
\[H_{p^{\pm}}q^{\pm} = H_{p^{\pm}} q^{\pm}_{out} + c \chi_{R}(|x|) - c
a^\pm_{\xi_k} \frac{x_k}{|x|} \chi'_{4R}(x) q_{in}^\pm.\]
As $\chi'_{4R}$ is supported where $H_{p^{\pm}}q^{\pm}_{out}$ is
strictly positive, by choosing $c$ sufficiently small, we have that
$H_{p^{\pm}}q^{\pm}$ is everywhere non-negative and positive on
$|x|\le 4R$.  Thus \eqref{Hpq} is satisfied.

Now we consider the truncation at frequency $\lambda$, which is also achieved 
at the level of $q^{\pm}$. Precisely, we set 
\[
q^{\pm}_{> \lambda} = e^{\sigma q^{\pm}} \chi( |\xi| e^{\sigma
  q^\pm}/\lambda)\chi(|\tau| e^{\sigma q^\pm}/\lambda)
\]
where $\chi$ is a nondecreasing smooth cutoff selecting the interval $[1,\infty]$.
We now can verify directly that 
\[
H_{p^{\pm}}q^{\pm}_{> \lambda} \gtrsim  \sigma c_k 2^{-k} q^{\pm}_{> \lambda}
,\qquad |x| \approx 2^k
\]
provided $\lambda$ is large enough.

We now combine the symbols constructed on the individual light cones into
\[q(t,x,\xi,\tau) = (\tau-a^-)q^+_{>\lambda} +
(\tau-a^+)q^-_{>\lambda},\]
and we compute
\[H_p q |_{\tau=a^\pm} = (a^+-a^-)^2 H_{p^{\pm}} q^{\pm}_{>\lambda} + (a_{\pm}-a_{\mp})q^{\pm}_{>\lambda}(-a^\mp_t +
a^{\pm}_t + a^{\pm}_{\xi_k}a^{\mp}_{x_k} - a^{\pm}_{x_k}a^{\mp}_{\xi_k}).\]
By choosing $\sigma$ sufficiently large, we have
\[H_p q |_{\tau=a^\pm} \ge \frac{1}{2} (a^+-a^-)^2 \sigma c_k 2^{-k}
q^{\pm}_{>\lambda},\quad |x|\approx 2^k.\]

We now consider the principal contribution of the skew-symmetric
portion of $P$, which has symbol
\[s_{skew} q= q\Im A_\alpha g^{\alpha 0} \tau + q \Im A_\alpha
g^{\alpha k}\xi_k.\]
Since 
\[s_{skew}q|_{\tau=a^{\pm}} = \pm (a^+-a^-) q^{\pm}_{>\lambda} (\Im
A_\alpha g^{\alpha 0} a^{\pm} + \Im A_\alpha
g^{\alpha k}\xi_k),\]
the choice of large $\sigma$ also allows us to absorb this, giving
\begin{equation}\label{on-cones}H_p q |_{\tau=a^\pm}  + s_{skew}q|_{\tau=a^{\pm}} \ge \frac{1}{4}
(a^+-a^-)^2 \sigma c_k 2^{-k} q^{\pm}_{>\lambda},\quad |x|\approx 2^k.
\end{equation}

At this point, we shall now modify by a multiple of $p$, which serves
as a Lagrangian correction, so as to
guarantee that it is everywhere nonnegative.  Indeed, we note that
\[H_p q + s_{skew}q = A\tau^2 + B\tau + C\]
where $A,B,C$ depend on $(t,x,\xi)$ and $(A(a^\pm)^2 + Ba^{\pm}+C)>0$.
We seek to choose $m(t,x,\xi)$ so that
\[A\tau^2+B\tau+C + m(\tau-a^+)(\tau-a^-)> 0.\]
It suffices to choose $m$ so that $A+m> 0$ and
\[g(m):=(B-m(a^++a^-))^2 - 4(A+m)(C+ma^+a^-)<0.\]
Taking the minimizing value for the quadratic $g$,
\[m=\frac{B(a^++a^-)+2C+2Aa^-a^+}{(a^+-a^-)^2},\]
gives
\[g(m)=-4(A(a^+)^2+Ba^++C)(A(a^-)^2+Ba^-+C)\]
and
\[A+m = \frac{(A(a^+)^2+Ba^++C)+(A(a^-)^2+Ba^-+C)}{(a^+-a^-)^2},\]
which clearly satisfy the desired conditions.  

It thus follows that
\[H_p q + s_{skew}q + mp \ge (A+m)\Bigl(\tau +
\frac{B-m(a^++a^-)}{2(A+m)}\Bigr)^2 - \frac{g(m)}{4(A+m)},\]
which by \eqref{on-cones} completes the proof.
\end{proof}

We now return to the quantum side and use the lemma to complete the
proof of \eqref{main-high}.  For that we compute
\begin{multline*}2\Im\int_0^T\int Pv \overline{(q^w-(i/2)m^w)}v \,dV
+(1/2)\Bigl\la \Bigl[m^w, \Bigl(\Im A_\alpha g^{\alpha\beta}D_\beta +
D_\alpha g^{\alpha\beta}\Im A_\beta\Bigr) \Bigr] v, v\Bigr\ra
\\-\Bigl\la i \Bigl[\Bigl(\Re A_\alpha g^{\alpha\beta}D_\beta +
D_\alpha g^{\alpha\beta}\Re A_\beta\Bigr), q^w\Bigr]v,v\Bigr>
\\- (1/2)\Bigl\la\Bigl\{ m^w \Bigl(\Re A_\alpha g^{\alpha\beta}D_\beta +
D_\alpha g^{\alpha\beta}\Re A_\beta\Bigr) + \Bigl(\Re A_\alpha g^{\alpha\beta}D_\beta +
D_\alpha g^{\alpha\beta}\Re A_\beta\Bigr) m^w\Bigr\} v, v\Bigr\ra
\\= \la i [\Box_g, q^w] v,v\ra + (1/2)\la (m^w\Box_g + \Box_g
m^w)v,v\ra
\\+ \frac{1}{i}\Bigl\la \Bigl\{q^w \Bigl(\Im A_\alpha g^{\alpha\beta}D_\beta +
D_\alpha g^{\alpha\beta}\Im A_\beta\Bigr) + \Bigl(\Im A_\alpha g^{\alpha\beta}D_\beta +
D_\alpha g^{\alpha\beta}\Im A_\beta\Bigr) q^w\Bigr\}v,  v\Bigr\ra.
\end{multline*}
Here $v = \beta(t-T+2) u$.  As $\|(1-\beta(t-T+2)) u\|_{LE^1}$ and the
error term $\la [P,\beta(t-T+2)]u, (q^w-(i/2)m^w)v\ra$ are
easily estimated using \eqref{local-en}, it suffices to establish
\eqref{main-high} for $v$.

The last three terms in the left side
  are lower order and are easily bounded by
$\lambda^{-1/2} \|v\|_{LE^1}$
using the frequency localization of $q, m$.
On the other hand, by the lemma, the principal symbol of
the last three terms on the right side is everywhere nonnegative and
bounded below by a multiple of $\mathbf{1}_{|\tau|,|\xi|\gtrsim \lambda}(\tau^2 + |\xi|^2)$ on $|x|\le 4R_0$.
Thus, by an application of G\r{a}rding's inequality, the right side is
bounded below by $\|\partial v_{\ge \lambda}\|^2_{LE_{<4R_0}}$ modulo errors that are
bounded by $\lambda^{-1/2} \|v\|_{LE^1}$, as desired.


\section{Carleman estimates and the medium frequency 
analysis}\label{s:med}

The medium frequency analysis for the local energy decay bounds is
based on Carleman type estimates. This emulates in the dynamical
setting the resolvent bounds used to prove the absence of embedded
eigenvalues/resonances.  For the latter, we refer the reader to
\cite{KT}. We first discuss the Carleman estimates and then show how
they apply for the local energy decay bounds.

\subsection{Carleman estimates for wave equations}

Here we consider two classes of Carleman estimates. One applies in the
small asymptotically flat regime and the other applies in a compact set.
The estimates we seek are of the form:
\begin{equation}
\|\omega_0 e^\varphi u \|_{L^2} + \|\omega_1 e^\varphi \nabla u \|_{L^2} 
\lesssim \| e^\varphi Pu\|_{L^2}
\end{equation}
for appropriate weights $\omega_0$ and $\omega_1$, with the key feature 
that the implicit constant is independent of the (parameters entering the) weight $\varphi$.
For our purposes it suffices to work with radial weights.

The first result is concerned with large $r$, so the asymptotic
behavior of $\varphi$ at infinity is important.  Here it is convenient
to express $\varphi$ in the form
\[
\varphi = \varphi(s), \qquad s = \log r
\]
and ask that it is  slowly varying as function of $s$ on the unit
scale.  To start, we consider the 
following set-up:

\begin{prop}\label{p:Carl}
  Suppose that $P$ is asymptotically flat. Let $\varphi$ be a convex weight
  which satisfies
\begin{equation}\label{varphi}
\lambda\lesssim \varphi'(s), \quad \lambda \lesssim
  \varphi''(s)\le \frac{1}{2}\varphi'(s),\quad |\varphi'''(s)|+|\varphi^{(iv)}| 
\ll \varphi',\quad \lambda \gg 1. 
\end{equation}
Then for $u \in LE^1_0$ supported in $\{r > R_0\}$ we have the uniform estimate
\begin{equation}\label{e:carl}
\|  r^{-1} (1+\varphi'')^\frac12 e^\varphi (r^{-1} (1+\varphi') u, \nabla u)  \|_{L^2} 
+ \| r^{-1} (1+\varphi')^\frac12 e^\varphi  \partial_t u \|_{L^2} 
\lesssim \| e^\varphi Pu\|_{L^2}.
\end{equation}
\end{prop}

A variation on this theme is needed in order to deal with the fact that we 
want to bend the weight in such a way so that it is constant near infinity.
That would make the weight nonconvex there, though we will stay with the 
monotonicity condition. The lack of convexity yields some errors in the estimate;
the point is that these errors will be lower order, and absorbed into the main term
provided that there is a bound from below on the allowed time frequencies.
To simplify matters, in the next result we harmlessly assume that this transition occurs
in a single dyadic region.

\begin{prop}\label{p:Carl-cut}
Suppose that $P$ is asymptotically flat.  Let 
$R > R_0$, and let $\varphi$ be an increasing weight which is as in \eqref{varphi}
for $s < \log R$ and is constant for $s > \log R +1$. 
Then for $u \in LE^1_0$ supported in $\{r > R_0\}$ we have the uniform estimate
\begin{equation}\label{e:carl-cut}
\begin{split}
\| r^{-1} (1+\varphi''_+)^\frac12 e^{\varphi}  &(\nabla u, r^{-1}(1+\varphi') u)\|_{L^2_{<R}}+ 
\| r^{-1} (1+\varphi')^\frac12 e^{\varphi} \partial_t u\|_{L^2_{<R}}
+ R^{-\frac12}
\| e^{\varphi} u\|_{LE^1_{>R}} \\ & \ \lesssim    \| e^{\varphi} P u \|_{L^2_{<R}}
+ R^{-\frac12} \| e^{\varphi} P u\|_{LE^*_{>R}} 
+  R^{-2} \| (1+\varphi')^{\frac{3}{2}} e^\varphi u\|_{L^2_R} . 
\end{split}
\end{equation}
\end{prop}

Comparing the weight on the added error term with the weight on the
$\partial_t u$ term on the left, we see that the error term is
negligible provided that the lower bound $\tau_0$ on the time frequencies
satisfies
\begin{equation}
\tau_0 \gg \varphi'(\log R) R^{-1}.
\end{equation}
This is quite similar to the proof in \cite{KT} of the result on the
absence of embedded eigenvalues, minus the initial part which uses
the symmetry of the elliptic operator to obtain the initial $LE^1_0$ decay at infinity
(which is directly assumed here).

\begin{proof}[Proof of Propositions~\ref{p:Carl}, \ref{p:Carl-cut}]
These are standard Carleman estimates. In both of them the magnetic 
and scalar potential terms are perturbative, so we can simply take
\[
P = D_\alpha g^{\alpha\beta} D_\beta.
\]
Conjugating out the exponential weight,
we rewrite \eqref{e:carl} in the form 
\begin{equation}\label{e:carl1}
\|  r^{-1} (1+\varphi'')^\frac12 (r^{-1} (1+\varphi') v, \nabla v)  \|_{L^2} 
+ \| r^{-1} (1+\varphi')^\frac12  \partial_t v \|_{L^2} 
\lesssim \|  P_\varphi v\|_{L^2}
\end{equation}
where $v = e^\varphi u$ and $P_\varphi$ represents the conjugated operator
\[
P_\varphi(x,D) = e^{\varphi} P(x,D) e^{-\varphi}=
P(x,D+i\nabla \varphi).
\]
This is split into a self-adjoint and a skew-adjoint part, which have the form
\[
P_\varphi^r =D_\alpha g^{\alpha\beta} D_\beta - \varphi_i g^{ij} \varphi_j , \qquad 
P_\varphi^i = iD_\alpha  g^{\alpha j} \varphi_j +  i\varphi_j g^{j\alpha} D_\alpha 
\]
where $\varphi_j = \partial_j \varphi$.
Then we have 
\[
\|P_\varphi v\|_{L^2}^2 = \|  P_\varphi^r v\|_{L^2}^2 + \|  P_\varphi^i v\|_{L^2}^2
+ \la [P_\varphi^r, P_\varphi^i] v,v\ra;
\]
therefore our estimate is essentially a statement about the positivity of the 
commutator on the characteristic set of $P_\varphi$, i.e. the pseudoconvexity 
condition.

We begin by checking the pseudoconvexity condition at the symbol level.
With $p(x,\xi) = g^{\alpha \beta} \xi_\alpha \xi_\beta$ we have 
\[
p_\varphi =g^{\alpha \beta} \xi_\alpha \xi_\beta  - g^{ij} \partial_i \varphi  \partial_j \varphi
+ 2 i g^{\alpha j} \xi_\alpha \partial_j \varphi,
\]
so the principal symbol of the commutator is
\[
\{\Re p_\varphi, \Im p_\varphi\} = 2 \{ g^{\alpha \beta} \xi_\alpha \xi_\beta- g^{ij} \partial_i \varphi  \partial_j \varphi,
g^{\alpha j} \xi_\alpha \partial_j \varphi\}
\]
All terms where $g$ is differentiated are perturbative,
so we write them as $err$. Then we have 
\[
\{\Re p_\varphi, \Im p_\varphi\} =  4 \xi_\alpha g^{\alpha j} \varphi_{jk} g^{k\beta} \xi_\beta + 4 \varphi_i g^{ij} 
\varphi_{jk} g^{kl} \varphi_l + err.
 \]
 Replacing $g$ on the right with the Minkowski metric $m$  is also
 perturbative, so it suffices to do a complete computation in the
 Minkowski case. There we have
\[
\Re p_\varphi = |\xi'|^2 - \xi_0^2 - r^{-2} (\varphi')^2, \qquad \Im p_\varphi = 2 r^{-1} \frac{x}{r} \cdot \xi' \varphi';   
\]
therefore
\[
\begin{split}
\{\Re p_\varphi, \Im p_\varphi\}  = & \ 2 \{ |\xi'|^2 - r^{-2} (\varphi')^2,  r^{-2} \varphi' \ x \cdot \xi'\} 
\\
= & \ 4 \varphi''  r^{-4} ( (x \cdot \xi')^2 + (\varphi')^2) + 4 r^{-2} \varphi' ( |\xi'|^2 - r^{-2} (\varphi')^2) - 8 r^{-4}
\varphi'(x \cdot \xi')^2  
\\
= & \ 4 \varphi''  r^{-4} ( (x \cdot \xi')^2 + (\varphi')^2)+ 4 r^{-2} \varphi' \xi_0^2 
 + 4 r^{-2} \varphi' \Re p_\varphi   - 2 (\varphi')^{-1}  (\Im p_\varphi)^2.  
\end{split}
\]
Then, choosing a function $b$ with the property that $b \approx \frac12  r^{-2} \varphi''$,
 we have a positivity  relation of the form 
\[
\{\Re p_\varphi, \Im p_\varphi\}  + (b-4r^{-2}\varphi') \Re p_\varphi + 2 (\varphi')^{-1} (\Im p_\varphi)^2
\gtrsim \varphi'' r^{-2} (|\xi'|^2 + r^{-2}(\varphi')^2) +  r^{-2} \varphi' \xi_0^2 
\]
where $g$ and $m$ are still perturbatively interchangeable. 
This leads to an inequality of the form
\begin{equation}\label{e:carl2}
LHS\eqref{e:carl1} \lesssim \la [P_\varphi^r, P_\varphi^i] v,v\ra_{L^2}+
\| (\varphi')^{-\frac12} P_\varphi^i v\|_{L^2}^2 + \la (b-4r^{-2}\varphi') v,  P^r_\varphi v\ra
\end{equation}
where all lower order terms in the commutator and in the integration by parts in the 
last term on the right are negligible.
Since we have 
\[
|\varphi'| \gg 1, \qquad | b- 4 r^{-2} \varphi'|^2 \ll \varphi'' r^{-4} (\varphi')^2,
\]
this implies \eqref{e:carl1}.

We emphasize again the advantage of working with \eqref{e:carl2} and
the preceding symbol bound, which is that they are stable with respect to
small asymptotically flat perturbations of $g$.  By contrast, such a substitution 
cannot be done directly in  \eqref{e:carl1}. 

\bigskip

Similarly, the bound \eqref{e:carl-cut} can be written in terms of $v$ and the conjugated 
operator $P_\varphi$,
\begin{equation}\label{e:carl-cut1}
\begin{split}
\| r^{-1} (1+\varphi''_+)^\frac12  (\nabla v, r^{-1}(1+\varphi') v)\|_{L^2_{<R}}+ &
\| r^{-1} (1+\varphi')^\frac12  \partial_t v\|_{L^2_{<R}}
+ R^{-\frac12}
\| v \|_{LE^1_{>R}} \\ & \ \lesssim    \|  P_\varphi v \|_{L^2_{<R}}
+ R^{-\frac12} \| P_\varphi v\|_{LE^*_{>R}} 
+ R^{-2} \| (1+\varphi')^{\frac{3}{2}}  v\|_{L^2_R} . 
\end{split}
\end{equation}
To prove this bound we divide the analysis into three overlapping regions:

a) The region $\{ r < R/2\}$, where $\varphi''$ is positive and \eqref{e:carl1}
applies.

b) The region $\{ R/4 < r < 2R\}$, where $\varphi' \gg 1$  but $\varphi''$ is in part negative,
with a bound $\varphi'' > - \varphi'/2$.

c) The region $\{ r > R\}$, where $\varphi$ is constant and the bound \eqref{high_Mourre} applies.

To glue the bounds in the three regions together one simply uses cutoff functions
adapted to the first two regions. The third region already has the cutoff built in the 
estimate \eqref{high_Mourre}.

Thus it remains to consider the region (b). But there the same reasoning as above applies,
with the only difference that the function $b$ is now chosen so that 
$b \approx r^{-2} \varphi'$. 
\end{proof}

The second class of Carleman estimates applies in a compact set. Here
our goal is to use the convexity of the weight in order to produce a
bound which has small high frequency errors, and thus applies for time
frequencies in a large but finite frequency range. This is done in order to
have a result which does not use the nontrapping assumption. The result below is 
a local result; therefore the choice of coordinates is not essential. For simplicity
we choose our weight as $\varphi= \varphi(r)$, though one can easily replace 
$r$ by $s$ away from zero.

\begin{prop}\label{p:Carl-in}
  Suppose $P$ is hyperbolic, so that the surfaces $t = const$ are
  uniformly space-like and $\partial_t$ is uniformly time-like.  Let
  $\varphi$ be a radial weight which satisfies
\begin{equation}\label{phinear0}
\varphi'(0) = 0, \qquad \varphi'' \approx \lambda + \sigma \varphi' , \quad 
| \varphi'''| \lesssim \sigma^2 \varphi', \quad | \varphi^{(iv)}| \lesssim \sigma^3 \varphi', 
\qquad  \lambda,
 \sigma  \gg 1. 
\end{equation}
Then we have the following estimate:
\begin{equation}\label{Carl-in}
  \| (\varphi'/r)^\frac12 e^\varphi \partial u\|_{L^2}+\|(\varphi'')^\frac12    e^\varphi  \varphi' u\|_{L^2} 
+ \| r^{-1}\varphi' e^\varphi u\|_{L^2} 
\lesssim \|e^\varphi Pu \|_{L^2} + \|(\varphi'{/\la r\ra})^\frac12
e^\varphi \partial_t u\|_{L^2_{\gtrsim 1}}.
\end{equation}
\end{prop}

\begin{proof}
The reasoning here is similar to the previous proof.  We first note
that the hypotheses \eqref{phinear0} insure that $\varphi$ is
increasing.  Next, we observe that the contribution of the 
potential $V$ and the magnetic potential $A$ in the above estimate is
again perturbative since \eqref{phinear0} gives that $\varphi'/r
\gtrsim \lambda$, so for the purpose of this proof we set $P = D_{\alpha} g^{\alpha\beta} D_{\beta}$.
After conjugation, the estimate \eqref{Carl-in} becomes 
\begin{equation}\label{Carl-in1}
 \|(\varphi'')^\frac12   \varphi' v\|_{L^2} + \| (\varphi'/r)^\frac12
 \partial v\|_{L^2} +\|r^{-1}\varphi' v\|_{L^2}
\lesssim \| P_\varphi v \|_{L^2} + \| (\varphi'{/\la r
  \ra})^\frac12  \partial_t v\|_{L^2_{\gtrsim 1}}.
\end{equation}
At the symbol level we compute 
\[
\begin{split}
\frac12 \{\Re p_\varphi, \Im p_\varphi\} = & \  \{ g^{\alpha \beta} \xi_\alpha \xi_\beta- g^{ij} \partial_i \varphi  \partial_j \varphi,
g^{\alpha j} \xi_\alpha \partial_j \varphi\}
\\ = & \  \partial_{i} \partial_j \varphi g^{i\alpha} g^{j\beta} \xi_\alpha \xi_\beta + 
 \partial_{i} \partial_j \varphi g^{ik} g^{jl} \partial_k
 \varphi  \partial_l \varphi + O( \varphi'{/\la r\ra}) (|\xi|^2
+ (\varphi')^2) 
\\ = & \ \varphi'' ( (\partial_i r g^{i\alpha} \xi_\alpha)^2 + |\nabla r|^4_g (\varphi')^2) 
+\varphi' \partial_i \partial_j r \  g^{i\alpha} g^{j\beta} \xi_\alpha
\xi_\beta\\&\qquad\qquad\qquad\qquad\qquad\qquad\qquad{+\varphi' \partial_i\partial_j r g^{ik}g^{jl}\partial_k
  \varphi \partial_l\varphi}
+ O( \varphi'{/\la r\ra})  (|\xi|^2
+ (\varphi')^2) 
\\ \gtrsim & \ \varphi'' (\varphi')^2 + r^{-1} \varphi' |g^{i\alpha}
\xi_\alpha|^2+ O( \varphi'{/\la r\ra})  (|\xi|^2
+ (\varphi')^2){.}
\end{split}
\]

Now we split the analysis into two regions, which can be easily assembled together
using an appropriate partition of unity.

{\em 1. The inner region $r \ll 1$.}  
Since the surfaces $t= const$ are uniformly space-like while
$\partial_t$ is time-like, it follows that $g^{00} < 0$ while the
quadratic form $g^{ij} \xi_i \xi_j$ is positive definite.  
Therefore we obtain an inequality of the form
\[
|\xi|^2 \lesssim - \Re p_\varphi + c_0 \sum_i  |g^{i\alpha} \xi_\alpha|^2
\]
for a large universal constant $c_0$. Thus for $r \ll 1$ we obtain the bound 
\[
\varphi'' (\varphi')^2 +  r^{-1} \varphi' |\xi|^2 \lesssim \{\Re p_\varphi, \Im p_\varphi\} 
- c_1  r^{-1} \varphi' \Re p_{\varphi}
\]
with a small universal $c_1$. Here no $\xi_0^2$ error term is needed
on the right; this can be viewed as a consequence of the fact that no
trapping can occur very close to $x = 0$.
This translates into a bound of the form 
\[
 \|(\varphi'')^\frac12   \varphi' v\|_{L^2}^{2} + \| (\varphi'/r)^\frac12 \partial v\|_{L^2}^{{2}}
\lesssim  \la [P_\varphi^r,P_\varphi^i] v ,v \ra_{L^2} 
- c_1 \la P^{{r}}_\varphi v,  r^{-1} \varphi' v \ra_{L^2} 
\]
where we remark that, by the uncertainty principle, 
 the norm $\| r^{-1}\varphi' v\|_{L^2}^2$ can also be freely
added on the left in order to compensate for the degeneracy of $\varphi'$ at zero.
Hence \eqref{Carl-in1} easily follows for $v$ supported in $\{r \ll 1\}$. 

{\em 2. The outer region $r \gtrsim 1$.}  
In this region the leading term $ r^{-1} \varphi' |g^{i\alpha} \xi_\alpha|^2$ and 
the error term $O( \varphi'{/\la r\ra})  |\xi|^2$ are indistinguishable, so we 
simply write
\[
\{\Re p_\varphi, \Im p_\varphi\}  \gtrsim  \ \varphi'' (\varphi')^2 +
O( \varphi'{/\la r\ra})  |\xi|^2.
\]

Since the surfaces $t= const$ are uniformly space-like while
$\partial_t$ is time-like, it follows that $g^{00} < 0$ while the
quadratic form $g^{ij} \xi_i \xi_j$ is positive definite. Thus the
following must hold:
\[
|\xi|^2  \lesssim  \Re p_\phi + c_0 ( |\xi_0|^2 + |\varphi'|^2)
\]
for a large universal constant $c_0$. 
We are then led to a bound of the form
\[
\varphi'' (\varphi')^2 +  \frac{\varphi'}{{r}} |\xi|^2 \lesssim \{\Re p_\varphi, \Im p_\varphi\} 
+ c_1  \frac{\varphi'}{{r}} \Re p_\phi + c_2    \frac{\varphi'}{{r}} |\xi_0|^2
\]
for large universal constants $1 \ll c_1 \ll c_2$.
This translates into the bound
\[
\|(\varphi'')^\frac12   \varphi' v\|^2_{L^2} + \| (\varphi'{/r})^\frac12 {\partial} v\|^2_{L^2} 
\lesssim 
 \la [P_\varphi^r,P_\varphi^i] v ,v \ra_{L^2} + 
c_1 \la P_\varphi^{{r}} v,  {(}\varphi'{/r)} v \ra_{L^2} {+}
c_2 \| (\varphi'{/r})^\frac12  \partial_t v\|_{L^2}^{{2}}
\]
where all the lower order terms occurring in the integration by parts are negligible
due to the conditions ${\sigma},\lambda \gg 0$. Thus the bound \eqref{Carl-in1} easily follows
 for functions $v$ supported in the region $r \gtrsim 1$. 

\end{proof}

\subsection{The medium frequency estimate}
\label{s:med2}
Our main result here is as follows:

\begin{thm}[Medium frequency LE bound]\label{med_LE_thm}
  Let $P$ be an AF operator which satisfies the ellipticity condition
  \eqref{zero_ell}.  Then for each $\delta > 0$ there exists a bounded
  increasing radial weight $\varphi= \varphi(\log (1+r))$ so that the following bound holds
for all $u \in LE^1_0$ with $Pu \in LE^*$:
\begin{equation} \begin{split}
\lp{(1+\varphi''_+)^\frac12 e^\varphi \nabla u}{LE} 
+ \lp{\la r \ra^{-1} (1+\varphi''_+)^\frac12 (1+\varphi') e^\varphi u}{LE} 
+ \lp{(1+ \varphi')^\frac12 e^\varphi \partial_t u}{LE} 
\ \lesssim  \\ 
\delta( \|(1+ \varphi')^\frac12 e^\varphi u\|_{LE} 
+ \|\la r \ra^{-1}  (1+ \varphi''_+)^\frac12 (1+\varphi') e^\varphi \partial_t u\|_{LE}) + 
 \lp{e^\varphi P u}{LE^*}  \label{M_LE}
\end{split}
\end{equation}
with an implicit constant which is independent of $\delta$.
\end{thm}

  This estimate is a Carleman estimate, which provides a quantitative
  version of the absence of positive eigenvalues embedded in the
  continuous spectrum of the operator $\Box_g+V$. We carefully observe 
that this bound does not preclude the embedded resonances; this would require
in addition some symmetry condition on $P$, which we do not assume here.
The role of the parameter $\delta$ is to allow for an arbitrarily large range 
of time frequencies in the medium frequency bound.

\begin{proof}
This is done in a relatively straightforward manner using the Carleman estimates
in Propositions~\ref{p:Carl-cut}, \ref{p:Carl-in}. Precisely, we will work with a weight 
$\varphi$ which is consistent with Proposition~\ref{p:Carl-in} for $r < 2R_0$, and with 
Proposition~\ref{p:Carl-cut} for $r > 2R_0$. For the former we fix the parameter $\sigma$
to be a sufficiently large universal constant, and retain the freedom to choose $\lambda$
sufficiently large. The two $\lambda$'s  must be comparable,
so our weight will satisfy 
\[
\varphi'(s) \approx \min\{ \lambda r,  \lambda \log(r+10)\},
\quad \varphi''(s) \approx \lambda 
  \qquad 1 \lesssim s \leq \log R{.}
\]
The choice here necessitates that
$e^{-s}\varphi'(s)\to 0$ as $s\to \infty$, for which we have given one
of many options that are easily compatible
with \eqref{varphi}.
Combining the two Carleman estimates (this is easily done using cutoff functions
and absorbing the errors) we obtain the bound 

\begin{equation*}
\begin{split}
& \|  (1+\varphi''_+)^\frac12 e^{\varphi}  (\nabla u, \la r\ra^{-1} (1+\varphi') u)\|_{{LE}_{<R}}+ 
\|  (1+\varphi')^\frac12 e^{\varphi} \partial_t u\|_{{LE}_{<R}}+ 
\| e^{\varphi} u\|_{LE^1_{>R}} \\ &\qquad \ \lesssim    \| e^{\varphi} {P} u \|_{LE^*_{<R}}
+ \| e^{\varphi} {P} u\|_{LE^*_{>R}}  + 
\lambda^{1/2}\| e^{\varphi} \partial_t u\|_{L^2_{{1\lesssim \cdot <2R_0}}}
+ R^{-{3/2}} \| (1+\varphi')^\frac{3}{2} e^\varphi u\|_{L^2_R}  
\end{split}
\end{equation*}
where the last two terms on the right are the error terms in the two
Propositions.  
This implies the bound \eqref{M_LE} provided that we choose $\lambda$ large enough for 
the first error term and $R$ large enough for the second (using that
$R^{-1}\varphi'(\log R)\to 0$).

If we want to be more accurate, suppose that we restrict the above
estimate to functions with time frequencies in the band $\tau_1 \leq
|\xi_0| \leq \tau_2$. Then we can absorb the first error provided that
$\tau_2 \ll \lambda$, and the second error provided that
$\varphi'(\log R)
R^{-1} \ll \tau_1$.
\end{proof}


\section{Low frequency analysis}
\label{s:low}

This section is devoted to a more detailed analysis of the low frequency regime.
One of the main goals will be to prove the following result, which shows that 
the zero nonresonance condition \eqref{zero-res2} provides full information
in a neighbourhood of time frequency zero.

\begin{thm}[Low frequency LE bound]\label{low_LE_thm}
Let $P$ be an AF operator which satisfies the ellipticity condition \eqref{zero_ell}
and the zero nonresonance condition \eqref{zero-res2} uniformly in time.
Then a modified form of stationary local energy decay holds,
\begin{equation}
		\lp{u}{LE^1} \ \lesssim   \|\partial_t u\|_{LE^1_{comp}} +   \lp{P u}{LE^*} \label{S_LE}
\end{equation}
for all compactly supported $u$.
\end{thm}
Comparing this with the earlier stationary local energy decay
condition \eqref{LE-stat}, we remark that the above bound is stronger
in that only a local norm of $\partial_t u$ is used on the right, but
also weaker in that derivatives of $\partial_t u$ are also used.
However, the latter fact is harmless, as the above bound is only used
at very low time frequencies.  This theorem yields directly just one of the
relations in Proposition~\ref{p-zero}, precisely the one we need for
the proof of the two point local energy decay bound in
Theorem~\ref{2pt_thm}.  However, the results we establish along the
way also suffice for the full proof of Proposition~\ref{p-zero}. We
discuss this at the end of the section.

The first step is to show that, in the context of the present work, the
zero resolvent bound \eqref{zero-res2} is stable with respect to small
perturbations of the operator $P$ and can be deduced from a more
standard non-quantitative from. It also holds uniformly when $V$ is
sufficiently non-negative. This is all contained in:

\begin{lem}\label{zspec_assm_lem}
  Let $P$ be stationary, asymptotically flat, so that the time slices $t=const$
  are uniformly space-like and $\partial_t$ is uniformly time-like.
  Then one has the following properties related to the zero resolvent
  bound \eqref{zero-res2}:
\begin{enumerate}[i)]
		\item (Stability) \label{zspec_stab} The  zero resolvent
                bound \eqref{zero-res2} is stable
		with respect to sufficiently small stationary AF perturbations of $g,A,V$. 
                \item (Non-quantitative version) \label{zspec_nonquant} Suppose that 
		there exists no distributional solutions to $P_0 u_0=0$ with
		$u_0\in H^1(\mathbb{R}^3)$. Then there exists a $K_0=K_0(t_0)$ such that
		estimate \eqref{zero-res2} holds.
              \item (Bound for sufficiently non-negative
                $V$) \label{zspec_nonneg} Suppose that $P$ is
                symmetric and that:
		\begin{equation}
	-\int_{\mathbb{R}^3}   \cal{Q}(u,u)  dx \ \leq \ \lambda_0
				\int_{\mathbb{R}^3} |D_Au|^2 dx \ , \qquad
				\hbox{where\ \ }
				\cal{Q}(u,u)= 2 A_0g^{0i} \Re(u \overline{(D_i+A_i) u}) + (g^{00}A_0^2+V) |u|^2
				 \label{W_cond}			
		\end{equation}
               and $|D_A u|^2 = g^{jk}(\partial_j + iA_j) u \cdot
               \overline{(\partial_k +i A_k)u}$.
		Then, if $\lambda_0<1$, the estimate \eqref{zero-res2}
                holds with $K_0=O(( 1-\lambda_0)^{-1})$.  In
                particular one has \eqref{zero-res2} for symmetric AF magnetic wave
                equations $\Box_{g,A}+W$ if $A_0=0$ and $W\geq 0$.
\end{enumerate}
\end{lem}

\begin{rem}
  Note that in case \ref{zspec_nonquant}) above we do not gain
  quantitative control on $K_0$.  Therefore the assumption that there
  exist no $H^1$ solutions to $Q u_0=0$ becomes less useful in the
  case when $g,V$ are non-stationary, and one must instead rely
  directly on \eqref{zero-res2}.
\end{rem}

\begin{proof}[Proof of Lemma \ref{zspec_assm_lem}]
The three parts are proved separately.

i) Using  the Hardy estimate 
\begin{equation}
 		\lp{|x|^{-1}\phi}{L^2(\mathbb{R}^3)} \ \leq \
                2\lp{\nabla\phi }{L^2(\mathbb{R}^3)} , \label{s_Hardy}
\end{equation}
and its dual we can enhance \eqref{zero-res2} to:
\begin{equation}
  \lp{ (  \nabla u,  \la x\ra^{-1} u)   }{L^{2}(\mathbb{R}^3)} \ 
  \lesssim \ K_0\lp{ P_0(x,D)u}{\dot H^{-1}+\la r \ra^{-1} L^2(\mathbb{R}^3)} \ .
\label{zero-res2'}	
\end{equation}
Now its  stability with respect to small AF perturbations is 
straightforward.
\medskip

ii) The proof of  follows from a standard compactness argument.
First assume there exists a sequence $u_n \in \dot H^1$
normalized so $\lp{u_n}{\dot H^1}=1$ and $\lp{P_0 u_n}{\dot H^{-1}} \to 0$.
Then WLOG we can assume $u_n\rightharpoonup u_0$ where
$\lp{u_0}{\dot H^1}\leq 1$ and $P_0 u_0=0$ in the sense of distributions. 
By the ellipticity of $P_0$, this convergence must be strong on compact sets.

There are now two scenarios: either $u_0\neq 0$ or $u_n\to 0$
strongly in $H^1(|x|\leq R)$ for all $R>0$. In the first case we have
reached a contradiction.  The latter case is ruled out since for $\{r
> R_0\}$ the operator $P_0$ is a small perturbation of $\Delta$;
therefore we have the truncated bound
\[
\| u_n\|_{ \dot H^1(\{r > R_0\})} \lesssim \| u_n\|_{  H^1(\{r \approx R_0\})}+ \| P_0 u_n\|_{\dot H^{-1}}.
\]

\medskip

iii) Integrating by parts we obtain:
\begin{equation}
		\int_{\mathbb{R}^3}  \big( |D_Au |^2  + \mathcal{Q}(u,u) \big) \, dx \ = \ 
		\int_{\mathbb{R}^3} P_0 u \cdot \overline{u } \, dx \ . \label{GI}
\end{equation}
Thus Kato's inequality $ \big| \nabla_x |u|  \big|^2\lesssim |D_Au|^2 $  and the condition
$\lambda_0<1$ on line \eqref{W_cond} give:
\begin{equation}
		\lp{| u| }{\dot{H}^1(\mathbb{R}^3)}^2 \ \lesssim \ (1-\lambda_0)^{-1}
		\lp{P_0 u }{\dot H^{-1}(\mathbb{R}^3)} \lp{ u }{\dot H^1(\mathbb{R}^3)} \ . \notag
\end{equation}
On the other hand going back to \eqref{GI} and instead
using Hardy's estimate \eqref{s_Hardy} to bound the
magnetic terms gives
\begin{equation}
	\int_{\mathbb{R}^3}   |D u |^2   \, dx \ \lesssim \| |u|\|_{\dot H^1}^2+ 
 \| u\|_{\dot H^1}\| |u|\|_{\dot H^1}+ 		
\lp{P_0 u }{\dot H^{-1}(\mathbb{R}^3)} \lp{ u }{\dot H^1(\mathbb{R}^3)}.
 \label{GI+}
\end{equation}
The desired conclusion follows by combining the last two lines.
\end{proof}

 Next we establish the following simple lemma for the Euclidean Laplacian
$\Delta=\sum\partial_k^2$, which shows that it satisfies a family of weighted
local energy type bounds:

\begin{lem}
For $s=0,1,2$  one has the  estimates
\begin{equation}
\begin{split}
			 \Delta^{-1}:& \  \ell^1 L^{2,\frac{1}{2}} \ \longrightarrow \ \ell^\infty L^{2,-\frac{3}{2}},	\\
\nabla^s	 \Delta^{-1}: & \ \ell^1 L^{2,\frac{1}{2}} \ \longrightarrow \ \ell^1 L^{2,-\frac{3}{2}+s}, \qquad s= 1,2,
\\
 \qquad 	\nabla^s \Delta^{-1}: & \ \ell^1 L^{2,\frac{3}{2}} \ \longrightarrow \ \ell^\infty L^{2,-\frac{1}{2}+s},
\qquad s=0,1,2.
		  \label{gl_lap'}
\end{split}
\end{equation}
\end{lem}
 
\begin{proof}
  By scaling it suffices to solve the equation $\Delta u = f$ with $f
  \in L^2$ supported in $\{r \approx 1\}$.  By elliptic regularity we
  have $u \in H^2(\{r \approx 1\})$.  Inside a small ball centered at
  zero we have uniform bounds for $u$ and its derivatives, which
  suffice for all bounds above.  Outside a large ball centered at zero
  we have pointwise decay estimates inherited from the fundamental
  solution of $\Delta$.  Precisely, $\nabla^s u$ decays like
  $r^{-1-s}$. This again suffices for all bounds in the lemma.
\end{proof}

The following step is to explore some equivalent formulations of the zero
resolvent estimate \eqref{zero-res2}, where different weights are used at infinity
or where the bound is restricted to a large ball.

\begin{lem}[Zero frequency LE bound] \label{p:other-LE0}
Assume that $P_0$ is asymptotically flat and that the 
zero resolvent  estimate \eqref{zero-res2} holds. 
Then

a)   For each $ u \in \LE^1_0$ we have 
\begin{equation}
		\lp{\la r \ra  u }{ \LE^1} 
		\  \lesssim K_0 \lp{\la r \ra P_0  u}{\LE^*} \ . \label{zero_LE-low}
\end{equation}

b) For each $ u \in \LE^1_0$  we have
\begin{equation}
		K_0^{-1} \lp{\la r \ra  u }{ \LE^1_{<K_0}} + \lp{ u }{\LE^1_{>K_0}}
+  \lp{ \la r \ra^{-1} \nabla  u }{\LE^*_{>K_0}}
		\  \lesssim \lp{P_0  u}{\LE^*} \ . \label{zero_LE}
\end{equation}

c)  For each $ u \in \LE^1_{loc}$  and $R_1 \gg K_0, R_0$ we have
\begin{equation}\label{zero_LE-trunc}
	K_0^{-1} \lp{\la r \ra  u }{ \LE^1_{<K_0}} + \lp{ u
        }{\LE^1_{R_1 > \cdot > K_0}}  +  \lp{ \la r \ra^{-1} \nabla  u }{\LE^*_{R_1 > \cdot >K_0}}
		\  \lesssim  \lp{P_0  u}{\LE^*_{<R_1}} +  \| r^{-1} \nabla (r  u)\|_{\LE_{R_1}} .
\end{equation}
\end{lem}

In the estimate in part (c), the last term on the right plays the role of a boundary term
where $\{ r \approx R_1\}$. Its form is based on the fundamental solution $r^{-1}$ for $\Delta$,
which captures well the behavior at infinity of $ u$.

\begin{proof}[Proof of Proposition~\ref{p:other-LE0}]
The proof of the  proposition is perturbative off the Laplacian,
and the bound \eqref{zero-res2} is used as a black box in a compact
set.  We begin with the bounds \eqref{zero_LE-low} and
  \eqref{zero_LE}.  The above lemma shows that the two bounds hold with
  $K_0 = 1$ if $P_0 = \Delta$.  Further, using the $s=2$ cases of
    \eqref{gl_lap'} and absorbing the errors
  directly, they hold for small asymptotically flat perturbations of
  $\Delta$.  To use this fact, we consider a small perturbation
  $\tilde P_0$ of $\Delta$ which agrees with $P_0$ for $r > R_0$.
Then the function $\tilde P_0^{-1} P_0  u$ satisfies the bounds 
\eqref{zero_LE-low} and \eqref{zero_LE} with $K_0 = 1$.
It remains to estimate the difference
\[
\tilde  u =  u - \tilde P_0^{-1} P_0  u,
\]
which satisfies an equation with right hand side
\[
P_0 \tilde  u = -(P_0 -\tilde P_0)\tilde P_0^{-1} P_0  u.
\]
This is supported in $\{ r < R_0\}$ and satisfies
\[
\| P_0 \tilde  u\|_{L^2} \lesssim \lp{P_0  u}{\LE^*}.
\]
Then by \eqref{zero-res2} we obtain 
\[
\| \nabla \tilde  u\|_{L^2} \lesssim K_0 \lp{P_0  u}{\LE^*},
 \]
which suffices for $r \lesssim R_0$. For larger $r$ we truncate
to see that 
\[
\| \tilde P_0 (\chi_{> R_0} \tilde  u)\|_{L^2} = \|  P_0 (\chi_{> R_0} \tilde  u)\|_{L^2}
 \lesssim K_0 \lp{P_0  u}{\LE^*}.
\]
Thus applying the bound \eqref{zero_LE-low} for $\tilde P_0$
we obtain
\[
	\lp{\la r \ra \tilde  u }{ \LE^1}  \lesssim K_0 \lp{P_0  u}{\LE^*},
\]
which suffices for both \eqref{zero_LE-low} and
  \eqref{zero_LE}.

\bigskip

We now turn our attention to the truncated bound \eqref{zero_LE-trunc}.
We shall apply the bound \eqref{zero_LE} to the function 
\[
 u_1 = \chi_{<R_1}  u + r^{-1}  \chi_{>R_1} (r u)_{R_1}
\]
where the latter factor $ (r u)_R$ stands for the average of $r u$ in the $R$
annulus. By the Poincar\'e inequality we have the gluing relation
\[
 \| \nabla(   u - r^{-1} (r  u)_R)\|_{L^2_R}  +
      R^{-1}         \|  u - r^{-1} (r  u)_R\|_{L^2_R} \lesssim  \| r^{-1} \nabla (r  u)\|_{L^2_R}.
\]
Then we compute
\[
P_0  u_1 =  \chi_{<R_1} P_0  u + [P_0,\chi_{<{R_1}}] (  u -  r^{-1}  (r u)_{R_1}) 
+ \chi_{>{R_1}} (r u)_{R_1} (P_0 - \Delta) r^{-1},
\]
and using the previous bound we estimate 
\[
\| P_0  u_1\|_{\LE^*} \lesssim  \|  \chi_{<{R_1}} P_0  u \|_{\LE^*}+  {R_1}^{-\frac12} \| r^{-1} \nabla (r  u)\|_{L^2_{R_1}}
+ {\bf c} {R_1}^{-1} |(r  u)_{R_1}|,
\]
while 
\begin{multline*}K_0^{-1}\|\la r\ra u_1\|_{\LE^1_{<K_0}} +
\|u_1\|_{\LE^1_{>K_0}}
+\|\la r\ra^{-1}\nabla u_1\|_{\LE^*_{>K_0}}
\\\approx K_0^{-1}\|\la r\ra u\|_{\LE^1_{<K_0}} +
\|u\|_{\LE^1_{R_1>\cdot>K_0}} + \|\la r\ra^{-1}\nabla
  u\|_{\LE^*_{R_1>\cdot>K_0}}+{R_1}^{-1}|(r u)_{R_1}|.
\end{multline*}
Thus, applying \eqref{zero_LE} and using ${\bf c} \ll 1$ we obtain the desired bound 
\eqref{zero_LE-trunc}. 
\end{proof}

\begin{proof}[Proof of Theorem \ref{low_LE_thm}]
Our main estimate \eqref{S_LE}  
will follow from a combination of the outer estimate \eqref{low_Mourre}
and the  zero  frequency LE bounds in Lemma~\ref{p:other-LE0}.
For that we choose parameters $R_0  \ll R_1$, as well as an intermediate 
 $R_0 < R_2 < R_1$. Applying 
 \eqref{low_Mourre} with $R = R_2$ we have 
\begin{equation}
	\lp{   u }{LE^1_{>R_2}}  
	\lesssim \ \lp{\partial  u}{LE_{R_2}}  + 
	\lp{P  u}{{LE}^*_{>R_2}}.
\end{equation}
On the other hand, integrating the bound \eqref{zero_LE-trunc} 
in time in $L^2$ gives
\begin{equation}\label{zero_LE-trunc-int}
	K_0^{-1} \lp{\la r \ra  u }{ LE^1_{<K_0}} + \lp{ u }{LE^1_{R_1 > \cdot > K_0}}
		\  \lesssim  \lp{P_0  u}{LE^*_{<R_1}} +  \|  u\|_{LE^1_{R_1}}.  
\end{equation}
Further, by the pigeonhole principle, we can choose $R_2$ so that 
\begin{equation}
\|\nabla  u\|_{LE_{R_2}} \lesssim  \lp{P_0  u}{LE^*_{<R_1}} + 
[\log(R_1/{K}_0)]^{-1}  \|  u\|_{LE^1_{R_1}}.    
\end{equation}
Combining the three we obtain
\[
K_0^{-1} \lp{\la r \ra  u }{ LE^1_{<K_0}} + \lp{ u }{LE^1_{> K_0}}
\lesssim 	\lp{P  u}{LE^*_{>R_2}} + 
\lp{P_0  u}{LE^*_{<R_1}} + \|\partial_t u\|_{LE_{R_2}}.
\]
Then the desired bound \eqref{S_LE}  is obtained by replacing $P_0$ above with $P$,
with errors involving only time derivatives of $ u$.
\end{proof}

\begin{proof}[Proof of Proposition~\ref{p-zero}]
  We begin with the observation that the stability property with
  respect to small AF perturbations of $P$ was proved based on
  property (i) 
in Lemma~\ref{zspec_assm_lem}. We now consider the equivalence of properties 
(a)-(e), as follows:

\begin{itemize}
\item $(a) \implies (c)$ was proved in Lemma~\ref{zspec_assm_lem} (ii).
\item $(c) \implies (b)$  was proved in Lemma~\ref{p:other-LE0} (b).
\item $(b) \implies (a)$ is trivial.
\item $(c) \implies (e)$ is the main result of Theorem~\ref{low_LE_thm}.
\item $(e) \implies (a)$ is proved below.
\item $(c) \implies  (d)$ repeats the proof of  Theorem~\ref{low_LE_thm} but for the 
resolvent equation.
\item $(d) \implies (b)$ is trivial.
\end{itemize} 

It remains to show $(e) \implies (a)$.  Suppose $P_0 u_0 = 0$ for some
$u_0 \in \LE^1_0$.  Then $u_{{0}}$ decays at infinity in an averaged sense, so
by elliptic estimates near infinity we must have $u_0 \in \dot H^1$. 

On the other hand, in \eqref{S_LE} one can add initial and final data
using the extension argument in Section~\ref{s:extend} to obtain the bound
\[
\lp{u}{LE^1} \ \lesssim   \|\partial_t u\|_{LE_{comp}} +   \lp{P u}{LE^*}
+\|\partial u(0)\|_{L^2} + \|\partial u(T)\|_{L^2}.
\]
Applying this to the time independent function $u_0$ we obtain
\[
T^\frac12 \| u_0\|_{\LE^1} \lesssim \| u_0\|_{\dot H^1},
\]
and the desired conclusion $u_0=0$ follows by letting $T \to \infty$.
\end{proof}

\section{Two point local energy decay 
and energy growth}
\label{s:2pt}

The goal of this section is to prove Theorems~\ref{2pt_thm} and \ref{exp_dich_thm}.
  
\subsection{An extension argument}
\label{s:extend}

 Our first goal is to reduce Theorem~\ref{2pt_thm}
to a simpler version of it, without boundary terms at the initial
and final time:

\begin{thm}[Unconditional  local energy decay]\label{comp_LE_thm}
Assume that $P$ is nontrapping, asymptotically flat,  and
satisfies the zero nonresonance condition \eqref{zero-res2}.
Assume in addition that $P$ is $\epsilon$-almost stationary,
where $\epsilon$ is small enough
\[
\epsilon \ll_{R_0,M_0,T_0,K_0} 1.
\]
Then for all $u$ in the Schwartz space $\mathcal{S}(\mathbb{R}^{3+1})$, we have
\begin{equation}
		\lp{u}{LE^1 } \ \lesssim \  \lp{Pu}{LE^* } \ .  \label{uncond_LE}
\end{equation}
\end{thm}

The advantage to not having initial and final boundary terms in
\eqref{uncond_LE} is that it is easier to frequency localize it with
respect to time.  

\begin{proof}[Proof that Theorem \ref{comp_LE_thm} 
implies Theorem \ref{2pt_thm}]

We begin with some straightforward simplifications. First, by direct energy estimates
we obtain the bound
\[
\| \partial u\|_{L^\infty L^2} \lesssim  \|\partial u(0)\|_{L^2} + \| u\|_{LE^1} + \|Pu\|_{L^1 L^2+ LE^*}.
\]
Thus it suffices to show that 
\begin{equation}\label{2pt-LEb}
\| u\|_{LE^1[0,T]} \lesssim  \|\partial u(0)\|_{L^2} + 
\|\partial u(T)\|_{L^2}  + \|Pu\|_{L^1 L^2+ LE^*}.
\end{equation}
Secondly, also by straightforward energy estimates and finite speed of
propagation, we can assume that $u$ is supported in $\{ r \lesssim T\}$.

Denote $Pu = f$. It suffices to construct a function $v$ that
matches the Cauchy data of $u$ at times $0$ and $T$ and satisfies the
bound
\begin{equation}
\| v\|_{LE^1} + \| Pv - f\|_{LE^*} \lesssim \|\partial u(0)\|_{L^2} + 
\|\partial u(T)\|_{L^2} + \| f \|_{L^1 L^2 + LE^*}.
\end{equation}
Then the desired estimate \eqref{2pt-LEb} is obtained by applying the
bound \eqref{uncond_LE} to $u-v$. 

To produce $v$, we use a partition of unity to localize the problem to three regions
for the support of $u[0]$, $u[T]$ and $f$.

\bigskip

{\bf The inner region $D_{in} = \{ r < 4R_0\}$.}  We split this region into
unit time intervals $D_{in}^j$ using a time partition of unity $1 =
\sum \chi_j^2$. In each of these time intervals we use local
theory and 
standard energy estimates to solve the equation $Pv^j = \chi^j(t) f $,
making sure that we match the initial data for $u$, respectively the
final data for $u$, in the first, respectively the last of these
intervals. Then we reassemble these localized solutions by setting $v
= \sum \chi^j v_j$.

\bigskip

{\bf The outer region $D_{out} = \{ r > 2 R_0\}$.} Here we will produce 
$v$ localized to $\{  r >  R_0\}$, so that the inner part of the operator $P$ 
is not involved at all. Then, without any restriction in generality we can assume
that $P$ is a small AF perturbation of $\Box$.

 By time reversal symmetry  we can assume that $u[T]=0$, and $f$ is supported in
$\{ t < 3T/4\}$. Then we use Theorem~\ref{t:small} to solve the problem 
\[
\tilde P w = f, \qquad w[0]= u[0],
\]
where $\tilde P$ is a small AF perturbation of $\Box$ that
  coincides with $P$ on $\{r>R_0\}$.
The function $w$ satisfies a good bound,
\[
\|w\|_{LE^1} + \|\partial w\|_{L^\infty L^2} \lesssim \|\partial u(0)\|_{L^2} + \|f\|_{L^1 L^2+LE^*},
\]
but it is not supported in $ \{r > R_0,\ t < T\} $. We remedy this with suitable cutoff functions,
and set 
\[
v = \chi_{> R_0}(r) \chi_{<T}(t) w.
\]
This still satisfies the $LE^1$ bound, but we need to estimate the truncation error,
\[
P v - f = [P,\chi_{> R_0}(r) ]  \chi_{<T}(t) w + \chi_{> R_0}(r) [P,\chi_{<T}(t)]w.
\]
The first term is supported in $\{ R_0 \leq r \leq
2R_0\}$ and is easily estimated in $LE^*$ 
in terms of the $LE^1$ norm of $w$. The second is supported in 
$\{ R_0 < r < CT, \ 3T/4 \leq t < T\}$ and satisfies the bound
\[
\| [P,\chi_{<T}(t)]w\|_{L^\infty L^2} \lesssim T^{-1} (\| \partial w\|_{L^\infty L^2} +
\|r^{-1} w\|_{L^\infty L^2}) \lesssim  T^{-1} \| \partial w\|_{L^\infty L^2},\]
where the Hardy inequality was used at the last step. This bound is easily converted
to an $LE^*$ bound.

\end{proof}

\subsection{ Time frequency localization}

Here we prove Theorem~\ref{comp_LE_thm} through three building blocks,
which deal with low, medium, respectively high frequencies.  Precisely, these 
three building blocks are the estimates \eqref{S_LE}, \eqref{M_LE} and \eqref{high_LE}.
The idea is  to apply each of the three estimates  to $u$ truncated to
appropriate frequency ranges and sum up. 

To be precise, the error in \eqref{S_LE}
can be absorbed on the left provided that the $u$ is time frequency localized 
in $\{ |\xi_0| \leq 2 \tau_0\}$ for a small enough $\tau_0$,
\[
\tau_0 \ll_{M_0,R_0,K_0} 1.
\]
Similarly, the the error in \eqref{high_LE} can be absorbed on the left
provided that the $u$ is time frequency localized in $\{ |\xi_0| >
 \tau_1\}$ for a large enough $\tau_1$,
\[
\frac{1}{\tau_1} \ll_{M_0,R_0,T_0} 1.
\]
Finally, by choosing $\delta$ small enough, $ \delta \ll_{M_0,R_0,K_0,T_0} 1$,
we insure that the error in \eqref{M_LE} can be absorbed on the left
provided that the $u$ is time frequency localized in $\{ \tau_0 < |\xi_0| 
< 2 \tau_1\}$. The implicit constant we obtain there depends on the
range of $\varphi$,
which in turn depends on $\delta$.

The above reasoning shows that we have the desired bound
\eqref{2pt_LE} provided that $u$ is time frequency localized in one of
the three ranges above.  For a general $u$, we simply use time
frequency multipliers to divide it into suitable parts.  To conclude
the argument we need to be able to absorb the frequency truncation
errors. For that, we should estimate the commutator of $P$ with
multipliers $Q _{\leq 1}= Q_{\leq 1}(D_t)$ whose $C^\infty_0$ symbols
are equal to $1$ in a neighbourhood of zero. Precisely, it suffices to
prove the following:

\begin{prop}
Assume that $P$ is asymptotically flat and $\epsilon$ - slowly varying.
Then for $Q$ as above we have the estimate
\begin{equation}\label{P-com}
\| [P,Q_{\leq 1}] u\|_{LE^*} \lesssim \epsilon (\|u\|_{LE^1} + \| Pu\|_{LE^*}).
\end{equation}
\end{prop}

\begin{proof}
  As the commutator will also involve terms containing two spatial
  derivatives of $u$, we first need to estimate those. We claim that 
  the following  elliptic bound holds:
\begin{equation}\label{Pu-ell}
\| \partial_x^2 u\|_{\partial_t LE+ \la r \ra^{-1} LE} \lesssim \|u\|_{LE^1} + \| Pu\|_{LE^*}.
\end{equation}
To see this we first observe that we can freely truncate $u$ first to
a spatial dyadic region $\{|x| \approx 2^k\} $, and then to a time
interval of similar size $2^k$.  Rescaling, we can reduce the problem
to the case when $u$ is supported in a unit ball $B$
\[
\| \partial_x^2 u\|_{\partial_t L^2+ L^2} \lesssim \|u\|_{H^1} + \| Pu\|_{L^2}.
\]
Next we write the equation for $u$ in the elliptic form
\[
-  (\partial_i g^{ij} \partial_j + \partial_t^2) u  = \partial_t f_1 + f_2, 
\qquad f_1 =  2g^{0j} \partial_j u + (g^{00} - 1) \partial_t u
\]
where the terms $f_1$ and $f_2$ have the regularity
\[
\| f_1\|_{L^2} + \|f_2\|_{L^2} \lesssim  \|u\|_{H^1} + \| Pu\|_{L^2}.
\]
Thus it remains to show that for $u$  supported in a unit ball $B$ 
we have the bound
\[
\| \partial_x^2 u\|_{\partial_t L^2+ L^2} \lesssim \|u\|_{H^1} + \| Lu  \|_{\partial_t L^2+L^2},
\qquad L = -   (\partial_i g^{ij} \partial_j + \partial_t^2).
\]
Here, after rescaling, we know that the coefficients of $L$ are
uniformly of class $C^{1,1}$.  The operator $L$ is elliptic and
coercive; therefore its inverse is well defined and has the same
mapping properties as $(-\Delta)^{-1}$.  Denote $v = L^{-1} f_1 \in
H^2_{x,t}$, where we use, say, Dirichlet boundary conditions on a
larger ball $4B$.  Then it suffices to estimate the difference $w = u
- \partial_t \chi_{2B} v$. This is still compactly supported, and
solves
\[
Lw = f_2+ [L,\partial_t \chi_{2B}] v \in L^2.
\]
Thus by elliptic regularity we have $w \in H^2$, and the proof of \eqref{Pu-ell}
is concluded.

In view of  \eqref{Pu-ell}, the estimate \eqref{P-com} reduces to 
\begin{equation}\label{P-com1}
\| [P,Q_{\leq 1}] u\|_{LE^*} \lesssim \epsilon (\|u\|_{LE^1} + \| \partial_x^2 u\|_{\partial_t LE+ \la r \ra^{-1}LE}) .
\end{equation}

We now split the coefficients of $P$ into low and high time frequencies
\[
P= P_{\ll 1} + P_{\gtrsim 1}.
\]
For the low frequency part, all the interactions in the commutator are
localized at fixed time frequency.  We represent the commutator in the
form
\[
[P_{\ll 1}, Q_{\leq 1}] u (t) = \int a(t_1,t_2) P_{t,\ll 1}(t+t_1) Q_1 u(t+t_2) dt_1 dt_2   
\]
where $P_{t}$ is obtained from $P$ by differentiating its coefficients in time,
$Q_{1}$ is a multiplier localized at time frequency $1$
and $a$ is a Schwartz function. Since the bound we use for $u$ and the assumption
on $P$ are both invariant with respect to time translations, it suffices to show that
\[
\| P_{t,\ll 1} Q_1 u \|_{LE^*} \lesssim \epsilon( \| u\|_{LE^1} + \| \partial_x^2 u\|_{\partial_t LE+ \la r \ra^{-1}LE}).
\]
Indeed, we have
\[
\| P_{t,\ll 1} Q_1 u \|_{LE^*} \lesssim (\llp{ \la r \ra g_{t, \ll 1}}{1} +  \llp{ \la r \ra A_{t,\ll 1}}{1}
+  	\llp{ \la r \ra V_{t,\ll 1}}{0})( \|Q_1 u\|_{LE} +  \|Q_1 \partial_x u\|_{LE}
+  \|Q_1 \partial_x^2 u\|_{LE}),
\]
which implies the previous bound.

For the high frequency part there is no commutator to worry about, so we have two terms 
to estimate. For the first we use the time frequency localization directly to write 
\[
\| P_{\gtrsim 1} Q_{\leq 1} u\|_{LE^*} \lesssim ( \llp{g_{\gtrsim 1}}{1}  + \llp{\la r \ra A_{\gtrsim 1}}{1}
+  \llp{\la r \ra^2 V_{\gtrsim 1}}{0} )( \|Q_{\leq 1} u\|_{LE^1} +  \|Q_{\leq 1} \partial_x^2 u\|_{LE})
\]
which suffices. For the second we represent 
\[
\partial_x^2 u = \partial_t f_1 + f_2, \qquad f_1 \in LE, \quad f_2 \in  \la r \ra^{-1}LE, 
\]
and commute out the time derivative in the second order spatial terms in $P$,
\[
g \partial_x^2 u = \partial_t( g f_1) -  g_t f_1 + g f_2.
\]
This allows us to estimate
\[
\| Q_{\leq 1} P_{\gtrsim 1} u \|_{LE^*} \lesssim ( \llp{g_{\gtrsim 1}}{2}  + \llp{\la r \ra A_{\gtrsim 1}}{1}
+  \llp{\la r \ra^2 V_{\gtrsim 1}}{0} )( \| u\|_{LE^1} +  \| \partial_x^2 u\|_{\partial_t LE+ \la r \ra^{-1}LE}),
\]
which again suffices.
\end{proof}


\subsection{The exponential energy dichotomy}

Here we prove Theorem~\ref{exp_dich_thm}. The idea is to combine the
two point local energy decay bound \eqref{2pt_LE} with an energy
relation for the inhomogeneous equation 
\begin{equation}\label{eq-inhom}
Pu = f \in L^1 L^2, \qquad u[0]= (u_0,u_1) \in \dot H^1 \times L^2.
\end{equation}
 The energy relation is closely connected with both 
the symmetry of $P$ and the stationarity of $P$. If both these properties hold,
then we have a conserved energy $E$, but which might not be positive definite.
However, it is positive definite outside a compact set, so 
for solutions to the homogeneous equation, the relation 
$E(u(0)) = E(u(T))$  implies that
\begin{equation}
  \| \partial u \|_{L^\infty L^2(0,T)} \lesssim 
\|\partial u(0)\|_{L^2} + \|\partial^{\le 1} u(T)\|_{
  L^2_{comp} }.
\end{equation}
Here $\partial^{\le 1} u = \sum_{|\alpha|\le 1} \partial^\alpha u$.
The general case uses the following slight extension of this:

\begin{lem} 
Assume that the asymptotically flat operator $P$ is $\epsilon$-almost symmetric 
and $\epsilon$-slowly varying. Then we have an energy relation of the form
\begin{equation}\label{almost-en}
  \| \partial u \|_{L^\infty L^2(0,T)} \lesssim 
\|\partial u(0)\|_{L^2} + \|\partial^{{\le 1}} u(T)\|_{ L^2_{comp}} + \epsilon \|u\|_{LE^1} +  \epsilon^{-1}
\|f\|_{L^1 L^2+LE^*}.
\end{equation}
\end{lem}
\begin{proof}
In a first step, by solving an exterior problem in the small asymptotically flat region
we reduce the problem to the case when $f \in LE^*_{comp}$.

  Assume first that $P$ is symmetric. Then we can define the energy as
  before, except that it is no longer conserved. Instead, the
  regularity of the coefficients insures that the time derivative of
  the energy can be estimated in terms of the local energy.

 If $P$ is not symmetric, then the same applies for the energy
  associated to its symmetrization. Thus we can write the energy relation
\[
E(u(T)) = E(u(0)) + O(\epsilon) \| u\|_{LE^1[0,T]}^2 +
\int_0^T\int | f u_t |
\, dx\, dt.
\]
Since $f \in LE^*$, we can use Cauchy-Schwarz for the last term and
the conclusion follows.
\end{proof}

For the remainder of this section we use \eqref{almost-en} and \eqref{2pt_LE}
to prove the conclusion of Theorem \ref{exp_dich_thm}. Denote by
\[
D = \| \partial u(0)\|_{L^2}^2 + \|f\|_{L^1 L^2+ LE^*}^2, \qquad E(t)
= \|\partial u(t)\|_{L^2}^2{,\qquad E_{comp}(t) = \|\partial^{\le
    1} u(t)\|^2_{L^2_{comp}}}
\]
the size of the data, respectively the solution size and the compact error.

 Combining  \eqref{almost-en} and \eqref{2pt_LE} we obtain
\begin{equation}\label{e-mod-low}
E(T) \lesssim D +  E_{comp} (T) 
\end{equation}
while by \eqref{2pt_LE},
\begin{equation}\label{le-comp1}
\int_{0}^T E_{comp}(t) dt \lesssim  D + E(T).
\end{equation}
Thus we get 
\begin{equation}\label{slow}
\int_{0}^T E(t) dt \lesssim  T D + E(T).
\end{equation}
The function
\[
m(T) = \int_{0}^T E(t) dt 
\]
solves the differential inequality
\[
m'(T) \geq c m(T) - C T D.
\]
Once we pass the threshold $m(T) =  2 c^{-1}CT D$,
we cannot cross it back; so  $m(T)$ must grow at a fixed exponential rate. 
 Otherwise we must have 
\[
 \int_{0}^T E(t) dt \lesssim T D.
\]
Thus on a sequence $T_n \to \infty$  we have 
\[
E(T_n) \lesssim D.
\]
Then from \eqref{2pt_LE} we get uniform energy boundedness and local energy decay.

\section{Spectral theory in the time independent case}
\label{s:stat}

\subsection{ Local energy decay vs. resolvent bounds}

Here we prove Proposition~\ref{p:le-res}, which establishes the
equivalence between local energy decay and resolvent bounds.
For convenience we restate the two estimates. The local energy decay 
bound has the form 
\begin{equation}
		\lp{u}{LE^1[0,T]} +\lp{\partial u}{L^\infty L^2[0,T]}   \ 
\lesssim \ \lp{\partial u(0)}{L^2} 
		 + \lp{P u}{LE^* + L^1L^2[0,T]} \ , \label{LE-re}
\end{equation}
while the resolvent local energy bound is
\begin{equation}\label{LE-fourier-re}
\| R_\omega \|_{\LE^* \to \LE^1_{{\omega}}} \lesssim 1, \qquad \Im \omega < 0.
\end{equation}

{\em A. Local energy decay implies LE resolvent bounds.}
A direct consequence of the local energy decay estimate \eqref{LE-re} is the 
uniform energy estimate
\[
 Pu=0\quad \implies\quad\lp{\partial u}{L^\infty L^2[0,\infty)} \lesssim \ \lp{\partial
   u(0)}{L^2}.
\]
This in turn shows, via the formula \eqref{fourier-laplace}, that the resolvent is 
$R_{\omega}$ is defined and holomorphic in the lower half-space, and the resolvent bound
\eqref{enres} holds. We remark that for $- \Im \omega \gtrsim 1$, the bound  
\eqref{LE-fourier-re} follows from \eqref{enres}. Thus we are left with the case
$-1 <  \Im \omega < 0$.


Let $u \in \dot H^1$ so that $P_{\omega } u = f$. Then we can produce an exponentially growing
 solution for the inhomogeneous $P$ equation by setting
\[
v = e^{i\omega t} u, \qquad g =  e^{i\omega t} f, \qquad P v = g.
\]
We apply the local energy bound \eqref{LE-re} to $v$ on the time
interval $[-T,0]$ and then let $T \to \infty$. This yields  the bound
\begin{equation}\label{LE2res}
\| u\|_{\LE^1_{\omega}} + |\Im \omega|^\frac12 \| u\|_{\dot H^1_{\omega}}  
\lesssim \| f\|_{\LE^* +  |\Im \omega|^\frac12 L^2},
\end{equation}
which implies \eqref{LE-fourier-re}.

{\em B. LE resolvent bounds imply local energy decay.}  We first
observe that the resolvent LE bounds \eqref{LE-fourier-re} imply the
resolvent energy bounds \eqref{enres}.  Indeed, if resolvent LE bounds
hold, then in particular there are no eigenvalues in the lower
half-plane. Hence, by Fredholm theory, the resolvent is holomorphic in
the lower half plane with values in $\mathcal L (L^2,\dot H^1_{\omega})$.

It
remains to prove the quantitative bound. Given $f \in L^2$, we split
it into two regions,
\[
f = f_{in} + f_{out} := \beta_{in} f +\beta_{out} f 
\]
where
\[
\beta_{in} =\beta_{ <| \Im \omega|^{-1}} , \qquad \beta_{out} = \beta_{ >| \Im \omega|^{-1}}. 
\]

For the inner part we have the estimate
\[
\|R_{\omega} f_{in}\|_{\LE^1_{{\omega}}}  \lesssim \| f_{in}\|_{\LE^*} \lesssim | \Im \omega|^{-\frac12} \|f\|_{L^2},
\]
which suffices in the region $\{|x| \lesssim | \Im \omega|^{-1}\}$.  To estimate 
$R_{\omega} f_{in}$ in the outer region we truncate it there, and set $u = \beta_{out}R_{\omega} f_{in}$.
In the support of $u$ we are in the small asymptotically flat regime, where 
we know the resolvent bound.  Indeed, by \eqref{local-res}, it
  suffices to consider $\Im \omega$ in a bounded strip such as
  $(-2\beta,0)$, and we then simply choose the constants sufficiently
  large depending on the constant $\beta$ in \eqref{local-en}.   Thus, commuting $\beta_{out}$ with $P$   we have
\[
\| u\|_{\dot H^1_{{\omega}}} \lesssim | \Im \omega|^{-1}\|P_\omega u\|_{L^2} \lesssim | \Im \omega|^{-\frac12}
\|R_{\omega} f_{in}\|_{\LE^1_\omega}  +  | \Im \omega|^{-1} \| f_{in}\|_{L^2} \lesssim  | \Im \omega|^{-1}
\| f_{in}\|_{L^2}.
\] 

For the outer part we proceed in two steps. First we introduce an auxiliary operator 
$\tilde P$ which is a small asymptotically flat perturbation of $\Box$ which agrees 
with $P$ in the outer region $\{|x| \gtrsim | \Im \omega|^{-1}\}$. We use its associated
resolvent to define the outer approximate solution
\[
u_{out} = \beta_{out} \tilde R_\omega f_{out}
\]
Commuting to estimate the truncation error, we have 
$P_\omega u_{out} - f_{out}$ supported in the region $\{ |x|
\lesssim | \Im \omega|^{-1}\}$
and satisfying the bound
\[
\| P_\omega u_{out} - f_{out}\|_{L^2} \lesssim| \Im \omega| \|\tilde R_\omega f_{out}\|_{\dot H^1_\omega}
+ \| f_{out}\|_{L^2} \lesssim \| f_{out}\|_{L^2}.
\]
Hence for this error we can use the same argument as for the inner part.

As a consequence of the uniform energy resolvent bound  \eqref{enres}, it follows
that the solutions to the linear $P$ equation have subexponential
energy decay, see Proposition~\ref{p:e-res} (c).  This implies that the resolvent formula
\eqref{fourier-laplace} holds pointwise in the lower
half-space. Suppose that $u$ solves the forward problem $Pu=f \in
L^2_{x,t}$, with $f$ supported in $t \geq 0$. 
Conjugating by a decaying exponential $e^{-\sigma t}$ we
obtain the equation
\[
P(x,D_x,D_t - i \sigma) v = g, \qquad v = e^{-\sigma t} u, \quad g = e^{-\sigma t} f.
\]
Then by the Fourier inversion formula applied on the 
line $\Im \omega = -\sigma$ we obtain the representation formula for the time 
Fourier transform
\[
\hat v(x,\tau) = R_{\tau-i\sigma}  \hat g(x,\tau).
\]
We claim that the resolvent LE bound \eqref{LE-fourier-re} implies the estimate
\begin{equation}\label{exp-u}
\| e^{-\sigma t} u\|_{LE^1} = \| v\|_{LE^1} \lesssim \|g\|_{LE^*} = \| e^{-\sigma t} f  \|_{LE^*}. 
\end{equation}
Given the dyadic $\ell^\infty$ structure of the $LE^1$ norm, respectively 
the $\ell^1$  structure of the $LE^*$ norm,  it suffices to prove the bound 
in a fixed dyadic region for $v$, assuming that $g$ is also supported in a fixed  
 (possibly different) dyadic region. But then \eqref{exp-u} follows 
from \eqref{LE-fourier-re} directly by Plancherel's theorem.

Letting $\sigma \to 0$ in \eqref{exp-u} we obtain the unweighted local energy bound
\[
\| u\|_{LE^1}\lesssim  \|  f  \|_{LE^*} 
\]
for the solution $u$ to the forward problem $Pu = f$ with zero Cauchy
data.  Finally, using the extension procedure in Section~\ref{s:extend},
one can easily add nontrivial finite energy Cauchy data, and then use
energy estimates to control the uniform energy norms and obtain
\eqref{LE-re}.

\subsection{Real resonances and the limiting 
absorption principle}

Here we consider the question of extending the resolvent to the real axis.
Our main result clarifies the definition of the resolvent on the real axis.

\begin{prop}
  Assume that the LE resolvent bounds \eqref{LE-fourier-re}  hold uniformly in the lower
  half-space $H$ near some real $\omega_0 \neq 0$. Then the limit 
\[
R_{\omega_0} f = \lim_{\omega \to \omega_0} R_\omega f, \qquad f \in \LE^*,
\]
converges strongly on compact sets. Further, the resolvent bound
\eqref{LE-fourier-re} holds at $\omega_0$ and $R_{\omega_0} f$
satisfies the outgoing radiation condition \eqref{outgoing}.
\end{prop}

\begin{proof}
We first prove that the bound 
\begin{equation}\label{real-res}
\| u\|_{\LE^1_{\omega_0}} \lesssim \| P_{\omega_0} u\|_{\LE^*}
\end{equation}
holds for all $u$ for which both norms above are finite and which
satisfy the outgoing radiation condition \eqref{outgoing}.  Suppose
$P_{\omega_0} u = f$.  Set $u_\epsilon = e^{-\epsilon r } u$. Then we
have
\begin{multline*}
P_{ \omega_0 - i\epsilon} u_\epsilon = e^{-\epsilon r } f
+e^{-\epsilon r }\Bigl( -\epsilon^2 (g^{rr} +g^{00}) u
+2 \epsilon( g^{rr} \partial_r u - i\omega_0 g^{00}u) \\+ \epsilon \O(|g - I|) |\nabla
u| +\varepsilon \O(|A_\alpha|+|\partial g|+|g-I|) |u| 
  + \frac{2\varepsilon}{r} u\Bigr).
\end{multline*}

The bound \eqref{real-res} follows by applying the resolvent bound for
$\omega_0 - i\epsilon$ and passing to the limit. For that we need to show 
that 
\[
\lim_{\epsilon \to 0} P_{ \omega_0 - i\epsilon} u_\epsilon -  e^{-\epsilon r } f 
= 0 \qquad \text{in } \LE^*
\]
 All terms in the difference decay due to the $\ell^1$ dyadic summation, except 
for the third term where we also use the radiation condition.

Next we show that for each $f \in \LE^*$ there exists $u \in \LE^1_{\omega_0}$,
satifying the outgoing radiation condition, so that $P_{\omega_0} u =
f$. Take a seqence $H \ni \omega_n \to \omega_0$ and corresponding
solutions $u_n $ to $P_{\omega_n} u_n = f$.  Then we obtain $u$ as a
weak limit of a subsequence of $u_n$, and the uniform $\LE^1$ bound
for $u_n$ transfers to $u$. Further, by elliptic regularity the
sequence $u_n$ is locally in $H^2$, therefore strong $H^1$ convergence
follows on compact sets. It remains to show that the outgoing radiation condition 
\eqref{outgoing} holds for the limit.

Here we shall again use a positive commutator argument.  A related,
though easier,
calculation for the Schr\"odinger equation can be found in \cite{MMT}.
For convenience, we abbreviate $\Delta_g = \Delta_{g,0}$, where the
latter notation is from \eqref{p_omega}.  And we shall again assume
without loss that $g^{00}=-1$.  We record that
\begin{multline}\label{rad_cond_comm}-2\Im \la Qu, (\Delta_g -
  (\omega_0-i\varepsilon)^2 + (\omega_0-i\varepsilon)(g^{0k}D_k+D_kg^{0k}) )u\ra \\= \la
i[\Delta_g + \omega_0(g^{0k}D_k+D_kg^{0k}) ,Q]u,u\ra +
4\varepsilon\omega_0 \la Qu,u\ra
\\-\varepsilon\la Qu,(g^{0k}D_k+D_kg^{0k})u\ra - \varepsilon\la (g^{0k}D_k+D_kg^{0k})u,Qu\ra
\end{multline}
for a symmetric operator $Q$.  We set
\[Q = f(R)\Bigl(\frac{x_i g^{ij}}{R}D_j +
g^{0k}D_k + \omega_0\Bigr) +
\Bigl(D_i\frac{g^{ij}x_j}{R} + D_k
g^{0k}+\omega_0\Bigr)f(R),\quad R^2= x_ig^{ij}x_j.\]

The choice of multiplier is motivated by the last three terms in
\eqref{rad_cond_comm}.  Indeed, we have
\begin{multline*}
  \omega_0 Q =
  \Bigl(D_k\frac{g^{kl}x_l}{R}+\omega_0\Bigr)f(R)\Bigl(\frac{x_i
    g^{ij}}{R}D_j+\omega_0\Bigr) 
+ \Bigl(D_i -D_l\frac{g^{lk}x_kx_i}{R^2}\Bigr) g^{ij}f(R)\Bigl(D_j -
\frac{x_jx_mg^{mn}}{R^2}D_n\Bigr) 
\\- \frac{f(R)}{2}(\Delta_g -
\omega_0^2+\varepsilon^2+(\omega_0-i\varepsilon)(g^{0k}D_k+D_kg^{0k}))
- (\Delta_g - \omega_0^2+\varepsilon^2 +(\omega_0-i\varepsilon)(g^{0k}D_k+D_kg^{0k}) )
\frac{f(R)}{2} \\-i\varepsilon \frac{f(R)}{2}(g^{0k}D_k + D_k g^{0k})
- i\varepsilon (g^{0k}D_k + D_k g^{0k})\frac{f(R)}{2}+f(R)\varepsilon^2 + \frac{1}{2}(\Delta_g f(R)).
\end{multline*}
Using this, we observe
\begin{multline*}
  4\varepsilon\omega_0 \la Qu,u\ra -\varepsilon\la
  Qu,(g^{0k}D_k+D_kg^{0k})u\ra - \varepsilon\la
  (g^{0k}D_k+D_kg^{0k})u,Qu\ra
\\= 4\varepsilon\la \Bigl(D_k\frac{g^{kl}x_l}{R} - D_k g^{0k} +\omega_0\Bigr)f(R)\Bigl(\frac{x_i
    g^{ij}}{R}D_j - g^{0j}D_j+\omega_0\Bigr) u,u\ra
+ 4\varepsilon \la D_k g^{0k} f(R) g^{0j}D_j u,u\ra
\\+ 4\varepsilon\la \Bigl(D_i -D_l\frac{g^{lk}x_kx_i}{R^2}\Bigr) g^{ij}f(R)\Bigl(D_j -
\frac{x_jx_mg^{mn}}{R^2}D_n\Bigr) u, u\ra + 4\varepsilon^3 \la f(R) u,u\ra
\\-4\varepsilon \Re\la (\Delta_g -
\omega_0^2+\varepsilon^2+(\omega_0-i\varepsilon)(g^{0k}D_k+D_kg^{0k}))
u, f(R) u\ra   +  \la E_1 u, u\ra + \la E_2 u, u\ra.
\end{multline*}
Here, upon fixing $f(R)=\frac{R}{R+2^j}$ and noting that
$f(R)|_{R\approx 2^k} \approx 2^{k-j}$ when $k<j$, we have that the
first four terms on the right are non-negative and that the errors satisfy
\[\la E_1 u, u\ra|_{|x|\approx 2^k} = \varepsilon^22^{-(k-j)^-} \O(c_k + 2^{-k})
\|(u,\nabla u)\|^2_{L^2},\]
\[\la E_2 u,u\ra|_{|x|\approx 2^k} =2^{-(k-j)^-}\O(c_k + 2^{-k})(\|u\|^2_{\LE} +
\|\nabla u\|^2_{\LE}).\]

On the other hand, for the purposes of the commutator, up to error
terms that are easily handled, we may replace $g^{ij}$ by the identity
matrix which yields
\begin{multline*}
  i[\Delta_g + \omega_0(g^{0k}D_k+D_kg^{0k}),Q] =2D\Bigl(2\frac{f(R)}{R}-f'(R)\Bigr)D -
  2Dx\Bigl(2\frac{f(R)}{R^3} - \frac{f'(R)}{R^2}\Bigr)xD \\+
  f'(R)(\Delta_g-\omega_0^2+\varepsilon^2) + (\Delta_g-\omega_0^2+\varepsilon^2)f'(R) +
  2(Dx+\omega_0r)\frac{f'(R)}{rR}(xD+r\omega_0) - 2f'(R)\varepsilon^2+ E_2.
\end{multline*}
This choice of $f$ gives
positive coefficients in the first two factors above.  Moreover, we
have that and
$f'(R)|_{R\approx 2^j}\approx 2^{-j}$.

We may thus combine the two calculations above to see that 
\begin{align*}2^{-j}\|(\partial_r + i\omega_0)u\|^2_{L^2(|x|\approx 2^j)}&\lesssim \Re\la (\Delta_g -
(\omega_0-i\varepsilon)^2 + (\omega_0-i\varepsilon)(g^{0k}D_k+D_kg^{0k}))u, (iQ - 4\varepsilon f(R) +
2f'(R))u\ra
\\& \qquad\qquad + \la E_1u,u\ra + \la E_2u,u\ra\\
&\lesssim \sum_{k\ge 0} 2^{-(k-j)^-}
\Bigl(2^{j/2}\|P_{\omega_0-i\varepsilon} u\|_{L^2(|x|\approx
  2^j)}\Bigr)\Bigl(2^{-j/2} \|(u,\nabla u)\|_{L^2(|x|\approx 2^j)}
\Bigr)
\\&\qquad\qquad\qquad\qquad+(\|(u,\nabla u)\|^2_{\LE} + \varepsilon^2 \|u\|^2_{L^2})\Bigl(\sum_{k\ge 0} 2^{-(k-j)^-}(c_k +
2^{-k})\Bigr).
\end{align*}
Applying this estimate to the sequence $u_n = R_{\omega_n} f$ and
utilizing the known resolvent energy and resolvent LE bounds shows
that the outgoing radiation condition holds.

Finally, we observe that by \eqref{real-res} the limit $u$ of the 
subsequence of $u_n$ is unique; therefore
the full sequence $u_n$ converges to $u$ strongly in $H^1_{loc}$.
\end{proof}

\subsection{Nonzero resonances and LE resovent bounds}  
Here we prove the equivalence of the three statements in
Proposition~\ref{p-nonzero}. We do this in several steps:

{\bf 1. A simple case.} Here we consider a simple situation, namely
when $A=0$, $V=0$. Then $P$ is self-adjoint, and the conserved energy
associated to $P$ is positive definite.  Then by
Theorem~\ref{2pt_thm}, local energy decay holds for $P$. Thus, the
resolvent extends to the real axis with uniform bounds.

{\bf 2. A perturbative formulation.}
Given a stationary operator $P$, we denote by $\tilde P$ its principal part, with 
$A = 0$ and $V=0$.  Then the resolvent $\tilde R_\omega$ of $\tilde P_\omega$ is defined  globally 
in the lower half-space, as well as on the real line. We seek a solution $u$
to $P_\omega u = f$ of the form $u = \tilde R_\omega g$. Denoting $Q_\omega  = P_\omega - \tilde P_{\omega}$, we rewrite our equation as 
\[
 (I + Q_\omega \tilde R_{\omega}) g = f.
\]
The operator $Q_\omega \tilde R_{\omega}$ is compact, and depends analytically 
on $\omega$ in the lower half-space and continuously on the real line. 

If $- \Im \omega$ is large enough then $I +  Q_\omega \tilde R_{\omega}$ is invertible. Hence 
it has zero Fredholm index everywhere. Thus statements (a), (b), (c) in 
Proposition~\ref{p-nonzero} are equivalent by standard Fredholm theory.

\subsection{ Stability in LE resolvent bounds}

Here we prove the stability statement in Proposition~\ref{p-nonzero}.
Precisely, we assume that the resolvent bounds \eqref{LE-fourier-re}
for $P$ hold at some real $\omega_0$, and show that they must hold
uniformly in a neighbourhood of $\omega_0$, not only for $P$ but also
for small AF perturbations of $P$.  

The proof is by contradiction. Suppose we have a sequence of operators
$P_n$ converging to $P$ in the $AF$ norm, a sequence of numbers $\omega_n \in H$
converging to $\omega$, and a sequence of functions $u_n$ so that
\[
\| u_n\|_{\LE^1_{\omega_n}} = 1, \qquad \| P_n u_n\|_{\LE^*} \to 0.
\]
By elliptic estimates the sequence $u_n$ is also locally bounded in $H^2$; 
therefore on a subsequence it converges to some $u \in \LE^1_{\omega_0}$ so that $P_{\omega_0}u =0$.
To reach a contradiction we need the following two observations:

(a)  By exterior local energy resolvent bounds, i.e. those
  obtained by mimicking the procedure that yielded \eqref{LE2res} when
  one starts instead from \eqref{small-cut}, a nontrivial portion of $u_n$ 
is localized to a compact set, $\|u_n\|_{H^1_{<2R_0}} \gtrsim 1$. This shows that
$u \neq 0$.

(b) By the same argument as in the previous section, the function $u$ satisfies the 
outgoing radiation condition \eqref{outgoing}. 

We remark that this argument implies in effect the slightly stronger conclusion
\[
\limsup_{n \to \infty} \| R_n(\omega_n)\|_{\LE^* \to \LE^1_{\omega_n}} 
\lesssim \| R(\omega_0)\|_{\LE^* \to \LE^1_{\omega_0}}.
\]

\subsection{  High frequency analysis.}
\label{s:le-high}
Here we prove that the high frequency bound in Theorem
\ref{1pt_thm_high} implies the uniform LE resolvent bound for large
$\omega$.  The case $-\Im \omega \gg 1$ follows directly from 
energy estimates; see \eqref{local-res}. Hence we assume that $0 < -
\Im \omega \lesssim 1$.

Suppose that $P_\omega u = f$. Then $v = e^{i\omega t} u$ solves 
$Pv = g$ with $g =  e^{i\omega t} f$. Applying the bound \eqref{high_LE} to 
$v$ in $(-\infty,0]$ we obtain
\[
\| u\|_{\LE^1} + |\omega|\|u\|_{\LE} + |\Im \omega|^{\frac12} \|u\|_{\dot H^1_{\omega}}
 \lesssim \| \la x\ra^{-2} u\|_{\LE} + \| f\|_{\LE^*+|\Im \omega|^{\frac12} L^2 }.
\]
If $\omega$ is large enough then the first term on the right is absorbed on the left,
and the conclusion follows.

\subsection{ The general  case}

Here we prove Theorem~\ref{t:stat} (a); part (b) will then follow due
to Theorem~\ref{nores}. Since there are no eigenvalues in the lower
half-space, the resolvent is holomorphic in $H$.  Combining this with
the above high frequency bound, it follows that we must have an
estimate of the form
\[
\| R_{\omega}\|_{L^2 \to \dot H^1_{\omega}} \leq c(\Im {\omega}), \qquad \Im \omega < 0.
\]
This in turn implies the LE resolvent bound outside any neighbourhood of the real line.
Now we consider a strip near the real line. We have three cases:

(i) $|\omega|$ is sufficiently large. Then, as discussed in the previous subsection, 
the LE resolvent bound \eqref{LE-fourier-re} holds. 

(ii) $|\omega|$ is sufficiently small.  Then from
Theorem~\ref{low_LE_thm}, arguing  as in the previous subsection,
we obtain again \eqref{LE-fourier-re}.

(iii) $|\omega|$ is away from zero or infinity. Here from hypothesis
in case (a) and from Theorem~\ref{nores}, we know that there are no
real resonances.  Then Proposition~\ref{p-nonzero} shows that
\eqref{LE-fourier-re} holds for $\omega$ sufficiently close to the
real axis.

The proof of  \eqref{LE-fourier-re}, and thus of Theorem~\ref{t:stat}, is concluded.





\section{Spectral Theory and Exponential Trichotomies}
\label{s:nonstat}

\bigskip

This section is concerned with deriving Theorems \ref{t:stat+} and \ref{exp_trich_thm}
from Theorems~\ref{2pt_thm} and \ref{exp_dich_thm}.  First we give a different
perspective for the spectral theory for symmetric $P(x,D)$ in the time
independent case. This material is more or less standard
(e.g.~\cite{B_indef}, \cite{L_gen}, \cite{L_KG}).

Next, we use the time independent analysis to derive a perturbative
result which holds on large time intervals assuming the symbol of
$P(t,x,D)$ has sufficiently slow time variation.  Based on this we
string together the local analysis via a Perron type argument to
globally construct the stable/unstable/center subspaces
$S^{\pm,0}(t)$. This last step follows closely work of Coppel
\cite{Coppel}.

 
\subsection{Spectral theory in the time independent case}
 
Here we assume that the operator $P$ is symmetric, and we write it in the form
\[
P = P_0 +   B(x,D_x) D_t +  g^{00} D_t^2
\]
where
\begin{align}
		P_0  \ &= \   (D+A)_j  g^{jk}(D+A)_k + W(x,D_x) \ , \notag\\
		W(x,D_x) \ &= \ 
		A_0g^{0k}(D_k+A_k) + (D_k+A_k) A_0 g^{0k}+ g^{00}A_0^2+ V(x) \ , \notag\\
		B(x,D_x)\ &=\ g^{0k}(D_k+A_k) + (D_k+A_k) g^{0k}   +2A_0g^{00}\ . \notag
\end{align}
Using this expression we recast the equation $Pu = f$ as a system, setting
\[
\U=\begin{pmatrix} u \\ \partial_t u
		\end{pmatrix}   =  \begin{pmatrix}  \U_0 \\ \U_1 \end{pmatrix}, \qquad
F = \begin{pmatrix} 0 \\ i (g^{00})^{-1} f  	\end{pmatrix} .  
\]
Then we write the
inhomogeneous wave equation $Pu = f$ as a Schr\"odinger type equation:
\begin{equation}
		D_t \U \ = \ \mathcal{H}\U + F \ , \qquad\hbox{where\ \ }
		\mathcal{H} \ = \ \frac{1}{i}\begin{pmatrix} 0 & 1\\
                  (g^{00})^{-1}P_0 & -i{(g^{00})^{-1}} B
		\end{pmatrix} \ . \label{H_form}
\end{equation} 
The homogeneous evolution  preserves the energy form on $\dot{H}^1\times L^2$ defined by:
\begin{equation}
		E[\U] \ = \ \la \U,\U\ra_E \ , \qquad
		\hbox{where\ \ } \la \U,\Psi\ra_E \ = \ \int_{\mathbb{R}^3} P_0 \U_0\overline{\Psi_0} - g^{00}\U_1\overline{\Psi_1}   dx \ . \label{E_form}
\end{equation}
One readily checks that $ \mathcal{H}$ is a symmetric operator with respect to this energy form, 
i.e.
\[
\la \mathcal{H} \U,\Psi\ra_{{E}}=\la  \U,\mathcal{H} \Psi\ra_{{E}}
\]
on an appropriate dense domain for $\mathcal{H}$ (say $H^2\times
H^1$), and that $\mathcal{H}$ extends to a closed operator on
$\dot{H}^1\times L^2$.

To relate the spectral theory of $\mathcal H$ with that for $P$, we
remark that the solutions to the linear eigenvalue problem
$\mathcal{H}\U_\omega=\omega \U_\omega$ are of the form $\U_\omega=
(u_\omega,i\omega u_\omega)$, where $u_\omega$ solves $P_\omega
u_\omega = 0$.  More generally, we can relate the resolvent of
$\mathcal H$, which we denote by $\mathcal R_\omega$, with the
resolvent $R_\omega$ of $P$ as follows:
\begin{equation}
\mathcal R_\omega = \begin{pmatrix} {-} R_\omega( B {+}\omega g^{00})  &  {i}R_\omega g^{00} \\ 
  {i(-}\omega R_\omega( B {+}\omega g^{00}) + {I)} &  {-}\omega R_\omega g^{00}
		\end{pmatrix}{.}
\end{equation}

As a straightforward consequence of Proposition~\ref{p:e-res} we have the following
\begin{prop}
  Assume that $P$ is asymptotically flat. Then the resolvent $\mathcal
  R_\omega \in \mathcal L(\dot H^1 \times L^2)$ is meromorphic in the
  lower half-space, with the same poles as $R_\omega$. Further, for
  each $\omega$ in the lower half-space the bounds \eqref{enres} and
\begin{equation}
\| \mathcal R_\omega \|_{\mathcal L(\dot H^1 \times L^2)} \lesssim |\Im \omega|^{-1}
\end{equation}
are equivalent.
\end{prop}

The non-trivial spectral theory of $\mathcal{H}$ in general stems from
the fact that $\la \cdot,\cdot\ra_E$ may not be positive definite.  In
fact, if this form has a negative index then there exist non-trivial
null vectors $\la \U,\U\ra_E=0$, and one easily checks that solutions
to $\mathcal{H}\U_\omega=\omega \U_\omega$ with non-real $\omega$ must
have this property.  On any plane where $\la \cdot,\cdot\ra_E$ is
trivial (an ``isotropic subspace'') the self-adjointness of
$\mathcal{H}$ gives no additional information so one can expect the
eigenspace structure of $\mathcal{H}$ to be more or less arbitrary
there, unless of course more structure on $P$ is assumed (e.g.~if
$A_0=g^{0k}=0$ then non-real eigenvalues must be purely imaginary).

The interesting fact is that in the case we are considering,
i.e.~where $P$ is symmetric and the zero frequency estimate
\eqref{zero-res2} holds, the discrete spectral theory of $\mathcal{H}$
is completely described by two maximal $\mathcal{H}$ invariant isotropic
subspaces $S^\pm$ whose dimension must be equal to the negative
index of $\la \cdot,\cdot\ra_E$. Then one can decompose
$\dot{H}^1\times L^2=(S^-+ S^+)\oplus_E S^0$,  where $\oplus_E$ denotes an orthogonal
direct sum with respect to $\la \cdot,\cdot\ra_E$.  The crucial fact
now is that $S^0$ is a closed sub-space of $\dot {H}^1\times L^2$ on
which $\la \cdot,\cdot\ra_E$ becomes positive definite. In particular
$\la \cdot,\cdot\ra_E$ restricted to $S^0$ defines a \emph{positive
  conserved} energy norm which is equivalent to $\lp{\cdot
}{\dot{H}^1\times L^2}$. We sum this information up as follows:
 
\begin{prop}[Spectral theory for symmetric $P(x,D)$]\label{z_spec_prop}
  Let $P(x,D)$ be a symmetric, stationary, asymptotically flat wave operator.  Let
  $\mathcal{H}$ be the Hamiltonian matrix of $P$ as defined on line
  \eqref{H_form}. Suppose in addition that the zero spectral estimate
  \eqref{zero-res2} holds for some $K_0>0$. Then one has the
  following:
\begin{enumerate}[i)]
		\item (The pair ($\dot{H}^1\times L^2,\la\cdot, \cdot \ra_E$) is a Pontryagin space)
                  \label{pont_prop} The energy form $\la\cdot, \cdot
                  \ra_E$ is nondegenerate on the energy space
                  $\mathcal{E}=\dot{H}^1\times L^2$, and we can write
                  it as orthogonal direct sum $\mathcal{E}=
                  \mathcal{E}^-\oplus_E \mathcal{E}^+$ where
                  $\mathcal{E}^-$ has dimension $\kappa\in\mathbb{N}$,
                  and where $\lp{\U}{E}^2=\la \U^+,\U^+\ra_E-\la
                  \U^-,\U^-\ra_E \approx \lp{\U}{\dot{H}^1\times
                    L^2}^2$. Here $\U=\U^-+\U^+$ is the decomposition
                  of $\U$ along $\mathcal{E}^\pm$.
              \item (Pontryagin's Theorem) \label{Pont_thm} There
                exists two  (maximal) $\kappa$-dimensional isotropic
                $\mathcal{H}$ invariant subspaces $S^{\pm}\subseteq
                \mathcal{E}$, and an $\alpha>0$ such that
                $\Im(\omega)\geq \alpha$ for all $\omega\in
                spec(\mathcal{H}|_{S^-})$ and $\Im(\omega) < - \alpha$
                for all $\omega\in spec(\mathcal{H}|_{S^+})$.
Moreover
                $\mathcal{E}_{disc}=S^-+S^+$
                accounts for the entire discrete spectrum of
                $\mathcal{H}$, and $\la\cdot ,\cdot \ra_E$ restricted
                to $\mathcal{E}_{disc}$ is a nondegenrate form of
                index $(-\kappa,\kappa)$.
		\item (Coercivity away from $S^\pm$) \label{E_coerce} Finally, if we write 
		$\mathcal{E}_{ac}=\mathcal{E}_{disc}^\perp$ then 
		$\la\cdot ,\cdot \ra_E$ restricted to  $\mathcal{E}_{ac}$ gives rise to a norm 
		equivalent to  $\dot{H}^1\times L^2$:
		\begin{equation}
				\la\U,\U\ra_E \ \geq  \ c \lp{\U}{\dot{H}^1\times L^2}^2 \ , 
				\qquad \U\in \mathcal{E}_{ac} \ . \label{coerc}
		\end{equation}
\end{enumerate}
\end{prop}


\begin{proof}[Proof of Proposition \ref{z_spec_prop}]
  This material is pretty standard and essentially covered in 
  \cite{B_indef}, \cite{L_gen}, \cite{L_KG}, etc.  For Part
  \ref{pont_prop}) considerations are similar to those in
  \cite{L_KG}. For Part \ref{Pont_thm}), which is the main invariant
  subspace theorem of Krein space theory, see in particular Theorem
  3.2 of \cite{B_indef}. Part \ref{E_coerce}) uses some basic geometry
  of Krein spaces (see Section 1 of \cite{L_gen}), and the closed
  graph theorem.
\end{proof}
 
To the subspaces $S^{\pm}$ we can associate commuting 
projectors $P^{\pm}$  in a standard fashion, using the resolvent:
\begin{equation}\label{ppm}
P^{\pm} = \int_{\gamma^{\pm}} \mathcal R_\omega d \omega
\end{equation}
where $\gamma^{\pm}$ are contours enclosing the spectrum of $\mathcal H$ in the lower,
respectively the upper half-space. Similarly, the wave flow on $S^{\pm}$ is expressed as
\begin{equation}\label{eih-ppm}
e^{it \mathcal H} P^{\pm} = \int_{\gamma^{\pm}} e^{it\omega} \mathcal R_\omega d \omega{.}
\end{equation}

Now we are able to complete the proof of Theorem~\ref{t:stat+}. The
key element in the proof is the following.

\begin{prop}\label{p:ppm}
  Let $P$ be stationary, symmetric, asymptotically flat, nontrapping
  and satisfying the zero resolvent bound \eqref{zero-res2}. Then 

a) We have 
the following estimates
\begin{equation}
\| P^{\pm}\|_{\mathcal L(\dot H^1 \times L^2)} \lesssim 1
\end{equation}
\begin{equation}
\begin{split}
  \|{e^{it \mathcal H}}_{|S^+}\|_{\mathcal L(\dot H^1 \times L^2)}
  \lesssim e^{ \alpha t}, \quad  t < 0,  & \ \qquad  \|{e^{it \mathcal
      H}}_{|S^+}\|_{\mathcal L(\dot H^1 \times L^2)} \lesssim e^{\beta
    t}, \quad t \geq 0
  \\
 \|{e^{it \mathcal H}}_{|S^-}\|_{\mathcal
    L(\dot H^1 \times L^2)} \lesssim  e^{-\beta t}, \quad t < 0,  & \ \qquad     \|{e^{it \mathcal H}}_{|S^-}\|_{\mathcal
    L(\dot H^1 \times L^2)} \lesssim e^{-\alpha t}, \quad t \geq 0
\end{split}
\end{equation}
with implicit constants as well as $\alpha$ and $\beta$ depending only
on our parameters $M_0,R_0,T_0$ and $K_0$.

b) The constant $c$ in \eqref{coerc} also depends only on our
parameters $M_0,R_0,T_0$ and $K_0$.

c) The operators $P^{\pm}$ and $e^{it \mathcal H} P^{\pm}$ have a Lipschitz dependence
on the coefficients of $P$ in the asymptotically flat topology.
\end{prop}

We remark that Theorem~\ref{t:stat+} follows immediately. The
exponential bounds for $S^+$ and $S^{-}$ follow immediately from the
above bounds.  On the other hand, solutions in $S_0$ have bounded
energy by part (iii) of Proposition~\ref{z_spec_prop}. Hence, by
Theorem~\ref{exp_dich_thm} we have both uniform energy bounds and
local energy decay. We further remark that on $S_0$ the energy and the
 form $\la \cdot,\cdot\ra_E$ are equivalent.

\begin{proof}[Proof of Proposition~\ref{p:ppm}]
  a) By symmetry it suffices to consider the ``$+$'' case. The result
  in part (a) will follow from the integral representations
  \eqref{ppm}, \eqref{eih-ppm} if we can find a contour rectangle $R =
  [-M,M] \times [-\alpha,-\beta]$ which contains all the spectrum of
  $P$ in the lower half-space, so that the estimate \eqref{enres}
  holds uniformly on the boundary $\gamma^+$ of $R$.

  The parameter $\beta$ is easily found as a consequence of the 
  energy estimates, see \eqref{local-res}.  The parameter $M$ is found
  as a consequence of the high frequency bounds in
  Section~\ref{s:le-high}.  Finally, to find $\alpha$ we need the
  full strength of the two point LE bound in 
  Theorem~\ref{2pt_thm}. Indeed, suppose that $P_\omega u = f$.
Then $v = e^{i\omega t} u $ solves $Pv = g$ with $g = e^{i\omega t} f$.
Hence applying the two point LE estimate to $v$ on $(-\infty,0]$ we obtain
\[
\| u\|_{\LE^1} + |\omega| \|u\|_{{\LE}} \lesssim \|f\|_{\LE^*} +  
|\Im \omega|^\frac12 (\| u\|_{\dot H^1} + |\omega|\|u\|_{L^2}).  
\]
On the other hand, for $R > R_0$ we have the energy relation \eqref{E-inhom}, which
combined with \eqref{e-pos} gives
\[
|\Im \omega| (\| u\|_{\dot H^1}^2 + |\omega|^2 \|u\|_{L^2}^2) \lesssim |\Im \omega| \|u\|_{L^2_{comp}}^2 +  |\omega |  \| f\|_{{\LE}^*+ |\Im \omega|^{\frac12}L^2 } \|u\|_{{\LE} \cap|\Im \omega|^{-\frac12}L^2 }.
\]
Combining the two bounds we obtain
\[
\| u\|_{\LE^1}+  |\omega| \|u\|_{{\LE}} + |\Im \omega|^{\frac12} (  \| u\|_{\dot H^1} + |\omega|\|u\|_{L^2})
\lesssim \|f\|_{\LE^*+ |\Im \omega|^{\frac12}L^2 } 
\]
provided that $|\Im \omega|$ is sufficiently small.  Here we have
  used a Hardy inequality to control $\|u\|_{L^2_{comp}}$ by
  $\|u\|_{\LE^1}$.  
This yields the desired resolvent energy bound.

b) The form $\la \cdot,\cdot\ra_E$ is semipositive on $(S^+)^\perp$, with degeneracy 
only on $S^+ \subset  (S^+)^\perp$, and similarly on $(S^-)^\perp$. Thus the constant $c$
can degenerate only if the subspaces $S^+$ and $S^-$ are very close. However, this 
is prohibited by the bound in part (a) on the projectors $P^+$ and $P^-$.

c) This is a straightforward consequence of the similar property of the resolvent
bound along the curve $\gamma^+$, which in turn is a consequence of the 
enhanced resolvent bound 
\begin{equation}\label{res-le+}
|\omega|^{-1} \| \nabla^2 u\|_{L^2}+ \| \nabla u\|_{L^2} + |\omega| \| u\|_{L^2}
 \lesssim |\Im \omega|^{-1} \|P_\omega u\|_{L^2}
\end{equation}
which is an elliptic consequence of the standard one.
\end{proof}
 
 
\subsection{Perturbation theory of $S^\pm$}
 
As a preliminary step toward the proof of Theorem \ref{exp_trich_thm},
in this subsection we prove a local version of Theorem
\ref{exp_trich_thm} for $\epsilon$ almost stationary, almost symmetric
operators $P(t,x,D)$. 

The spectral decomposition in the previous section is 
associated to symmetric operators. In order to use it here, we denote by 
$P_{sym}(t,x,D)$ the symmetrized operators $P(t,x,D)$ (i.e. where $A$ and $V$
are replaced by their real parts). Correspondingly, we define the time dependent 
subspaces $S^{\pm,0}(t)$ and projectors $P^{\pm,0}(t)$.

\begin{prop}[Local stable/unstable/center manifold construction]\label{loc_sus_prop}
  Let $P$ satisfy all of the assumptions of Theorem
  \ref{exp_dich_thm}.  Then there exists a time threshold $T_* =
  T_*(\epsilon)$, such that $T_*\to \infty$ as $\epsilon\to 0$, so
  that the estimates \eqref{ed1}--\eqref{ed3} hold for the evolution
  $\mathcal{F}(t,t_0)$ of $P(t,x,D)$ for all $t,s\in J_{t_0} =
  [t_0,t_0+T_*]$. Moreover, one has:
\begin{equation}
		\lp{ P^{\pm,0}(t)\mathcal{F}(t,t_0) -\mathcal{F}(t,t_0)P^{\pm,0}(t_0)
		}{\dot{H}^1\times L^2\to \dot{H}^1\times L^2} \ \lesssim  \epsilon e^{C|t-t_0|} \ ,
\qquad t \in J_{t_0}.
 \label{S_match}
\end{equation}
\end{prop}
We remark that the time threshold $T^*$ is chosen so that $\epsilon = e^{-CT^*}$ for a large
constant $C$.
\begin{proof}
We begin with a simple observation, namely that 
\[
\| \mathcal H - \mathcal H_{sym}\|_{\dot{H}^1\times L^2\to \dot{H}^1\times L^2} = o(\epsilon).
\]
Thus, by Gronwall,  one can harmlessly replace $P$ by $P_{sym}$ and 
$\mathcal H$  by $ \mathcal H_{sym}$ in the proposition.

Secondly, consider initial data $U(t_0) \in S^{\pm{,0}}(t_0)$. Then we claim that $e^{i{(}t{-t_0)} \mathcal H(t_0)}$ 
is a good $\epsilon$-approximate solution. This is easy to see, since 
\[
\| {\mathcal H}(t) - \mathcal H(t_0)\|_{\dot H^2 \cap \dot{H}^1\times \dot H^1 \cap  L^2\to \dot{H}^1\times L^2} \lesssim 
\epsilon(1+|t-t_0|),
\]
 while 
\[
\| P^{\pm{,0}} \|_{\dot{H}^1\times L^2 \to \dot H^2 \cap \dot{H}^1\times \dot H^1 \cap L^2} \lesssim 1.
\]
Thus we obtain 
\[
\| (D_t - \mathcal H) {e^{i(t-t_0)\mathcal
    H(t_0)}}U(t_0)\|_{{\dot{H}^1\times L^2}} {\lesssim}
{\epsilon e^{C|t-t_0|} \|e^{i\mathcal H (t-t_0)}
  U(t_0)\|_{\dot{H^1}\times L^2} ,}
\]
which by the Duhamel formula {and an estimate akin to \eqref{local-en}} shows that
\begin{equation}\label{S+-app}
\| (\mathcal{F}(t,t_0) - e^{i{(}t{-t_0)} \mathcal H(t_0)}) U(t_0)\|_{\dot{H}^1\times L^2} \lesssim 
\epsilon e^{C|t-t_0|} {\|e^{i\mathcal H (t-t_0)}
  U(t_0)\|_{\dot{H^1}\times L^2}}.
\end{equation}
This already proves the bounds \eqref{ed1} and \eqref{ed2} for $s = t_0$.

We now establish \eqref{S_match}. For this we write
\begin{multline*}
P^{\pm{,0}}(t)\mathcal{F}(t,t_0) P^{\pm{,0}}(t_0)-\mathcal{F}(t,t_0)P^{\pm{,0}}(t_0) \\= 
(P^{\pm{,0}}(t) - P^{\pm{,0}}(t_0)) \mathcal{F}(t,t_0) P^{\pm{,0}}(t_0)
+ {(}P^{\pm{,0}}(t_0){-I)} (\mathcal{F}(t,t_0) -  e^{it \mathcal H(t_0)})) P^{\pm{,0}}(t_0).
\end{multline*}
Using part (c) of Proposition~{\ref{p:ppm}} for the first term on the right
and \eqref{S+-app} for the second we obtain
\[
\| P^{\pm{,0}}(t)\mathcal{F}(t,t_0) P^{\pm{,0}}(t_0)-\mathcal{F}(t,t_0)P^{\pm{,0}}(t_0)\|_{
\dot{H}^1\times L^2\to \dot{H}^1\times L^2} 
\lesssim 
\epsilon e^{C|t-t_0|}.
\]
Then \eqref{S_match} follows by time reversal. The bounds
 \eqref{ed1} and \eqref{ed2} for $s \neq t_0$ are a direct consequence 
of \eqref{S_match} and the $s = t_0$ case.

It remains to prove \eqref{ed3}. Due to \eqref{S_match} it again
suffices to consider the case $s = t_0$.  For this we consider the
energy conservation relation, which in the time dependent case reads
\[
\frac{d}{dt} E(U(t)) \lesssim \epsilon \|U(t)\|_{\dot{H}^1\times L^2}^2. 
\]
This implies that 
\[
| E(U(t)) - E(U(t_0))| \lesssim \epsilon e^{C|t-t_0|} \| U(t_0)\|_{\dot{H}^1\times L^2}^2.
\]
We apply this to $U(t) =  \mathcal{F}(t,t_0) P^0(t_0) U_0 $,
which by \eqref{S_match} is well approximated by $P^0(t) \mathcal{F}(t,t_0)U_0 $.
By part (b) of Proposition~{\ref{p:ppm}}, the (time dependent) energy functional
$E(t)$ is uniformly positive definite on the $S^0(t)$, therefore we can write
\[\begin{split}
\|  U(t) \|_{\dot{H}^1\times L^2}^2 = & \ \| P^0(t) \mathcal{F}(t,t_0)U_0 \|_{\dot{H}^1\times L^2}^2 
+ O(\epsilon e^{C(t-t_0)}) \|  U_0 \|_{\dot{H}^1\times L^2}^2
\\ \approx & \  E(P^0(t) \mathcal{F}(t,t_0)U_0)  + O(\epsilon e^{C(t-t_0)}) \|  U_0 \|_{\dot{H}^1\times L^2}^2
\\ = & \  E(U(t))  + O(\epsilon e^{C(t-t_0)}) \|  U_0 \|_{\dot{H}^1\times L^2}^2
\\ = & \  E(U(t_0))  + O(\epsilon e^{C(t-t_0)}) \|  U_0 \|_{\dot{H}^1\times L^2}^2
\\  \approx & \ \|   P^0(t_0) U_0 \| _{\dot{H}^1\times L^2}^2+ O(\epsilon e^{C(t-t_0)}) \|  U_0 \|_{\dot{H}^1\times L^2}^2.
\end{split}
\]
Hence \eqref{ed3} follows, and the proof of the proposition is complete.

\end{proof}

 
\subsection{Global construction of $P^{\pm,0}$ via Perron sums}

We now prove the full version of Theorem \ref{exp_trich_thm}. In light
of the energy estimates in Proposition \ref{loc_sus_prop} and Theorem
\ref{exp_dich_thm} it suffices to prove the following slightly weaker
discrete version:

\begin{thm}[Discretized  version of Theorem \ref{exp_trich_thm}]\label{exp_trich_thm_disc}
Let $P$ be as in the statement of Theorem \ref{exp_trich_thm}. Then there exists
a time increment $T_*>0$ such that $T_*\to\infty$ as $\epsilon\to 0$
with the following property. There exists 
$\kappa\in \mathbb{N}$ and real $C,\alpha,\alpha_0>0$,
where $\alpha_0=o_1(\epsilon)$, and   two $\kappa$ dimensional 
linear subspaces  $S^\pm\subseteq \dot{H}^1\times L^2$ with corresponding finite rank 
commuting projections $P^\pm$ and $P^0= I - P^+-P^-$, such that if $t_n=nT_*$ one has:
\begin{align}
		\lp{\mathcal{F}(t_n)P^- \mathcal{F}^{-1}(t_m)}{\dot{H}^1\times L^2\to \dot{H}^1\times L^2}
		\ &\leq\  Ce^{-\alpha(t_n-t_m)} \ , &t_n\geq t_m\geq 0 \ , \label{ed1_disc}\\
		\lp{\mathcal{F}(t_m)P^+ \mathcal{F}^{-1}(t_n)}{\dot{H}^1\times L^2\to \dot{H}^1\times L^2}
		\ &\leq\  Ce^{-\alpha(t_n-t_m)} \ , &t_n\geq t_m\geq 0 \ , \label{ed2_disc}\\
		\lp{\mathcal{F}(t_n)(1-P^+-P^-) \mathcal{F}^{-1}(t_m)}{\dot{H}^1\times L^2\to \dot{H}^1\times L^2}
		\ &\leq\  Ce^{\alpha_0|t_n-t_m|} \ , &\hbox{all\ \ } t_n,t_m \geq 0 \ , \label{ed3_disc}
\end{align}
where $0\leq \alpha_0\ll 1$ is sufficiently small so that estimate
\eqref{slow_growth} holds in case \eqref{ed3_disc} above. Further, the
push-forward of the projectors $P^{\pm,0}$ are close to the the time dependent projectors
$P^{\pm,0}(t)$, namely 
\begin{equation}
\| P^{\pm{,0}}(t_n) -  \mathcal{F}(t_n)P^{{\pm,0}} \mathcal{F}^{-1}(t_n)\|_{\dot{H}^1\times L^2\to \dot{H}^1\times L^2} \lesssim \epsilon.
\end{equation}
\end{thm}

Here the time $T^*$ is is as in Proposition \ref{loc_sus_prop}, which
also takes care of the estimates within each time interval
$[t_n,t_{n+1}]$. 

\begin{rem}
  The above decomposition of $\dot{H}^1\times L^2 = S^+ + S^- + S^0$
  is not uniquely defined. Precisely, it is the subspaces $S^-$ and
  $S^-+S^0$ which are uniquely determined, but $S^+$ and $S^0$ are
  not. In our construction below we select these subspaces uniquely by
  adding the supplimentary requirement that $S^+= S^+(0)$ and $S^+ +
  S^0 = S^+(0) + S^0(0)$.  We further remark that this ambiguity is
  due to the fact that our set-up is on the positive time line. One
  also has a similar result that corresponds to the full time line,
  and there all subspaces $S^+$, $ S^-$ and $S^0$ are uniquely
  determined.
\end{rem}

The proof of the above Theorem will follow from Proposition
\ref{loc_sus_prop} and the following abstract Perron type construction
for discrete time flows on Banach spaces:

\begin{prop}\label{Perron_prop}
Let $X$ be a Banach space, and let $\mathcal{G}_n\in\mathcal{B}(X)$ 
be a sequence of invertible operators, with $\mathcal{G}_0=I$,  which induces an exponential trichotomy:
\begin{align}
		\lp{\mathcal{G}_n Q^- \mathcal{G}_m^{-1}}{\mathcal{B}(X)} \ &\leq \ 
		Ce^{-\gamma(n-m)} \ , &n&\geq m\geq 0 \ , \label{G_dich1}\\
		\lp{ \mathcal{G}_m Q^+ \mathcal{G}_n^{-1}}{\mathcal{B}(X)} \ &\leq \ 
		Ce^{-\gamma(n-m)} \ , &n&\geq m\geq 0 \ , \label{G_dich2}\\
		\lp{ \mathcal{G}_n (1-Q^- -Q^+) \mathcal{G}_m^{-1}}{\mathcal{B}(X)} \ &\leq \ 
		Ce^{\gamma_0 |n-m|}  \ , &n&,m \geq 0 \ , \label{G_dich3}
\end{align}
for a pair of projection matrices $Q^\pm$ with $Q^+Q^-=Q^-Q^+=0$,
where we also assume $1 \leq \gamma_0 \ll \gamma$.  Let
$\mathcal{F}_n$ be another sequence of bounded operators, with
$\mathcal{F}_0=I$, and such that:
\begin{equation}
		\lp{\mathcal{F}_n\mathcal{F}_{n-1}^{-1}-
		\mathcal{G}_n \mathcal{G}_{n-1}^{-1}}{\mathcal{B}(X)}\
              \leq \ \epsilon, \label{1step_cond}
\end{equation}
for $\epsilon>0$ sufficiently small. Then there exist commuting projections $P^\pm$ with 
$\lp{P^\pm-Q^\pm}{\mathcal{B}(X)}\lesssim \epsilon$ such that:
\begin{align}
		\lp{\mathcal{F}_n P^- \mathcal{F}_m^{-1}}{\mathcal{B}(X)} \ &\leq \ 
		C'e^{-\gamma'(n-m)} \ , &n&\geq m\geq 0 \ , \\
		\lp{ \mathcal{F}_m P^+ \mathcal{F}_n^{-1}}{\mathcal{B}(X)} \ &\leq \ 
		C'e^{-\gamma'(n-m)} \ , &n&\geq m\geq 0 \ , \\
		\lp{ \mathcal{F}_n (1-P^- -P^+) \mathcal{F}_m^{-1}}{\mathcal{B}(X)} \ &\leq \ 
		C'e^{\gamma'_0|n-m|}  \ , &n&,m \geq 0 \ ,
\end{align}
where $0\leq \gamma_0' , \gamma'$ and 
$|\gamma_0-\gamma_0'|+|\gamma-\gamma'|\lesssim \epsilon$.

\end{prop}

To apply Proposition \ref{Perron_prop} to the proof of 
Theorem \ref{exp_trich_thm_disc} we
let $\mathcal{F}_n=\mathcal{F}(t_n)$ denote the time $t_n$ 
flow of $P(t,x,D)$ and let $X=\dot{H}^1\times L^2$. For each time step
we ``correct'' $\mathcal{F}_n$ so it maps $S^\pm(t_n)$ perfectly into $S^{\pm}(t_{n+1})$.
That is we define $\mathcal{G}_n$ inductively via:
\begin{equation}
		\mathcal{G}_{n+1}\mathcal{G}_{n}^{-1} \ = \ \sum_{a=\pm,0} P^a(t_{n+1}) \mathcal{F}_{n+1}
		\mathcal{F}_{n}^{-1} P^a(t_n) \ , \notag
\end{equation}
and set $Q^\pm = P^\pm(0)$. Using estimate \eqref{S_match} we have \eqref{1step_cond}.
By induction and the fact that 
$\mathcal{F}_{n+1}\mathcal{F}_{n}^{-1}$ induces an exponential trichotomy of the form 
\eqref{ed1'}--\eqref{ed3'} (again thanks to Proposition \ref{loc_sus_prop}) we have the bounds
for all $\xi\in X$:
\begin{align}
		\lp{\mathcal{G}_n Q^- \xi }{X} \ &\leq \ 
		Ce^{-\gamma(n-m)}\lp{\mathcal{G}_m Q^- \xi }{X} 
		 \ , &n&\geq m\geq 0 \ , \label{G_dich1'}\\
		\lp{ \mathcal{G}_m Q^+ \xi }{X} \ &\leq \ 
		Ce^{-\gamma(n-m)} \lp{ \mathcal{G}_n Q^+ \xi }{X}
		\ , &n&\geq m\geq 0 \ ,  \label{G_dich2'}\\\
		\lp{ \mathcal{G}_n (1-Q^- -Q^+) \xi }{X} \ &\leq \ 
		Ce^{\gamma_0 |n-m|}  \lp{ \mathcal{G}_m (1-Q^- -Q^+) \xi }{X} 
		\ , &n&,m \geq 0 \ .  \label{G_dich3'}\
\end{align}
where $\gamma\approx T_*\alpha$ and $\gamma_0=O(1)$ uniform in $T_*$. In addition
by an energy estimate such as \eqref{local-en}, we have the uniform growth rate 
$\lp{\mathcal{G}_n\mathcal{G}_m^{-1}}{\mathcal{B}({X})}\leq Ce^{C|n-m|}$. A simple argument
then shows that the final estimate:
\begin{equation}
		\lp{\mathcal{G}_n {Q}^\pm \mathcal{G}_n^{-1}}{\mathcal{B}(X)} \ \lesssim \ 1 \ , \notag
\end{equation}
follows automatically from \eqref{G_dich1'}--\eqref{G_dich3'} (see Chapter 2 of \cite{Coppel}). Thus we have 
\eqref{G_dich1}--\eqref{G_dich3} which suffices for our application.
It remains to prove Proposition \ref{Perron_prop}.
 
\begin{proof}[Proof of Proposition \ref{Perron_prop}]
By instead considering $\td{\mathcal{G}}_n=e^{\frac{1}{2}(\gamma +\gamma_0) n}\mathcal{G}_n$ and 
$\td{\mathcal{F}}_n=e^{\frac{1}{2}(\gamma +\gamma_0) n}\mathcal{F}_n$ 
to construct $P^-$, and similarly 
$\td{\td{\mathcal{G}}}_n=e^{-\frac{1}{2}(\gamma +\gamma_0) n}\mathcal{G}_n$ and 
$\td{\td{\mathcal{F}}}_n=e^{-\frac{1}{2}(\gamma +\gamma_0) n}\mathcal{F}_n$ 
to construct $P_+$,
we can reduce to the more standard situation
where there exists a single $\gamma>0$ and a single projection $Q$ with:
\begin{align}
		\lp{\mathcal{G}_n Q  \mathcal{G}_m^{-1}}{\mathcal{B}(X)} \ &\leq \ 
		Ce^{-\gamma(n-m)} \ , &n&\geq m\geq 0 \ , \label{G_red1}\\
		\lp{ \mathcal{G}_m (1-Q) \mathcal{G}_n^{-1}}{\mathcal{B}(X)} \ &\leq \ 
		Ce^{-\gamma(n-m)} \ , &n&\geq m\geq 0 \ , \label{G_red2}
\end{align}
in which case we are trying to construct a single projection $P$ 
for the $\mathcal{F}_n$ evolution with the bounds:
\begin{align}
		\lp{\mathcal{F}_n P  \mathcal{F}_m^{-1}}{\mathcal{B}(X)} \ &\leq \ 
		C'e^{- \gamma'(n-m)} \ , &n&\geq m\geq 0 \ , \label{F_red1}\\
		\lp{ \mathcal{F}_m (1-P) \mathcal{F}_n^{-1}}{\mathcal{B}(X)} \ &\leq \ 
		C'e^{- \gamma'(n-m)} \ , &n&\geq m\geq 0 \ , \label{F_red2}
\end{align}
where $\gamma'=\gamma-C\epsilon$. To both uniquely determine $P$
and insure that the two constructions are compatible, we will further require that 
$Range(1-P) = Range(1-Q)$, or equivalently that 
\begin{equation}\label{pq-comp}
P(1-Q) = Q(1-P) = 0.
\end{equation} 
The argument is now essentially that of Chapter 4 in \cite{Coppel}. 
We postpone that for a moment in order to show how the two steps mix
together.

The first step produces the projector $P^-$, with the additional property that 
\[
Q^-(1-P^-) = P^- (Q^++Q^0) = 0.
\]
 The second step produces the projector 
$P^+$, with the additional property that 
\[
(1-P^+)Q^+ = (Q^0+Q^-) P^+ = 0.
\]
These two relations imply that
\[
P^- P^+ = P^-(Q^++ Q^0+Q^-)  P^+ = 0
\]
On the other hand, matching the decay properties of solutions we see
that $Range(P^-) \subset Range(1-P^+)$, i.e. $P^+ P^- = 0$. This shows
that the projectors ${P}^+$ and ${P}^-$ are commuting.

We now return to the construction of $P$ as in \eqref{F_red1},
\eqref{F_red2} given $Q$ as in \eqref{G_red1}, \eqref{G_red2}.  First,
we write the evolution law for $\mathcal{G}_n$ as a discrete
non-autonomous ODE.  Setting $\Delta
\mathcal{G}_n=\mathcal{G}_{n+1}-\mathcal{G}_n$ we have:
\begin{equation}
		\Delta \mathcal{G}_n \ = \ \mathcal{A}(n)\mathcal{G}_n \ , \qquad
		\hbox{where\ \ } \mathcal{A}(n) \ = \ (\mathcal{G}_{n+1}\mathcal{G}_n^{-1}-I) \ . \label{G_homog}
\end{equation}
On the other hand one can consider the inhomogeneous equation:
\begin{equation}
		\Delta \mathcal{F}_n \ = \ \mathcal{A}(n)\mathcal{F}_n+\mathcal{K}_n \ , \label{G_inhomog}
\end{equation}
which can be 
solved via Duhamel's principle with initial data $M\in\mathcal{B}(X)$ via:
\begin{equation}
		\mathcal{F}_n \ = \ \mathcal{G}_n M + \sum_{j=0}^n \mathcal{G}_n\mathcal{G}_{j}^{-1}
		\mathcal{K}_{j-1} \ , \qquad \hbox{where \ \ }\mathcal{K}_{-1}=0 \ . \notag
\end{equation}
Choosing $M=M_0-\sum_{j=0}^\infty (I-Q)\mathcal{G}_j^{-1}\mathcal{K}_{j-1}$ and assuming the sum
converges we can also write the previous line as:
\begin{equation}
		\mathcal{F}_n \ = \ \mathcal{G}_n M_0 + \sum_{j=0}^n \mathcal{G}_nQ\mathcal{G}_{j}^{-1}
		\mathcal{K}_{j-1} 
		- \sum_{j=n+1}^\infty \mathcal{G}_n(I-Q)\mathcal{G}_{j}^{-1}
		\mathcal{K}_{j-1}
		\ , \qquad \hbox{where \ \ }\mathcal{K}_{-1}=0 \ . \notag
\end{equation}
Now the equation for $\mathcal{F}_n$ is of the form \eqref{G_inhomog}
where $\mathcal{K}_n=\mathcal{B}(n)\mathcal{F}_n$ with
$\mathcal{B}(n)=\mathcal{F}_{{n+1}}\mathcal{F}_{{n}}^{-1}- \mathcal{G}_{{n+1}}
\mathcal{G}_{{n}}^{-1}$ producing a small perturbation of the dynamics
governed by $\mathcal{A}{(n)}$. To construct the projection $P$
associated to the stable evolution for $\mathcal{F}_n$ we solve the fixed
point equation $\mathcal{J}_n = I(\mathcal{J}_n)$ where:
\begin{equation}
		I(\mathcal{J}_n) \ = \ \mathcal{G}_nQ + \sum_{j=0}^n \mathcal{G}_nQ\mathcal{G}_{j}^{-1}
		\mathcal{B}(j-1) \mathcal{J}_{j-1}
		- \sum_{j=n+1}^\infty \mathcal{G}_n(I-Q)\mathcal{G}_{j}^{-1}
		\mathcal{B}(j-1) \mathcal{J}_{j-1}
		\ , \qquad \hbox{where \ \ }\mathcal{B}_{-1}=0 \ . \notag
\end{equation}
By the smallness of $\mathcal{B}(n)$ and the estimates
\eqref{G_red1}--\eqref{G_red2}, the map $\mathcal{J}_n\mapsto
I(\mathcal{J}_n)$ is a contraction on the space
$\llp{\mathcal{J}_n}{}=\sup_{n\geq 0}\lp{e^{ \gamma'
    n}\mathcal{J}_n}{\mathcal{B}(X)}$ and therefore has a unique fixed
point which we also denote by $\mathcal{J}_n$. Since this solves
the homogeneous evolution equation for $\mathcal{F}_n$ we have
$\mathcal{J}_n=\mathcal{F}_n P$ where $P=\mathcal{J}_0$. It remains to
show $P$ is a projection which satisfies the bounds \eqref{F_red1} and
\eqref{F_red2} as well as the relation \eqref{pq-comp}. 

For the latter we observe that applying the fixed point formula for $n=0$
yields $P = Q +(1-Q) R$ where $\|R\|_{\mathcal B(X)} \lesssim \epsilon$.
Thus $1-P = (1-Q)(1-R)$, and \eqref{pq-comp} easily follows.

Multiplying the fixed point equation $\mathcal{F}_mP=I(\mathcal{F}_mP)$ by 
$\mathcal{G}_nQ\mathcal{G}_m^{-1}$ we get the auxiliary formula:
\begin{equation}
		\mathcal{G}_nQ\mathcal{G}_m^{-1}\mathcal{F}_mP \ = \ 
		\mathcal{G}_nQ+ \sum_{j=0}^m \mathcal{G}_mQ\mathcal{G}_{j}^{-1}
		\mathcal{B}(j-1) {\mathcal{F}_{j-1}P} \ . \label{aux_FP_iden}
\end{equation} 
Setting $n=m=0$ in this last line yields $QP=Q$. This implies that the fixed point formula
also yields the identity $ \mathcal{F}_mP^2=I(\mathcal{F}_mP^2)$, so by uniqueness
of the fixed point and evaluation at $n=0$ we must have $P^2=P$. 

Next, we can subtract \eqref{aux_FP_iden} from the fixed point equation 
$\mathcal{F}_nP=I(\mathcal{F}_nP)$ for $n> m$ which gives:
\begin{equation}
		\mathcal{F}_nP \ = \ \mathcal{G}_nQ\mathcal{G}_m^{-1}\mathcal{F}_mP
		 + \sum_{j=m+1}^n \mathcal{G}_nQ\mathcal{G}_{j}^{-1}
		\mathcal{B}(j-1)\mathcal{F}_{j-1}P 
		 - \sum_{j=n+1}^\infty \mathcal{G}_n(I-Q)\mathcal{G}_{j}^{-1}
		\mathcal{B}(j-1) \mathcal{F}_{j-1}P
		\ . \notag
\end{equation} 
Then applying {\eqref{1step_cond},} \eqref{G_red1}--\eqref{G_red2}
and some simple estimates {(see \cite[Chapter 4]{Coppel})}, we get
for all $\xi\in X$:
\begin{equation}
		\lp{\mathcal{F}_n P  \xi}{X} \ \lesssim \ 
		e^{- \gamma'(n-m)}\lp{\mathcal{F}_m P  \xi}{X}  
		\ , \qquad n\geq m\geq 0 \ . \notag
\end{equation}

Finally, by performing  algebraic similar to those above we  arrive at the identity
 (again for $n>m$):
 \begin{equation}
\begin{split}
		\mathcal{F}_m(I-P)  = & \   \mathcal{G}_m(I-Q)\mathcal{G}_n^{-1}\mathcal{F}_n(I-P)
		 + \sum_{j=0}^m \mathcal{G}_mQ\mathcal{G}_{j}^{-1}
		\mathcal{B}(j-1)\mathcal{F}_{j-1}(I-P) 
		\\ & \ - \sum_{j=m+1}^n \mathcal{G}_m(I-Q)\mathcal{G}_{j}^{-1}
		\mathcal{B}(j-1) \mathcal{F}_{j-1}(I-P)
		, \notag
\end{split}
\end{equation} 
which suffices to estimate:
\begin{equation}
		\lp{\mathcal{F}_m (I-P)  \xi}{X} \ \lesssim \ 
		e^{-\gamma'(n-m)}\lp{\mathcal{F}_n (I-P)  \xi}{X}  
		\ , \qquad n\geq m\geq 0 \ . \notag
\end{equation}
\end{proof}


\bibliography{scalar}

\begin{thebibliography}{10}

\bibitem{Alinhac}
Serge Alinhac.
\newblock On the {M}orawetz--{K}eel-{S}mith-{S}ogge inequality for the wave
  equation on a curved background.
\newblock {\em Publ. Res. Inst. Math. Sci.}, 42(3):705--720, 2006.

\bibitem{B_indef}
J{\'a}nos Bogn{\'a}r.
\newblock {\em Indefinite inner product spaces}.
\newblock Springer-Verlag, New York-Heidelberg, 1974.
\newblock Ergebnisse der Mathematik und ihrer Grenzgebiete, Band 78.

\bibitem{BT}
Jean-Marc Bouclet and Nikolay Tzvetkov.
\newblock On global {S}trichartz estimates for non-trapping metrics.
\newblock {\em J. Funct. Anal.}, 254(6):1661--1682, 2008.

\bibitem{bpstz}
Nicolas Burq, Fabrice Planchon, John~G. Stalker, and A.~Shadi Tahvildar-Zadeh.
\newblock Strichartz estimates for the wave and {S}chr\"odinger equations with
  potentials of critical decay.
\newblock {\em Indiana Univ. Math. J.}, 53(6):1665--1680, 2004.

\bibitem{Coppel}
W.~A. Coppel.
\newblock {\em Dichotomies in stability theory}.
\newblock Lecture Notes in Mathematics, Vol. 629. Springer-Verlag, Berlin-New
  York, 1978.

\bibitem{Doi}
Shin-ichi Doi.
\newblock Remarks on the {C}auchy problem for {S}chr\"odinger-type equations.
\newblock {\em Comm. Partial Differential Equations}, 21(1-2):163--178, 1996.

\bibitem{ggh}
Vladimir Georgescu, Christian G\'erard, and Dietrich H\"afner.
\newblock Resolvent and propagation estimates for {K}lein-{G}ordon equations
  with non-positive energy.
\newblock {\em J. Spectr. Theory}, 5(1):113--192, 2015.

\bibitem{g}
Christian G\'erard.
\newblock Scattering theory for {K}lein-{G}ordon equations with non-positive
  energy.
\newblock {\em Ann. Henri Poincar\'e}, 13(4):883--941, 2012.

\bibitem{hy}
Kunio Hidano and Kazuyoshi Yokoyama.
\newblock A remark on the almost global existence theorems of {K}eel, {S}mith
  and {S}ogge.
\newblock {\em Funkcial. Ekvac.}, 48(1):1--34, 2005.

\bibitem{KSS}
Markus Keel, Hart~F. Smith, and Christopher~D. Sogge.
\newblock Almost global existence for some semilinear wave equations.
\newblock {\em J. Anal. Math.}, 87:265--279, 2002.
\newblock Dedicated to the memory of Thomas H. Wolff.

\bibitem{kpv}
Carlos~E. Kenig, Gustavo Ponce, and Luis Vega.
\newblock On the {Z}akharov and {Z}akharov-{S}chulman systems.
\newblock {\em J. Funct. Anal.}, 127(1):204--234, 1995.

\bibitem{KT-ucp}
Herbert Koch and Daniel Tataru.
\newblock Carleman estimates and unique continuation for second-order elliptic
  equations with nonsmooth coefficients.
\newblock {\em Comm. Pure Appl. Math.}, 54(3):339--360, 2001.

\bibitem{KT}
Herbert Koch and Daniel Tataru.
\newblock Carleman estimates and absence of embedded eigenvalues.
\newblock {\em Comm. Math. Phys.}, 267(2):419--449, 2006.

\bibitem{L_gen}
Heinz Langer.
\newblock Spectral functions of definitizable operators in {K}re\u\i n spaces.
\newblock In {\em Functional analysis ({D}ubrovnik, 1981)}, volume 948 of {\em
  Lecture Notes in Math.}, pages 1--46. Springer, Berlin-New York, 1982.

\bibitem{L_KG}
Heinz Langer, Branko Najman, and Christiane Tretter.
\newblock Spectral theory of the {K}lein-{G}ordon equation in {P}ontryagin
  spaces.
\newblock {\em Comm. Math. Phys.}, 267(1):159--180, 2006.

\bibitem{MMT}
Jeremy Marzuola, Jason Metcalfe, and Daniel Tataru.
\newblock Strichartz estimates and local smoothing estimates for asymptotically
  flat {S}chr\"odinger equations.
\newblock {\em J. Funct. Anal.}, 255(6):1497--1553, 2008.

\bibitem{MetSo}
Jason Metcalfe and Christopher~D. Sogge.
\newblock Long-time existence of quasilinear wave equations exterior to
  star-shaped obstacles via energy methods.
\newblock {\em SIAM J. Math. Anal.}, 38(1):188--209, 2006.

\bibitem{MetSo2}
Jason Metcalfe and Christopher~D. Sogge.
\newblock Global existence of null-form wave equations in exterior domains.
\newblock {\em Math. Z.}, 256(3):521--549, 2007.

\bibitem{MT}
Jason Metcalfe and Daniel Tataru.
\newblock Global parametrices and dispersive estimates for variable coefficient
  wave equations.
\newblock {\em Math. Ann.}, 353(4):1183--1237, 2012.

\bibitem{MTT}
Jason Metcalfe, Daniel Tataru, and Mihai Tohaneanu.
\newblock Price's law on nonstationary space-times.
\newblock {\em Adv. Math.}, 230(3):995--1028, 2012.

\bibitem{morawetz1}
Cathleen~S. Morawetz.
\newblock Exponential decay of solutions of the wave equation.
\newblock {\em Comm. Pure Appl. Math.}, 19:439--444, 1966.

\bibitem{morawetz2}
Cathleen~S. Morawetz.
\newblock Time decay for the nonlinear {K}lein-{G}ordon equations.
\newblock {\em Proc. Roy. Soc. Ser. A}, 306:291--296, 1968.

\bibitem{morawetz3}
Cathleen~S. Morawetz.
\newblock Decay for solutions of the exterior problem for the wave equation.
\newblock {\em Comm. Pure Appl. Math.}, 28:229--264, 1975.

\bibitem{mrs}
Cathleen~S. Morawetz, James~V. Ralston, and Walter~A. Strauss.
\newblock Decay of solutions of the wave equation outside nontrapping
  obstacles.
\newblock {\em Comm. Pure Appl. Math.}, 30(4):447--508, 1977.

\bibitem{smithsogge}
Hart~F. Smith and Christopher~D. Sogge.
\newblock Global {S}trichartz estimates for nontrapping perturbations of the
  {L}aplacian.
\newblock {\em Comm. Partial Differential Equations}, 25(11-12):2171--2183,
  2000.

\bibitem{ST}
Gigliola Staffilani and Daniel Tataru.
\newblock Strichartz estimates for a {S}chr\"odinger operator with nonsmooth
  coefficients.
\newblock {\em Comm. Partial Differential Equations}, 27(7-8):1337--1372, 2002.

\bibitem{sterb}
Jacob Sterbenz.
\newblock Angular regularity and {S}trichartz estimates for the wave equation.
\newblock {\em Int. Math. Res. Not.}, (4):187--231, 2005.
\newblock With an appendix by Igor Rodnianski.

\bibitem{Strauss}
Walter~A. Strauss.
\newblock Dispersal of waves vanishing on the boundary of an exterior domain.
\newblock {\em Comm. Pure Appl. Math.}, 28:265--278, 1975.

\bibitem{T}
Daniel Tataru.
\newblock Local decay of waves on asymptotically flat stationary space-times.
\newblock {\em Amer. J. Math.}, 135(2):361--401, 2013.

\end{thebibliography}

\end{document}